\documentclass[12pt]{amsart}

\usepackage[margin=1in]{geometry}

\usepackage{amssymb,amsmath,amsthm,mathrsfs,enumerate,multicol,esint}
\usepackage{color}
\usepackage{makecell}
\usepackage{tikz}
\usetikzlibrary{arrows,calc}

\allowdisplaybreaks

\title[Hermite pseudo-multipliers and molecules]{Pseudo-multipliers and smooth molecules on\\ Hermite Besov and Hermite Triebel--Lizorkin spaces}

\author{Fu Ken Ly}
\address{School of Mathematics and Statistics, The Faculty of Science, \& The Learning Hub (Mathematics), Education Portfolio, The University of Sydney,
NSW 2006, Australia.}
\email{ken.ly@sydney.edu.au}

\author{Virginia Naibo}
\address{Department of Mathematics, Kansas State University.
138 Cardwell Hall, 1228 N. 17th Street, Manhattan, KS  66506, USA.}
\email{vnaibo@ksu.edu}

\thanks{The second author was  partially supported by the NSF under grant DMS 1500381 and the Simons Foundation under grant 705953.}

\newcommand{\RR}{\mathbb{R}} 
\newcommand{\NN}{\mathbb{N}} 
\newcommand{\ZZ}{\mathbb{Z}} 
\newcommand{\sz}{\mathscr{S}} 
\newcommand{\MM}{\mathcal{M}}
\newcommand{\AM}{\mathcal{A}}
\newcommand{\N}{\mathcal{N}}
\newcommand{\K}{\mathcal{K}}

\newcommand{\f}{\frac}

\newcommand{\LL}{\mathcal{L}}
\newcommand{\PP}{\mathbb{P}} 
\newcommand{\QQ}{\mathbb{Q}} 
\newcommand{\ip}[1]{\langle #1 \rangle} 
\newcommand{\E}{\mathcal{E}} 
\newcommand{\X}{\mathcal{X}} 
\newcommand{\Ind}{\mathbf{1}} 
\newcommand{\Q}{\mathcal{Q}}
\newcommand{\B}{\mathcal{B}}
\newcommand{\wt}[1]{\widetilde{#1}} 
\newcommand{\diff}{\triangle}
\newcommand{\ve}{\varepsilon}
\newcommand{\vp}{\varphi}
\newcommand{\e}{\mathrm{e}}
\newcommand{\CN}{\mathcal{C}} 
\newcommand{\SC}{S} 
\newcommand{\SM}{\mathcal{S}} 
\newcommand{\floor}[1]{\lfloor #1 \rfloor}
\newcommand{\ceil}[1]{\lceil #1\rceil}
\newcommand{\A}[1]{A^{(#1)}}
\newcommand{\vph}{\varphi}
\newcommand{\cro}{\varrho} 
\newcommand{\dd}{\delta} 
\newcommand{\rr}{\rho} 

\newcommand{\kk}{k} 
\newcommand{\xs}{k} 
\newcommand{\vk}{\varkappa} 

\DeclareMathOperator{\supp}{supp\,} 

\theoremstyle{plain}
\newtheorem{Theorem}{Theorem}[section]
\newtheorem{Lemma}[Theorem]{Lemma}

\newtheorem{Corollary}[Theorem]{Corollary}

\newtheorem{Definition}[Theorem]{Definition}
\theoremstyle{definition}
\newtheorem{Remark}[Theorem]{Remark}
\theoremstyle{remark}
\newtheorem{Example}[Theorem]{Example}
\theoremstyle{example}

\numberwithin{equation}{section}

\def\barint{\kern4pt
\raise3.4pt\hbox{\vrule height.8pt width5pt}%
\kern-9pt 
\int}

\newcommand{\aver}[1]{-\hskip-0.46cm\int_{#1}}

\def\XXint#1#2#3{{\setbox0=\hbox{$#1{#2#3}{\int}$}
     \vcenter{\hbox{$#2#3$}}\kern-.5\wd0}}

\begin{document}


\begin{abstract}
We obtain new molecular decompositions and molecular synthesis estimates for Hermite Besov  and Hermite Triebel--Lizorkin spaces and  use such tools to prove boundedness properties of Hermite pseudo-multipliers on those spaces. The notion of molecule we develop leads to boundedness of pseudo-multipliers associated to symbols of H\"ormander-type adapted to the Hermite setting on spaces for which the  smoothness allowed  includes non-positive values; in particular, we obtain continuity results for such operators on Lebesgue and Hermite local Hardy spaces. As a byproduct of  our results on boundedness properties of pseudo-multipliers, we show that Hermite Besov spaces and Hermite Triebel--Lizorkin spaces are closed under non-linearities. 
\end{abstract}

\maketitle

\section{Introduction}\label{sec: intro}

This article contributes  new results to the theory of function spaces and corresponding boundedness properties of  pseudo-multipliers associated to the Hermite operator.  This operator is defined as $$ \LL=-\Delta+|x|^2,\quad x\in \RR^n,\,n\ge 1, $$ where $\Delta=\sum_{i=1}^n \f{\partial^2}{\partial x_i^2}$ is the Laplacian operator.  For a symbol $\sigma: \RR^n\times \NN_0\to\mathbb{C},$ the Hermite pseudo-multiplier $T_\sigma$ is given by
\begin{align*}
T_\sigma f(x)=\sum_{\kk \in \NN_0}\sigma(x,\lambda_{\kk}) \PP_\kk f(x),\quad x\in \RR^n,
\end{align*}
where $\lambda_\kk=2\kk+n$ and $\PP_\kk$ are orthogonal projectors onto spaces spanned by  Hermite functions (see Section \ref{sec: prelim}). When $\sigma$ is independent of $x,$  $T_\sigma$ is called a Hermite multiplier and can be expressed in the form
$$
T_\sigma(f)=\mathcal{F}^{-1}_\LL(\sigma \mathcal{F}_\LL(f)),
$$
where $\mathcal{F}_\LL$ is the Fourier--Hermite transform. These operators are counterparts of the well-known pseudo-differential operators and Fourier multiplier operators defined in the Euclidean setting in terms of the Fourier transform. 

In this paper, we obtain new molecular decompositions (Theorem~\ref{th:decomp}) and molecular synthesis estimates (Theorem~\ref{th:synth}) for the Hermite Besov  spaces, denoted $ B^{p,q}_{\alpha}(\LL),$ and the Hermite Triebel--Lizorkin spaces, denoted  $F^{p,q}_{\alpha}(\LL),$ and  use such tools to prove new results on boundedness of Hermite pseudo-multipliers on those spaces (Theorems~\ref{th: main1} and \ref{th: main2} and Corollary~\ref{coro:main}).

The works \cite{MR600625, MR1008442, MR1215939} pioneered the study of boundedness properties of Hermite multipliers on Lebesgue spaces. For instance, a H\"ormander-Mikhlin type multiplier theorem for Hermite expansions  was proved in  \cite{MR1008442}; more precisely,  if $\sigma$ satisfies 
$$ 
 |\diff^\kappa \sigma(\xs)| \lesssim  \xs^{-\kappa}\quad \forall \xs\in  \NN_0
$$
for $\kappa=0,\dots, \floor{\textstyle{\frac{n}{2}}}+1,$ then $T_\sigma$ is bounded on $L^p(\RR)$ for $1<p<\infty.$ The symbol $\diff$ denotes the forward difference operator, that is, for a function $f$ defined over the integers,  $\diff f(\kk)=f(\kk+1)-f(\kk)$ and $\diff^\kappa f(\kk)=\diff(\diff^{\kappa-1} f)(\kk)$ for $\kappa\ge 2$.

The first  results on boundedness properties of Hermite pseudo-multipliers  in Lebesgue spaces appeared in  \cite{MR1343690} and are  restricted to dimension one. This work was followed by investigations in \cite{MR3280055} where it is proved that if  $T_\sigma$ is bounded on $L^2(\RR^n)$ and 
\begin{align}\label{eq: zerotwozero}
|\partial_x^\nu\diff^\kappa_\xs\sigma(x,\xs)|\lesssim (1+\xs)^{-\kappa}\quad \forall \;(x,\xs)\in \RR^n\times \NN_0
\end{align}
for $\kappa=0,1,\dots,n+1$ and $\nu\in \NN_0^n$ such that $0\le|\nu|\le 1,$  then $T_\sigma$  is bounded on  $L^p(\RR^n)$ for $1<p<\infty.$ The theorems in  \cite{MR3280055,MR1343690} on boundedness of Hermite pseudo-multipliers on $L^p(\RR^n)$ for some range of $1<p<\infty$ assume the boundedness of the operator  on $L^2(\RR^n).$ For Hermite multipliers, boundedness on $L^2(\RR^n)$ is equivalent to the symbol being a bounded function on $\NN$ by Parseval's identity for Hermite expansions; however, the situation is more complicated in the case of Hermite pseudo-multipliers.  This prompted   the authors in \cite{MR3280055} to pose the question about what conditions on the symbol of a Hermite pseudo-multiplier  imply its boundedness on $L^2(\RR^n).$ Corollaries \ref{cor:Lp1} and \ref{cor:Lp2} of our results give an answer to this question; 
for instance, we obtain that $T_\sigma$ is bounded on $L^p(\RR^n)$ for $1<p<\infty$ if $\sigma$ satisfies \eqref{eq: zerotwozero} for $\kappa=0,\dots, n+1$ and $\nu\in \NN_0^n$ such that $0\le |\nu|\le 2\ceil{\frac{n+1}{2}}$ or if $\sigma$  satisfies a certain cancellation condition along with the weaker estimate 
\begin{align*}
|\partial_x^\nu\diff^\kappa_\xs\sigma(x,\xs)|\lesssim (1+\sqrt{\xs})^{-2\kappa+|\nu|}\quad \forall \;(x,\xs)\in \RR^n\times \NN_0
\end{align*}
for $\kappa=0,1,\dots,n+1$ and  $\nu\in \NN_0^n$ such that $0\le|\nu|\le 1.$ Other sufficient conditions for boundedness on $L^p(\RR^n)$ for $1<p<\infty$ of Hermite pseudo-multipliers, as well as a comparison of our  corresponding  results with those in the literature,  are presented in Section~\ref{sec:lebesgue}.  
  For the range $0<p\le 1,$ we  obtain results in the context of Hermite local Hardy spaces, which are better suited  than Lebesgue  or Hardy spaces for the study of boundedness properties of pseudo-multipliers (see Section~\ref{sec:hardy}).

Continuity properties of pseudo-differential operators have been extensively studied in various function spaces that include the scale of the classical Besov and Triebel--Lizorkin spaces associated to the Laplacian operator (see for instance, \cite{MR946282, MR1704146, MR639462, MR3759556, MR812311, MR1232192, MR1077277}). While there are some recent works on the study of boundedness properties of pseudo-differential operators in spaces associated to other non-negative self-adjoint operators (see \cite{Bui}  and \cite{MR3622656} addressed below), such investigations lead to results only in the context of spaces with positive smoothness. Our article seems to be the first one treating boundedness properties of Hermite pseudo-multipliers in the scales of the  Hermite Besov spaces $B^{p,q}_{\alpha}(\LL)$ and the Hermite Triebel--Lizorkin spaces $F^{p,q}_{\alpha}(\LL)$ that allow for non-positive values of the smoothness parameter $\alpha$  and for the whole range $(0,\infty)$ for the parameters $p,q$ (in particular, $\alpha=0$ and $q=2$ give the scale of Lebesgue and  Hermite local Hardy spaces). The spaces $B^{p,q}_{\alpha}(\LL)$ and $F^{p,q}_{\alpha}(\LL)$  are defined in terms of Littlewood-Paley operators associated to the Hermite operator.  A $\varphi$-transform characterization for them, in the spirit of the fundamental works \cite{MR808825, MR1070037} for the Euclidean setting,  was obtained in \cite{MR2399106}. Among the main contributions in this article, we introduce  new concepts of molecules for  Hermite Besov spaces and  Hermite Triebel--Lizorkin spaces; we then prove new corresponding almost orthogonality estimates, molecular decompositions and molecular synthesis estimates (see Section~\ref{sec: molecules}). The latter constitutes a crucial tool for the proofs of the boundedness properties of the Hermite pseudo-multipliers here studied, which require the use of both a $\varphi$-transform characterization as well as molecular synthesis estimates for the spaces. 

The classical molecules in \cite{MR808825, MR1070037} satisfy the following  cancellation condition
$$ \int_{\RR^n} y^\gamma m(y)\,dy=0 \qquad  \forall \gamma \in \NN_0\text{ such that } |\gamma|\le M,$$
for a suitable integer $M$. On the other hand, the (almost) cancellation condition imposed to  the molecules we define in the Hermite context is expressed in terms of estimates of the type
$$ \Big|\int_{\RR^n} (y-x)^\gamma m(y)\,dy\Big|\le |B(x,r)|^{\f{1}{2}}\cro(x)^{|\gamma|} \Big(\f{r}{\cro(x)}\Big)^{M+1} \qquad   \forall \;\gamma\in \NN_0 \text{ such that } |\gamma|\le M,$$
where $\cro(x)=\f{1}{1+|x|},$ $x\in\RR^n$ and $r\le \cro(x)$ (see Definition \ref{def: molecules}). This condition is essential in our proofs of boundedness results for Hemite pseudo-multipliers  that allow for the smoothness index to be zero or negative.

Our results on Hermite Besov and Hermite Triebel--Lizorkin spaces complement several articles that have extended the 
 work in \cite{MR808825, MR1070037} for the classical  Besov and Triebel--Lizorkin spaces to a variety of other settings associated to non-negative self-adjoint operators. For instance, the authors of \cite{MR3319540} develop molecular decompositions and molecular synthesis estimates for scales of Besov and Triebel--Lizorkin spaces associated to the Hermite operator that are defined in terms of heat kernels via square functions. We note that the  spaces $B^{p,q}_{\alpha}(\LL)$ and $F^{p,q}_{\alpha}(\LL)$  and the concept of molecules introduced in this article are different to those treated in \cite{MR3319540}; in particular, the notion of  molecules we present is new in this context and closer in spirit to the  molecules in \cite{MR808825, MR1070037} as described above.  
 The authors of \cite{MR3271256} present  frame decompositions for Besov and Triebel--Lizorkin spaces in the context of a Dirichlet space with a doubling measure and local scale-invariant Poincar\'e inequality; the authors of \cite{MR3601596,GKKP2018} study discrete frame decompositions and atomic and molecular decompositions for homogeneous Besov and Triebel--Lizorkin spaces in the setting of a doubling metric measure space in the presence of a non-negative self-adjoint operator whose heat kernel has Gaussian localization and the Markov property; the authors of \cite{MR4070730} obtain atomic decompositions for weighted Besov and Triebel--Lizorkin spaces in the context of spaces of homogeneous type with a non-negative self-adjoint operator satisfying Gaussian upper bounds on its heat kernels. 

The classes of symbols we consider  are reminiscent of the well-known H\"ormander classes corresponding to the Euclidean setting; they include symbols that satisfy estimates of the type
\begin{align}\label{eq: symbols}
	| \partial^\nu_x \diff^\kappa_\xs \sigma(x,\xs)| \lesssim g(x,\xs)(1+\sqrt{\xs})^{m-2\kappa+\delta |\nu|}  \qquad \forall (x,\xs)\in \RR^n\times\NN_0
\end{align}
for $0\le |\nu|\le \N$ and $0\le \kappa\le \K,$ where $\N,\K\in \NN_0,$ $m\in \RR,$  $0\le \delta\le 1,$ and  $g$ is a function that admits some exponential growth. 
Symbols in the spirit of \eqref{eq: symbols} have been considered (albeit in the absence of the growth function $g$) in   \cite{Bui} and  \cite{MR3622656}  for  pseudo-multipliers associated  with the Hermite operator and for pseudo-differential operators associated with non-negative self-adjoint operators, respectively. These works prove boundedness of such operators in corresponding Besov and Triebel--Lizorkin spaces, but only with positive smoothness and indices $1<p<\infty$ and $1<q<\infty$. They also do not address the endpoint $\dd=1$ (which is analogous to the so-called ``forbidden class'' in the classical setting). 
In contrast, we are able to show that condition \eqref{eq: symbols} alone with $0\le \delta<1$, or condition \eqref{eq: symbols}  with $\delta=1$ along with newly introduced cancellation conditions on $\sigma$ (see Definition \ref{def: canc})  lead to boundedness properties of $T_\sigma$ from $B^{p,q}_{\alpha+m}(\LL)$ to  $B^{p,q}_{\alpha}(\LL)$ and  from $F^{p,q}_{\alpha+m}(\LL)$ to $F^{p,q}_{\alpha}(\LL)$ for $0<p<\infty,$ $0<q<\infty$ and  ranges of the smoothness parameter $\alpha$ that allow for negative values (see Section~\ref{sec: PDOs} for more results and details). It is also worth mentioning that while the smoothness and decay conditions that we impose on the symbols in \eqref{eq: symbols} have some  similarities with those  in \cite{Bui,MR3622656}, our symbols  admit some further exponential growth as described above. Moreover, the techniques used in  \cite{Bui} and  \cite{MR3622656} are different from ours and rely on decompositions of the symbols in elementary pieces.

 Inspired by the works \cite{MR631751, MR639462},  our results on boundedness properties of pseudo-multipliers also imply that Hermite Besov spaces and Hermite Triebel--Lizorkin spaces are closed under non-linearities. More precisely, we prove that if $f$ is a real-valued function in  $B^\alpha_{p,q}(\LL)\cap L^\infty(\RR^n)$  (or, $F^\alpha_{p,q}(\LL)\cap L^\infty(\RR^n)$), then $H(f)\in B^\alpha_{p,q}(\LL)\cap L^\infty(\RR^n)$ (respectively, $H(f)\in F^\alpha_{p,q}(\LL)\cap L^\infty(\RR^n)$), where $H$ is an infinitely differentiable function defined on $\RR$ that satisfies $H(0)=0.$  In particular, the spaces $ B^\alpha_{p,q}(\LL)\cap L^\infty(\RR^n)$ and $ F^\alpha_{p,q}(\LL)\cap L^\infty(\RR^n)$ turn out to be closed under pointwise multiplication. We refer the reader to Section~\ref{sec:  meyer} for more details and conditions on the parameters.

The organization of the article is as follows. In Section~\ref{sec: prelim}, we introduce notation and present background material related to Hermite functions, Hermite tiles and Besov and Triebel--Lizorkin spaces in the Hermite setting. In Section~\ref{sec: molecules}, we introduce a novel notion of molecules associated to such spaces  and prove new  results concerning almost orthogonality (Lemma~\ref{lem: AO}), molecular decomposition and synthesis (Theorems~\ref{th:decomp} and \ref{th:synth}). In Section~\ref{sec: PDOs}, we pursue the study of boundedness properties  in Hermite Besov and Hermite Triebel--Lizorkin spaces for pseudo-multipliers with symbols in H\"ormander-type classes adapted to the Hermite setting (Theorems~\ref{th: Tsmooth}, \ref{th: Tcanc}, \ref{th: main1} and \ref{th: main2} and Corollary~\ref{coro:main}). In Section~\ref{sec: applications}, we present examples and applications of the theorems proved in Section~\ref{sec: PDOs}; in particular, we discuss boundedness results for Hermite pseudo-multipliers in  Lebesgue and Hermite local Hardy spaces (Corollaries~\ref{cor:Lp1}, \ref{cor:Lp2}, \ref{cor:wLp} and \ref{cor:hardy}) and show that Hermite Besov spaces and Hermite Triebel--Lizorkin spaces are closed under non-linearities (Theorem~\ref{thm: appnonlin} and Corollary~\ref{coro: appnonlin}). Finally, the appendices contain the proofs of some technical lemmas.

\medskip

\textbf{Notation:}
We set
$\NN_0=\NN\cup\{0\}$ and
 $|x|_\infty = \max_{1\le i\le n} |x_i|$. 
The notation $Q(x_Q,r_Q)=\{x\in\RR^n: |x-x_Q|_\infty<r_Q\}$ denotes the cube centered at  $x_Q\in \RR^n$ of side-length $2r_Q$. The smallest integer greater than $\alpha\in\RR$ is denoted by $\ceil{\alpha}$, while the largest integer not exceeding $\alpha$ is denoted $\floor{\alpha}$.  We also set $\alpha_+ = \max\{0,\alpha\}$ and $\alpha^*=\alpha-\floor{\alpha}$. Given real numbers $a$ and $b,$ we set  $a\lor b=\max\{a,b\}$ and $a\land b=\min\{a,b\}.$

We make use of  the following multi-index conventions. If $\alpha,\beta \in \NN_0^n$ are multi-indices, then $\alpha\le\beta$ means $\alpha_i\le\beta_i$ for each $i\in\{1,\dots,n\}$, $\binom{\beta}{\alpha} = \binom{\beta_1}{\alpha_1}\binom{\beta_2}{\alpha_2}\dots\binom{\beta_n}{\alpha_n}$, $\alpha != \alpha_1 !\alpha_2!\dots \alpha_n!$,  $x^\alpha = x_1^{\alpha_1}x_2^{\alpha_2}\dots x_n^{\alpha_n}$, and $\partial^\alpha = \partial_1^{\alpha_1}\partial_2^{\alpha_2}\dots \partial_n^{\alpha_n}$.

For $i\in\{1,\dots,n\}$ we set $\A{x}_i=-\f{\partial}{\partial x_i}+x_i$. When the variable under consideration is clear we just write $A_i$. If $\alpha\in\NN_0^n,$ we define  $A^\alpha=A_1^{\alpha_1}A_2^{\alpha_2}\dots A_n^{\alpha_n}$. 

We denote by $\sz(\RR^n)$ the space of Schwartz functions on $\RR^n$ and by $\sz'(\RR^n)$ the space of tempered distributions on $\RR^n$. The letter $n$ will always mean Euclidean dimension.

For a locally integrable function $f$ and measureable set $E\subset \RR^n$  also use the notation  
$\displaystyle \aver{E} f=\f{1}{|E|}\int f$
to denote the average  of $f$ over $E$.

\medskip

{\bf{Acknowledgements:}} The authors thank Lesley Ward  and the anonymous referees for their valuable input and suggestions. The first author thanks The Anh Bui for useful discussions.

\section{Preliminaries}\label{sec: prelim}
In this section  we introduce notation and present background material related to Hermite functions, Hermite tiles, and Besov and Triebel--Lizorkin spaces in the Hermite setting.

For each $k\in\NN_0$, the Hermite function of degree $\kk$ is 
\begin{align*}
	h_k(t)=(2^k k! \sqrt{\pi})^{-1/2}H_k(t)e^{-t^2/2} \qquad \forall\;t\in\RR{},
\end{align*}
where 
\begin{align*}
	H_k(t)=(-1)^k e^{t^2}\partial_t^k(e^{-t^2})
\end{align*}
is the $k$-th Hermite polynomial. 

The $n$-dimensional Hermite functions $h_\xi$ are defined over the multi-indices $\xi\in \NN_0^n$ by 
\begin{align*}
	h_\xi(x)=\prod_{j=1}^n h_{\xi_j}(x_j)\qquad \forall x\in\RR^n.
\end{align*}
The Hermite functions are eigenfunctions of $\LL$ in the sense that 
\begin{align*}
	\LL(h_\xi) = \lambda_{|\xi|} h_\xi,
\end{align*}
where $\lambda_k = 2k+n$.
Furthermore, they form an orthonormal basis for $L^2(\RR^n)$.

Let $W_k=\text{span}\{h_\xi: |\xi|=k\}$ and $ V_N = \bigoplus^N_{k=0} W_k$. We define the orthogonal projection of $f\in L^2(\RR^n)$ onto $W_k$ by
\begin{align*}
	\PP_k f = \sum_{|\xi|=k} \ip{f, h_\xi} h_\xi
	\qquad\text{with kernel}\qquad \PP_k(x,y)=\sum_{|\xi|=k} h_\xi(x)h_\xi(y).
\end{align*}
We also define the orthogonal projection of $f$ onto $V_N$ by
\begin{align*}
	\QQ_N f = \sum_{k=0}^N \PP_k f
	\qquad\text{with kernel}\qquad
	\QQ_N(x,y)=\sum_{k=0}^N \sum_{|\xi|=k}h_\xi(x)h_\xi(y).
\end{align*}
The following bounds are known (see \cite[p.376]{MR2399106}): there exists $\vartheta>0$ such that for any $N\in \NN$
\begin{align}\label{QQ bound}
	\QQ_{N}(x,x) 
	\lesssim \left
\lbrace 
	\begin{array}{ll}
			N^{n/2}  \qquad &\forall x,\\
			e^{-2\vartheta|x|^2} \qquad &\text{if}\quad |x| \ge \sqrt{4N+2}.
				\end{array}
\right.
\end{align}

We will use the following function throughout:
\begin{align}\label{efn}
	\e_{N}(x) = \left
\lbrace 
	\begin{array}{ll}
			1  \qquad &\text{if}\quad |x|  < \sqrt{N},\\
			e^{-\vartheta |x|^2} \qquad &\text{if}\quad |x|\ge \sqrt{N}.
	\end{array}
\right.
\end{align}
It follows  from \eqref{QQ bound} that for any $\ve>4$ and $N\in\NN$, we have
\begin{align}\label{QQ est}
\QQ_{4^j+N}(x,x)\lesssim 2^{jn}(\e_{\ve 4^j}(x))^2 \qquad \forall \; j\in\NN_0,
\end{align}
where the implicit constant depends only on $N,$ $n,$ $\ve$ and $\vartheta.$
\subsection{Hermite tiles}
In our setting, the notion of \emph{Hermite tiles} replaces that of dyadic cubes. We briefly present such a concept in this section; further details can be found in \cite{MR4205734, MR2399106}. 

Fix a positive structural constant $\delta_\star$ small enough ($0<\delta_\star<\f{1}{37}$ suffices) and, for each $j\in\NN_0,$ consider 
the collection $\X_j$ of \emph{nodes}  defined as  the set of $n$-tuples of zeros of the Hermite polynomial $H_{2N_j}$ with 
$$ N_j = \floor{(1+11\delta_\star)(\tfrac{4}{\pi})^2 4^j} +3. $$
To each node $\zeta\in\X_j,$ we associate a \emph{tile} $R_\zeta$ with sides parallel to the axes, so that each such tile contains precisely one node and any two  different tiles with nodes in $\X_j$ have disjoint interiors. The tiles are approximately cubes along the diagonals of $\RR^n$, and are rectangular boxes off the diagonal (see Lemma \ref{lem: tiles} below). Given any tile $R$ we denote its node by $x_R$. 

We set $\E_j=\{R_\zeta\}_{\zeta\in\X_j}$ (i.e. $\E_j$ is   the collection of all $j$th level tiles)  and we define
$$\E=\bigcup_{j\ge 0}\E_j.$$ 
Note that, by construction, $\E_j$ contains approximately $4^{jn}$ tiles. 
Our article relies on the following properties of these tiles (see \cite{MR2399106}), which holds if $\delta_\star$ is chosen small enough:
\begin{Lemma}\label{lem: tiles}
There exist constants $c_0$, $c_1$, $c_2$, $c_3$ and $c_4$  depending only on $\delta_\star$ and $n,$ such that for each $j\in\NN_0$ and each tile $R\in\E_j$ the following properties hold. 
\begin{enumerate}[\upshape(a)]
\item If $|x_R|\le (1+4\delta_\star)2^{j+1},$ it holds that   
\begin{align*}
R\subset Q(x_R,c_0 2^{-j}).
\end{align*}
\item In general, it holds that
\begin{align*}
 Q(x_R, c_1 2^{-j})\subset R \subset Q(x_R, c_2 2^{-j/3}).
 \end{align*}
 \item Set $\Q_j=\bigcup\limits_{P\in\E_j}P = Q(0,\zeta_{N_j}+2^{-j/6});$ it holds that
\begin{align*} 
Q(0,2^j)\subset \Q_j \subset Q(0,c_32^j).
\end{align*}
\item $R$ can be subdivided into a disjoint union of subcubes with sidelength roughly equal to $2^{-j};$ more precisely, each such subcube $Q$ satisfies
$$ Q(x_Q, c_4 2^{-j-1})\subset Q \subset Q(x_Q, c_4 2^{-j}).$$
Denoting by $\widehat{\E}_j$ the collection of all subcubes corresponding to tiles in $\E_j,$ it holds that 
\begin{align}\label{hermite cubes}
 \bigcup_{Q\in\widehat{\E}_j}Q= \Q_j.
 \end{align}
 \end{enumerate}
 \end{Lemma}

By considering the two types of tiles listed in Lemma~\ref{lem: tiles} and  using \eqref{efn}, the following estimate is obtained with $\ve=4(1+4\delta_\star)^2$ and $s>0:$
\begin{align}\label{eq:tile control}
|R|^{1/2} 2^{jn} \big(\e_{\ve 4^j}(x_R)\big)^{s}\lesssim |R|^{-1/2}\quad \forall R\in \E_j, j\in \NN_0.
\end{align}
 Indeed, if $|x_R|\le (1+4\delta_\star)2^{j+1}$ then $|R|\sim 2^{-jn}$ and \eqref{eq:tile control} follows. On the other hand, if $|x_R|>(1+4\delta_\star)2^{j+1}$ then \eqref{eq:tile control} follows from the fact that $2^{-jn}\lesssim |R|\lesssim 2^{-jn/3}$  and the bound  $\e_{\ve 4^j}(x_R)^{s}\lesssim 2^{-j\beta}$ with $\beta>0$.

We next observe, for later use, that if  $\beta\ge 0$ then it holds that 
\begin{align}\label{eq: ebound}
\e_{\ve 4^j}(x)\lesssim \Big(1+\f{|x|}{2^j}\Big)^{-\beta}\quad \forall x\in \RR^n, j\in \NN_0,
\end{align}
where the implicit constant depends on $\ve$ and $\beta.$ Indeed, if $|x|\le \sqrt{\ve}2^j$ then, for any $\beta\ge 0,$ we have
$$ \e_{\ve 4^j}(x)=1\le (1+\sqrt{\ve})^\beta \Big(1+\f{|x|}{2^j}\Big)^{-\beta}\quad \forall x\in \RR^n, \forall j\in \NN_0.$$
On the other hand, if $|x|>\sqrt{\ve}2^j$ then $2(1+|x|)>2^j+|x|$ since $\ve > 4;$ hence, for any $\beta\ge0$, we obtain
$$ \e_{\ve 4^j}(x)=e^{-\vartheta|x|^2} \le C_\beta (1+|x|)^{-\beta} \le C_\beta 2^\beta (2^j+|x|)^{-\beta} \quad \forall x\in \RR^n, \forall j\in \NN_0,$$
which gives \eqref{eq: ebound} since $2^{j\beta}\ge 1$.

The following  estimates will be employed throughout the paper.
\begin{Lemma}\label{lem: hoppe}
Let $\phi$ be a smooth function defined in $[0,\infty),$ set $\phi_j(x)=\phi(2^{-j}x)$ for $j\in \NN_0$ and consider  $\ell, N\in \NN,$ $N>\ell.$
\begin{enumerate}[\upshape(a)]
\item If $\phi^{(m)}(0)=0$ for all $m\in \NN,$ it holds that 
\begin{equation*}
 |\diff^\ell_k(\phi_j(\sqrt{\lambda_k}))|\lesssim   \|\phi^{(N)}\|_{L^\infty}  2^{-jN} \lambda_k^{N/2-\ell}\quad \forall j,k\in \NN_0,
 \end{equation*}
 where the implicit constant depends on $N$ and $\ell.$
 \item If $\phi^{(m)}(0)\ne 0$ for some $m\in \NN,$ it holds that
 \begin{equation*}
 |\diff^\ell_k(\phi_j(\sqrt{\lambda_k}))|\lesssim  \max_{1\le s\le N}\{\|\phi^{(s)}\|_{L^\infty}\} \lambda_k^{-\ell/2}\quad \forall j,k \in \NN_0,
 \end{equation*}
  where the implicit constant depends on $N$ and $\ell.$

\end{enumerate}
\end{Lemma}
\begin{proof} 
By the mean value theorem for finite differences and Hoppe's chain rule we have 
\begin{align}\label{eq: hoppe}
|\diff^\ell_k(\phi_j(\sqrt{\lambda_k}))|=|\partial_\nu^\ell(\phi_j(\sqrt{\lambda_\nu}))|=\left|\sum_{r=1}^\ell c_r  \phi^{(r)}(2^{-j}\sqrt{\lambda_\nu})\lambda_\nu^{r/2-\ell}2^{-jr}\right|
\end{align}
 for some $k\le \nu\le k+\ell$ and appropriate constants $c_r.$

 \medskip
 
To prove (a),
 let $\bar{x}$ be between $0$ and $x$ and such that
 \begin{equation*}
 \phi^{(r)}(x)=\sum_{s=0}^{N-r-1}\frac{\phi^{(r+k)}(0)}{s!} x^s+\frac{\phi^{(N)}(\bar{x})}{(N-r)!} x^{N-r}= \frac{\phi^{(N)}(\bar{x})}{(N-r)!} x^{N-r},\quad r\in \NN.
 \end{equation*}
 This leads to $|\phi^{(r)}(x)|\lesssim \|\phi^{(N)}\|_{L^\infty}|x|^{N-r}$ and therefore, by \eqref{eq: hoppe}, we obtain
 \begin{equation*}
 |\diff^\ell_k(\phi_j(\sqrt{\lambda_k}))|\lesssim \|\phi^{(N)}\|_{L^\infty} \sum_{r=1}^\ell |c_r|  |2^{-j}\sqrt{\lambda_\nu}|^{N-r}\lambda_\nu^{r/2-\ell}2^{-jr}\lesssim  \|\phi^{(N)}\|_{L^\infty}  2^{-jN} \lambda_k^{N/2-\ell}.
 \end{equation*}
For part (b), $\eqref{eq: hoppe}$ gives
  \begin{equation*}
 |\diff^\ell_k(\phi_j(\sqrt{\lambda_k}))|\lesssim \max_{1\le s\le \ell}\{\|\phi^{(s)}\|_{L^\infty}\} \lambda_k^{-\ell/2} \le \max_{1\le s\le N}\{\|\phi^{(s)}\|_{L^\infty}\} \lambda_k^{-\ell/2}. \qedhere
 \end{equation*}
 \end{proof}

\subsection{Hermite Besov and Hermite Triebel--Lizorkin spaces}\label{sec: spaces}

We begin this section by introducing some notation that will be used in the definition of the Hermite Besov and Hermite Triebel--Lizorkin spaces.

\begin{Definition}[Admissible functions]\label{def: admissible}
We say that $(\varphi_0, \varphi)$ is an admissible pair if $\varphi_0, \varphi\in C^\infty(\RR_+)$ and for some constants $b_0>0,$ $0<b_1<1$  and $1/4<b_2<b_3<1,$
\begin{align*}
	&\supp\varphi_0\subset [0,1], & & |\varphi_0| > b_0 \quad\text{on}\quad [0,b_1], & & \varphi_0^{(m)}(0)=0 \quad \forall m\in\NN, \\
	&\supp\varphi\subset [\tfrac{1}{4},1], & & |\varphi|>b_0 \quad\text{on}\quad [b_2,b_3].
\end{align*}
Given an admissible pair $(\varphi_0, \varphi)$, we set $\varphi_j(\lambda)=\varphi(2^{-j}\lambda)$ if $j\in \NN$, $\lambda\in\RR_+$ and call the resulting collection $\{\varphi_j\}_{j\in\NN_0}$ an admissible system. 
\end{Definition}

Recall that the Hermite functions $h_\xi$ with $\xi\in \NN_0^n$ are members of $\sz(\RR^n)$. Then for an admissible system $\{\varphi_j\}_{j\in\NN_0}$ we may define the operators $\varphi_j(\sqrt{\LL})$ on $\sz'(\RR^n)$ by
$$ \varphi_j(\sqrt{\LL})f(x)=\sum_{\kk\in\NN_0}\varphi_j(\sqrt{\lambda_\kk}) \,\PP_\kk(f)(x)\qquad \forall f\in\sz'(\RR^n), x\in \RR^n,$$
where  $ \ip{f,\phi}=f(\phi)$ for $f\in\sz'(\RR^n)$ in the expression for $\PP_\kk(f).$ 
The kernels of the operators $\vp_j(\sqrt{\LL})$ are given by
$$ \vp_j(\sqrt{\LL})(x,y) 
= \sum_{\kk\in\NN_0} \vp_j(\sqrt{\lambda_\kk}) \, \PP_\kk (x,y) 
=\sum_{\kk\in\NN_0} \vp_j(\sqrt{\lambda_\kk}) \sum_{|\xi|=\kk} h_\xi(x) h_\xi(y).
$$
Denote by $\{I_j\}_{j\in\NN_0}$ the following subsets of $\NN_0$: $I_j=[\frac{1}{2}4^{j-2}-\floor{\frac{n}{2}}, \frac{1}{2}4^{j}-\ceil{\frac{n}{2}}]\cap \NN_0$ for $j\in \NN,$ $I_0=\{0\}$ if $n=1$ and $I_0=\emptyset$ if $n\ge 2$.
In view of the support of $\varphi_j,$ it follows that
$$ \varphi_j(\sqrt{\LL})f(x)=\sum_{\kk\in I_j}\varphi_j(\sqrt{\lambda_\kk})\, \PP_\kk(f)(x).$$
We note that the kernels $\vp_j(\sqrt{\LL})(x,y)$ satisfy some useful smoothness and cancellation estimates that will play an important throughout this paper. These estimates can be found in Lemma \ref{lem: phiest} in the Appendix.

We are now ready to define the Hermite Besov and Hermite Triebel--Lizorkin spaces.

\begin{Definition}[Hermite distribution spaces]
Let  $\alpha\in\RR$ and  $0<q\le\infty$.
For $0<p\le\infty,$ we define the  Hermite Besov space $B^{p,q}_{\alpha} = B^{p,q}_{\alpha}(\LL)$ as the class of tempered distributions $f\in \sz'(\RR^n)$ such that 
$$ \Vert f\Vert_{B^{p,q}_{\alpha}} = \Big(\sum_{j\in\NN_0}\big(2^{j\alpha}\Vert \varphi_j(\sqrt{\LL})f\Vert_{L^p}\big)^q\Big)^{1/q}<\infty;$$
for $0<p<\infty,$ we define the Hermite Triebel--Lizorkin space  $F^{p,q}_{\alpha} = F^{p,q}_{\alpha}(\LL)$ as the class of tempered distributions $f\in \sz'(\RR^n)$ such that 
$$ \Vert f\Vert_{F^{p,q}_{\alpha}} = \Big\Vert \Big(\sum_{j\in\NN_0}\big(2^{j\alpha}| \varphi_j(\sqrt{\LL})f|\big)^q\Big)^{1/q} \Big\Vert_{L^p}<\infty.$$
\end{Definition}

\begin{Definition}[Hermite sequence spaces]
Let  $\alpha\in\RR$ and  $0<q\le\infty$.
For $0<p\le\infty,$  we define the Hermite Besov sequence space $b^{p,q}_{\alpha}=b^{p,q}_{\alpha}(\LL)$ as the set of all sequences of complex numbers $s=\{s_R\}_{R\in \E}$ such that 
$$ \Vert s\Vert_{b^{p,q}_{\alpha}} =\bigg\{ \sum_{j\in\NN_0} 2^{j\alpha q}\Big(\sum_{R\in\E_j}  \big(|R|^{1/p-1/2}|s_R|\big)^p \Big)^{q/p}\bigg\}^{1/q} <\infty;
$$
for $0<p<\infty,$  we define the  Hermite Triebel--Lizorkin sequence space $f^{p,q}_{\alpha}=f^{p,q}_{\alpha}(\LL)$ as the set of all sequences of complex numbers $s=\{s_R\}_{R\in \E}$ such that
$$ \Vert s\Vert_{f^{p,q}_{\alpha}}= \bigg\Vert \Big(\sum_{j\in\NN_0} 2^{j\alpha q}\sum_{R\in\E_j}\big(\Ind_R(\cdot)|R|^{-1/2} |s_R|\big)^q\Big)^{1/q}\bigg\Vert_{L^p} <\infty.
$$
\end{Definition}

We will use the notation $A^{p,q}_\alpha(\LL)$  (or $A^{p,q}_\alpha$) to refer to $B^{p,q}_\alpha(\LL)$ or $F^{p,q}_\alpha(\LL),$ with the understanding that  $\alpha\in \RR,$ $0<q\le \infty,$ $0<p\le \infty$ if $A=B$ and $0<p<\infty$ if $A=F.$ An analogous comment applies to the sequence spaces, denoted by $a^{p,q}_\alpha(\LL)$ (or $a^{p,q}_\alpha$). 

The spaces $A^{p,q}_\alpha(\LL)$ are independent of the choice of $(\varphi_0,\varphi)$ (see \cite[Theorems 3 and 5]{MR2399106} and also the earlier works  \cite{MR1330218,MR1403122} for Triebel--Lizorkin spaces). Moreover, $A^{p,q}_{\alpha}(\LL)$ are quasi-Banach spaces continuously embedded in $\sz'(\RR^n)$ (see \cite[Proposition 4 and p.392]{MR2399106}) and have $\sz(\RR^n)$ as a dense subspace for finite values of $p$ and $q.$  In addition, as shown in \cite{MR2399106}, the spaces $A^{p,q}_{\alpha}(\LL)$ are in general different from the classical Triebel--Lizorkin and Besov spaces associated to the Laplacian operator. On the other hand, it holds that $F^{p,2}_0(\LL)=L^p(\RR^n) $ for $1<p<\infty$ with equivalent norms.

We adopt the notation
\begin{align*}
	n_{p,q} = \left
\lbrace 
	\begin{array}{ll}
			\f{n}{\min\{1,p,q\}}  \qquad &\text{for } \quad F^{p,q}_\alpha(\LL),\\
			\f{n}{\min\{1,p\}} \qquad &\text{for } \quad B^{p,q}_\alpha(\LL).
	\end{array}
\right.
\end{align*}

\subsubsection{Frame decompositions} 
The construction of frames  for Hermite Besov and Hermite Triebel--Lizorkin spaces given in \cite{MR2399106} relies on a certain cubature formula for functions in $V_N$. 
Before presenting the cubature formula we introduce the function 
$$ \tau(N,x)= \f{1}{\QQ_N(x,x)} \qquad \forall x\in\RR, N\in \NN_0,$$
which is the well known \emph{Christoffel function}. It has certain useful asymptotic properties which are listed in \cite[p.376]{MR2399106}. 

 	The following cubature formula  \cite[Corollary 2]{MR2399106}
	\begin{align}\label{eq:cubature}
	 \int_{\RR^n} f(x)\,g(x)\,dx \sim \sum_{\zeta\in \X_j} \tau_\zeta\, f(\zeta)\,g(\zeta),\qquad \zeta = (\zeta_{\alpha_1},\dots,\zeta_{\alpha_n}), \quad\tau_\zeta = \prod_{k=1}^n \tau(2N_j,\zeta_{\alpha_k}), \end{align}
	is exact for all $f\in V_k$ and $g\in V_\ell$ with $k+\ell \le 4N_j-1$.

If $\{\varphi_j\}_{j\in\NN_0}$ is an admissible system,  for  each tile $R\in \E_j$ we set
$$ \varphi_R(x)=\tau_R^{1/2} \varphi_j(\sqrt{\LL})(x,x_R),$$
where $x_R$ is the node of $R$ and $\tau_R =\tau_{x_R}$ is the coefficient in the cubature formula \eqref{eq:cubature}. It holds that $\tau_{R}\sim |R|$ for any tile $R$ (see \cite[(2.33)]{MR2399106}). Following the convention introduced in \cite{MR2399106}, we  refer to the functions $\varphi_R$ as \emph{needlets}.

Given  admissible systems $\{\varphi_j\}_{j\in\NN_0}$ and $\{\psi_j\}_{j\in\NN_0}$ we define the \emph{analysis}  operator $S_\varphi$ and \emph{synthesis} operator $T_\psi$ by
\begin{align*}
	S_\varphi: f\longmapsto \{ \ip{f,\varphi_R}\}_{R\in\E} &&\text{and}&&
	T_\psi: \{s_R\}_{R\in\E}\longmapsto \sum_{R\in\E} s_R\psi_R.
\end{align*}

The following frame decompositions were proved in \cite{MR2399106} (see also \cite{MR4205734}).
\begin{Theorem}[Frame decomposition]\label{th: frame}
Let $\alpha\in\RR$, $0<q\le \infty,$ and $0<p<\infty$ if $A^{p,q}_\alpha(\LL)=F^{p,q}_\alpha(\LL)$ or $0<p\le\infty$ if $A^{p,q}_\alpha(\LL)=B^{p,q}_\alpha(\LL)$. Then,
\begin{enumerate}[\upshape(a)]
	\item the operator $T_\psi: a^{p,q}_{\alpha}(\LL)\to A^{p,q}_{\alpha}(\LL)$ is bounded;
	\item the operator $S_\varphi: A^{p,q}_{\alpha}(\LL)\to a^{p,q}_{\alpha}(\LL)$ is bounded;
	\item if $\{\varphi_j\}_{j\in\NN_0}$ and $\{\psi_j\}_{j\in\NN_0}$   satisfy
	\begin{align}\label{CRF 1}
	\sum_{j\ge 0}\psi_j(\lambda)\varphi_j(\lambda) = 1 \qquad \forall \lambda\ge 0,
	\end{align}
	 then $T_\psi\circ S_\varphi = I$ on $A^{p,q}_{\alpha}(\LL)$ (with convergence in $\sz'(\RR^n)$ --  \cite[Proposition 3]{MR2399106}.)
\end{enumerate}
\end{Theorem}

\section{Molecules in the Hermite setting}\label{sec: molecules} 

In this section, we introduce a novel notion of molecules associated to the Hermite Besov and Hermite Triebel--Lizorkin spaces  and prove new  results concerning almost orthogonality (Lemma~\ref{lem: AO} in Section~\ref{sec:ao}), molecular decomposition and synthesis (Theorem~\ref{th:decomp} in Section~\ref{sec:decomp} and Theorem~\ref{th:synth} in Section~\ref{sec:synth}). Lemma~\ref{lem: AO} is used in the proof of Theorem~\ref{th:synth}, which is a crucial tool in the proof of the results of Section~\ref{sec: PDOs}.

\subsection{Molecules and almost orthogonality results}\label{sec:ao} We start this section with our new  definition of molecule associated to the Hermite setting.
\begin{Definition}[Smooth molecules]\label{def: molecules}
Let $(M,\theta)\in \{\NN_0\times(0,1)\} \cup \{(-1,1)\}$, $N\in\NN_0$, $0\le \delta\le 1$ and $\mu\ge1$. A function $m\in C^N(\RR^n)$ is said to be an $(M,\theta, N,\delta, \mu)$-molecule for $\LL$ associated with a tile $R\in\E_j$ for some $j\in\NN_0$ if 
\begin{enumerate}[\upshape(i)]
\item for each multi-index $\gamma$ with $0\le|\gamma|\le N$ we have
$$ |\partial^\gamma m(x)| \le \f{|R|^{-1/2}2^{j|\gamma|}}{(1+2^j|x-x_R|)^\mu} \f{1}{\Big(1+\f{|x|}{2^j}\Big)^{N+\delta}}\quad \forall x\in \RR^n,$$
\item for each multi-index $\gamma$ with $|\gamma|=N$ we have
$$ |\partial^\gamma m(x)-\partial^\gamma m(y)| \le |R|^{-1/2}2^{j|\gamma|}\Big(\f{|x-y|}{2^{-j}}\Big)^\delta \f{1}{\big(1+2^j|x-x_R|\big)^\mu} $$
for every $x,y\in\RR^n$ with $|x-y|\le 2^{-j}$. 
\item for each multi-index $\gamma$ with $0\le|\gamma|\le M$ we have
$$\Big|\int_{\RR^n} (y-x_R)^\gamma m(y)\,dy\Big|\le |R|^{-1/2} 2^{-j(n+|\gamma|)}\Big(\f{1+|x_R|}{2^j}\Big)^{M+\theta-|\gamma|}.$$
If $(M,\theta)=(-1,1),$ part {\upshape{(iii)}} is taken to be void. 
\end{enumerate}
\end{Definition}
\begin{Remark}\label{re: molecules}  The following calculations show that if $m$ satisfies (i) for some $N \in\NN$, then $m$ also satisfies (ii) (modulo a constant) for $N-1,$  any $0\le \delta\le 1$  and the same value of $\mu.$

If $|x-y|\le 2^{-j}$ then $(2^j|x-y|)\le (2^j|x-y|)^\delta$ for any $\delta\in[0,1]$. By the mean value theorem and part (i) we have, for some $\widetilde{x}$ between $x$ and $y,$
\begin{align*}
|\partial^\gamma m(x)-\partial^\gamma m(y)|
&\le \sum_{|\beta|=N} |\partial^\beta m(\widetilde{x})||x-y| \\
&\lesssim |R|^{-1/2}2^{jN} |x-y| \big(1+2^j|\widetilde{x}-x_R|\big)^{-\mu} \\
&= |R|^{-1/2}2^{j(N-1)} \Big(\f{|x-y|}{2^{-j}}\Big) \big(1+2^j|\widetilde{x}-x_R|\big)^{-\mu} \\
&\le |R|^{-1/2}2^{j(N-1)}\Big(\f{|x-y|}{2^{-j}}\Big)^\delta\big(1+2^j|\wt{x}-x_R|\big)^{-\mu}.
\end{align*}
We may then conclude our estimate after observing that the triangle inequality, along with the fact that $|x-\wt{x}|\le |x-y|\le 2^{-j}$, yields
$$ \f{1}{\big(1+2^j|\wt{x}-x_R|\big)^{\mu}}\le \Big(\f{1+2^j|x-\wt{x}|}{1+2^j|x-x_R|}\Big)^\mu \le \f{2^\mu}{\big(1+2^j|x-x_R|\big)^\mu}.$$
\end{Remark}

The following lemma presents a first example of molecules.

\begin{Lemma}[Needlets are molecules]\label{lem: needlets}
If $\{\vp_j\}_{j\ge 0}$ is an admissible system, the functions $\{\vp_R\}_{R\in\E}$ are multiples of $(M,\theta, N,\delta,\mu)$-molecules for any $(M,\theta)\in \{\NN_0\times(0,1)\} \cup \{(-1,1)\}$, $N\in \NN_0,$ $0\le \delta\le 1$  and $\mu\ge1$.
\end{Lemma}
\begin{proof}[Proof of Lemma \ref{lem: needlets}]
We will use  Lemma \ref{lem: phiest} to show that $\vp_R$ is a constant multiple of a molecule. The estimate in part (i) of Definition \ref{def: molecules}, with a uniform constant in $R$ and $j,$ follows  for all $\gamma\in \NN_0^n$ from \eqref{phiest A} and estimates \eqref{eq:tile control}
and \eqref{eq: ebound}. The estimate in  part (ii) of Definition \ref{def: molecules} follows, with a uniform constant in $R$ and $j,$ from Remark~\ref{re: molecules}. Part (iii) of Definition \ref{def: molecules} follows, with a uniform constant in $R$ and $j,$ from \eqref{phiest B} and \eqref{eq:tile control} by choosing $K=M+\theta$. 
\end{proof}

The following almost orthogonality result will be useful in the proof of Theorem~\ref{th:synth} and is the main result of this subsection.

\begin{Lemma}[Almost orthogonality]\label{lem: AO}
Let $\{\vp_j\}_{j\in \NN_0}$ be an admissible system and $\{m_R\}_{R\in\E}$ be a collection of $(M, \theta, N, \delta, \mu)$-molecules for some $(M, \theta) \in \{\NN_0\times(0,1)\}\cup\{(-1,1)\}$, $N \in \NN_0$, $0\le \delta\le 1$ and $\mu>\max\{\eta, n + M +\theta\}$ for some $\eta>0$. Then it holds that 
\begin{align}\label{eq:AO1}
 |\vp_j(\sqrt{\LL})m_R(x)| \lesssim \f{|R|^{-1/2}}{(1+2^{j\land k}|x-x_R|)^\eta} 2^{-(n+M+\theta)[(k-j)\vee 0]-(N+\delta)[(j-k)\vee 0]}
 \end{align}
for all $x\in\RR^n,$ $j, k\in\NN_0$ and $R\in\E_k.$ 
\end{Lemma}
\begin{proof}[Proof of Lemma  \ref{lem: AO}]
We will use the estimates for $\{\vp_j\}_{j\in\NN_0}$ from Lemma \ref{lem: phiest} with decay $\eta_0>\eta+n+N+\delta$ and $|\gamma|\le M+1$ (in estimate \eqref{phiest A}), and $K=N+\delta$ (in estimate \eqref{phiest B}).
We will prove each of the following four cases with constants independent of $x,$ $j,$ $k$ and $R$:
\begin{description}
\item[1a]  $j\le k$, $(M,\theta)=(-1,1)$, $(N,\delta)\in  \NN_0\times [0,1]$, $\mu>\max\{\eta,n\}$
\begin{align}\label{eq:AO1a}
 |\vp_j(\sqrt{\LL})m_R(x)| \lesssim \f{|R|^{-1/2}}{(1+2^j|x-x_R|)^\eta} 2^{-(k-j)n} \end{align}
\item[1b] $j\le k$, $(M,\theta)\in \NN_0\times(0,1)$, $(N,\delta)\in  \NN_0\times [0,1]$, $\mu>\max\{\eta,n+M+\theta\}$
\begin{align}\label{eq:AO1b}
 |\vp_j(\sqrt{\LL})m_R(x)| \lesssim \f{|R|^{-1/2}}{(1+2^j|x-x_R|)^\eta} 2^{-(k-j)[n+M+\theta]} \end{align}
\item[2a] $j> k$,  $(M,\theta)\in \{\NN_0\times(0,1)\}\cup\{(-1,1)\}$, $(N,\delta)=(0,0)$, $\mu\ge \eta$
\begin{align}\label{eq:AO2a}
 |\vp_j(\sqrt{\LL})m_R(x)| \lesssim \f{|R|^{-1/2}}{(1+2^k|x-x_R|)^\eta} \end{align}
\item[2b] $j> k$, $(M,\theta)\in \{\NN_0\times(0,1)\}\cup\{(-1,1)\}$, $(N,\delta)\in \NN_0\times[0,1]$,  $(N,\delta)\neq (0,0),$ $\mu\ge \eta$
\begin{align}\label{eq:AO2b}
 |\vp_j(\sqrt{\LL})m_R(x)| \lesssim \f{|R|^{-1/2}}{(1+2^k|x-x_R|)^\eta} 2^{-(j-k)(N+\delta)} \end{align}
\end{description}
Estimate \eqref{eq:AO1} will then follow by combining each of the above cases appropriately. 

Fix $x$ and $R;$ to handle Cases {\bf{1a}} and {\bf{1b}}, it will be useful to divide $\RR^n$ into the following regions:
\begin{align*}
	\Omega_1&=\{y\in\RR^n: |y-x_R|\le 2^{-j}\},\\
	\Omega_2&=\{y\in\RR^n: |y-x_R|> 2^{-j} \;\text{and}\; |y-x|\le \f{1}{2}|x-x_R|\}, \\
	\Omega_3&=\{y\in\RR^n: |y-x_R|> 2^{-j} \;\text{and}\; |y-x|> \f{1}{2}|x-x_R|\}.
\end{align*}
The following two observations will be useful in the sequel: 
\begin{align}\label{est for omega1}
1+2^j|x-x_R|\le 2(1+2^j|x-y|) \qquad \forall\; y\in \Omega_1\cup \Omega_3,
\end{align}
and
\begin{align}\label{est for omega2}
2^{(k-j)}(1+2^j|x-x_R|) < 3\cdot 2^k|y-x_R| \qquad \forall \; y\in \Omega_2.
\end{align}

\medskip

\underline{Case {\bf{1a}}:}  We have
\begin{align*}
 |\vp_j(\sqrt{\LL})m_R(x)|
 &\le \int_{\Omega_1\cup\Omega_3} \big|\vp_j(\sqrt{\LL})(x,y)\big|\,|m_R(y)|\,dy + \int_{\Omega_2} \big|\vp_j(\sqrt{\LL})(x,y)\big|\,|m_R(y)|\,dy.
 \end{align*}
From \eqref{phiest A}  with $|\gamma|=0$ and $\eta_0\ge \eta+n$, Definition \ref{def: molecules} (i) with $|\gamma|=0$, and the inequality \eqref{est for omega1} we obtain
\begin{align*}
 \int_{\Omega_1\cup\Omega_3} \big|\vp_j(\sqrt{\LL})(x,y)\big|\,|m_R(y)|\,dy 
 &\lesssim \mathop{\int}_{1+ 2^j|x-x_R|\lesssim 1+2^j|y-x|} \f{2^{jn}}{(1+2^j|x-y|)^{\eta_0}}\f{|R|^{-1/2}}{(1+2^k|y-x_R|)^\mu}\,dy \\
&\lesssim 
 \f{|R|^{-1/2}2^{-(k-j)n} }{(1+2^j|x-x_R|)^\eta} 
\end{align*}
since $\mu>n$ and $\eta_0\ge \eta$. 
Proceeding similarly and applying \eqref{est for omega2}, we have
\begin{align*}
 \int_{\Omega_2} \big|\vp_j(\sqrt{\LL})(x,y)\big|\,|m_R(y)|\,dy 
&\lesssim \int_{\Omega_2} \f{2^{jn}}{(1+2^j|x-y|)^{\eta_0}}\f{|R|^{-1/2}}{(1+2^k|y-x_R|)^\mu}\,dy\\ 
&\lesssim \f{2^{-(k-j)n}|R|^{-1/2}}{(1+2^j|x-x_R|)^\eta} 
\end{align*}
since $j\le k,$ $\mu\ge \max\{n,\eta\}$ and $\eta_0>n$.

\medskip

\underline{Case {\bf{1b}}:}
We write
\begin{align*}
\vp_j(\sqrt{\LL})m_R(x)
&= \int_{\RR^n} \vp_j(\sqrt{\LL})(x,y)\,m_R(y)\,dy \\
&= \int_{\RR^n} \Big[ \vp_j(\sqrt{\LL})(x,y) - \sum_{|\gamma|\le M} \f{1}{\gamma !}\partial_y^\gamma \vp_j(\sqrt{\LL})(x,x_R) \,(y-x_R)^\gamma\Big]m_R(y)\,dy  \\
&\qquad + \sum_{|\gamma|\le M} \f{1}{\gamma !}\partial_y^\gamma \vp_j(\sqrt{\LL})(x,x_R) \int_{\RR^n}(y-x_R)^\gamma m_R(y)\,dy  \\
&=: I+ II.
\end{align*}
To study term $I$ we further subdivide $I=I_1+I_2+I_3$ where
$$ I_i = \int_{\Omega_i}\Big[ \vp_j(\sqrt{\LL})(x,y) - \sum_{|\gamma|\le M} \f{1}{\gamma !}\partial_y^\gamma \vp_j(\sqrt{\LL})(x,x_R) \,(y-x_R)^\gamma\Big]m_R(y)\,dy.$$
By Taylor's theorem we have
\begin{align*}
I_1= \sum_{|\gamma|= M+1} \f{1}{\gamma !}\int_{\Omega_1} \partial_y^\gamma \vp_j(\sqrt{\LL})(x,\wt{y})\,(y-x_R)^\gamma m_R(y)\,dy
\end{align*}
where $\wt{y}$ lies on the line segment connecting $y$ and $x_R$. Then, by \eqref{phiest A} with $|\gamma|=M+1$ and Definition \ref{def: molecules} (i) with $|\gamma|=0,$ we have
\begin{align*}
|I_1| 
&\lesssim \int_{\Omega_1} \f{2^{j(n+M+1)}}{(1+2^j |x-\wt{y}|)^{\eta_0}}|y-x_R|^{M+1}\f{|R|^{-1/2}}{(1+2^k |y-x_R|)^\mu}	\,dy \\
&\lesssim  \f{|R|^{-1/2} 2^{j(n+M+1)}}{(1+2^j |x-x_R|)^{\eta_0-1+\theta}}\int_{\Omega_1}\f{1}{(1+2^j|y-x_R|)^{1-\theta}}\f{|y-x_R|^{M+1}}{(1+2^k |y-x_R|)^{\mu}}	\,dy\\
&\lesssim  \f{|R|^{-1/2} 2^{j(n+M+\theta)}}{(1+2^j |x-x_R|)^{\eta_0-1+\theta}}\int_{\RR^n}\f{1}{(1+2^j|y-x_R|)^{1-\theta}}\f{|y-x_R|^{M+\theta}}{(1+2^k |y-x_R|)^{\mu}}	\,dy\\
&\lesssim \f{|R|^{-1/2} 2^{-(k-j)(n+M+\theta)}}{(1+2^j |x-x_R|)^{\eta}},
\end{align*}
since $\mu>n+M+\theta$ and $\eta_0> \eta+n>\eta+1-\theta$. In the second step we applied the triangle inequality along with the facts  $|y-x_R|\le 2^{-j}$ and $|\wt{y}-x_R|\le 2^{-j}$. 

For the second term $I_2$ we have 
\begin{align*}
|I_2|
\le  \int_{\Omega_2}  |\vp_j(\sqrt{\LL})(x,y)| |m_R(y)|\,dy +  \sum_{|\gamma|\le M} |\partial_y^\gamma \vp_j(\sqrt{\LL})(x,x_R)| \int_{\Omega_2} \,|y-x_R|^{|\gamma|} |m_R(y)|\,dy. 
\end{align*}
We next apply \eqref{phiest A} for $|\gamma| \le M$ and Definition \ref{def: molecules} (i) with $|\gamma|=0$ to each integral above.
The first integral in $I_2$ can be estimated in a similar way to the integral over $\Omega_2$ in Case {\bf{1a}} to  obtain
\begin{align*}
\int_{\Omega_2}  |\vp_j(\sqrt{\LL})(x,y)| |m_R(y)|\,dy
\lesssim \f{2^{-(k-j)\mu}|R|^{-1/2}}{(1+2^j|x-x_R|)^\mu} 
\le \f{2^{-(k-j)(n+M+\theta)}|R|^{-1/2}}{(1+2^j|x-x_R|)^\eta},
\end{align*}
since $j\le k$ and $\mu>\max\{\eta, n+M+\theta\}.$  
For the second term in $I_2$ we apply \eqref{est for omega2} to obtain, for each $|\gamma|\le M$, 
\begin{align*}
	|\partial_y^\gamma \vp_j(\sqrt{\LL})(x,x_R)| &\int_{\Omega_2} \,|y-x_R|^{|\gamma|} |m_R(y)|\,dy  \\
&\qquad\lesssim
\f{2^{j(n+|\gamma|)}|R|^{-1/2}}{(1+2^j|x-x_R|)^{\eta_0}} \int_{\Omega_2}\f{|y-x_R|^{|\gamma|}}{(1+2^k|y-x_R|)^\mu}\,dy \\
&\qquad\lesssim \f{|R|^{-1/2} 2^{j(n+|\gamma|)}2^{-(k-j)\mu}}{(1+2^j|x-x_R|)^{\eta_0+\mu}} \int_{\Omega_2} |y-x_R|^{|\gamma|}\,dy \\
&\qquad\lesssim |R|^{-1/2}2^{-(k-j)\mu} \f{(2^j |x-x_R|)^{n+M}}{(1+2^j|x-x_R|)^{\eta_0+\mu}}  \\
&\qquad\lesssim |R|^{-1/2}2^{-(k-j)(n+M+\theta)} \f{1}{(1+2^j|x-x_R|)^{\eta}},
\end{align*}
since $\mu \ge n+M+\theta$ and $\eta_0\ge \eta$.

For the third term,   we apply  \eqref{phiest A} with $|\gamma|\le M,$  Definition \ref{def: molecules} (i) with $|\gamma|=0,$ and use that  $|x-y|\gtrsim |x-x_R|,$ $\mu\ge n+M+\theta$ and $\eta_0\ge \eta,$ to obtain
\begin{align*}
	|I_3|
	&\lesssim \f{|R|^{-1/2}2^{jn}}{(1+2^j|x-x_R|)^{\eta_0}}\int_{\Omega_3} \Big[1+\sum_{|\gamma|\le M} (2^j|y-x_R|)^{|\gamma|}\Big]\f{1}{(1+2^k|y-x_R|)^\mu}\,dy \\
	&\lesssim \f{|R|^{-1/2}2^{j(n+M)}}{(1+2^j|x-x_R|)^{\eta_0}}\int_{|y-x_R|>2^{-j}} \f{|y-x_R|^M}{(1+2^k|y-x_R|)^\mu}\,dy \\
	&\lesssim \f{|R|^{-1/2}2^{-(k-j)(n+M+\theta)}}{(1+2^j|x-x_R|)^{\eta}}.
\end{align*}

For the term $II$ we apply \eqref{phiest A} and Definition \ref{def: molecules} (iii) with  $0\le|\gamma|\le M$ to obtain
\begin{align*}
|II|
&\lesssim 
\sum_{|\gamma|\le M} \f{2^{j(n+|\gamma|)}}{(1+2^j|x-x_R|)^{\eta_0}} \e_{\ve 4^j}(x_R) |R|^{-1/2}2^{-k(n+|\gamma|)}\Big(\f{1+|x_R|}{2^k}\Big)^{M+\theta-|\gamma|} \\
&\le \sum_{|\gamma|\le M} \f{|R|^{-1/2} 2^{-(k-j)(n+M+\theta)}}{(1+2^j|x-x_R|)^\eta} 2^{-j(M+\theta-|\gamma|)}(1+|x_R|)^{M+\theta-|\gamma|}\e_{\ve 4^j}(x_R),
\end{align*}
since $\eta_0\ge \eta.$ By considering separately the  cases when $|x_R| \le \sqrt{\ve}2^j$ and $|x_R| > \sqrt{\ve}2^j$ along with  \eqref{efn}, it follows  that
\begin{align}\label{eq:AO2}
 2^{-j(M+\theta-|\gamma|)}(1+|x_R|)^{M+\theta-|\gamma|}\e_{\ve 4^j}(x_R)\le C_{M,\theta,\vartheta,\ve}.
 \end{align}
 Inserting this bound into the preceding estimate  gives
 $$ |II|\lesssim \f{|R|^{-1/2} 2^{-(k-j)(n+M+\theta)}}{(1+2^j |x-x_R|)^{\eta}}.$$
 Combing the estimates for both terms $I$ and $II$ yields \eqref{eq:AO1b}. 
 
 \medskip
 
 \underline{Case {\bf{2a}}:} Using  \eqref{phiest A} and Definition \ref{def: molecules} (i) with $|\gamma|=0,$    $\mu\ge \eta$, the triangle inequality taking into account the $k< j$, and that  $\eta_0>\eta+n,$ we obtain
\begin{align*}
 |\vp_j(\sqrt{\LL})m_R(x)|
 \le \f{|R|^{-1/2}}{(1+2^k|x-x_R|)^{\eta}} \int_{\RR^n} \f{2^{jn}}{(1+2^j|y-x|)^{\eta_0-\eta}}\,dy 
\lesssim 
 \f{|R|^{-1/2} }{(1+2^k|x-x_R|)^\eta}.
\end{align*}
 
 \medskip

\underline{Case {\bf{2b}}:} We argue as in Case {\bf{1b}}, but reverse the roles of $\vp_j(\sqrt{\LL})$ and $m_R$. We have
\begin{align*}
\vp_j(\sqrt{\LL})m_R(x)
=I+II
\end{align*}
where
\begin{align*}
I&=  \int_{\RR^n}  \vp_j(\sqrt{\LL})(x,y) \Big[ m_R(y)- \sum_{|\gamma|< N} \f{1}{\gamma !}\partial^\gamma m_R(x) \,(y-x)^\gamma\Big]\,dy\\
&\hspace{80pt} -  \sum_{|\gamma|= N} \f{1}{\gamma !}\partial^\gamma m_R(x) \int_{\RR^n}(y-x)^\gamma \vp_j(\sqrt{\LL})(x,y)\,dy,\\
II&=\sum_{|\gamma|\le N} \f{1}{\gamma !}\partial^\gamma m_R(x) \int_{\RR^n}(y-x)^\gamma \vp_j(\sqrt{\LL})(x,y)\,dy.
\end{align*}

By Taylor's remainder theorem we write, for some $\wt{y}$ on the line segment between $x$ and $y$, 
\begin{align*}
I = \sum_{|\gamma| = N} \f{1}{\gamma !}\int_{\RR^n} \big[\partial^\gamma m_R(\wt{y})-\partial^\gamma m_R(x)\big] \,(y-x)^\gamma \vp_j(\sqrt{\LL})(x,y)\,dy =:I_1+I_2
\end{align*}
where
\begin{align*}
I_1 = \sum_{|\gamma| = N} \f{1}{\gamma !}\int_{|x-y|\le 2^{-k}} \big[\partial^\gamma m_R(\wt{y})-\partial^\gamma m_R(x)\big] \,(y-x)^\gamma \vp_j(\sqrt{\LL})(x,y)\,dy
\end{align*}
and
\begin{align*}
I_2 = \sum_{|\gamma| = N} \f{1}{\gamma !}\int_{|x-y|> 2^{-k}} \big[\partial^\gamma m_R(\wt{y})-\partial^\gamma m_R(x)\big] \,(y-x)^\gamma \vp_j(\sqrt{\LL})(x,y)\,dy.
\end{align*}
By setting $\eta_0>\eta+N+n+\delta,$ we have by \eqref{phiest A} with $|\gamma|=0$ and Definition \ref{def: molecules} (ii) with $|\gamma|=N$,   
\begin{align*}
|I_1| 
\lesssim \int_{|x-y|\le 2^{-k}}\f{2^{jn}|R|^{-1/2}2^{k(N+\delta)}}{\big(1+2^k |x-x_R|\big)^\mu}  \f{|x-\wt{y}|^\delta |y-x|^{N}}{(1+2^j |x-y|)^{\eta_0}} 	\,dy. 
\end{align*}
Taking into account $\mu\ge\eta$ and $|\wt{y}-x|\le |y-x|$ we have
\begin{align*}
|I_1| 
\lesssim \f{|R|^{-1/2}2^{k(N+\delta)}}{\big(1+2^k |x-x_R|\big)^\eta} \int_{\RR^n} \f{2^{jn}|y-x|^{N+\delta}}{(1+2^j |x-y|)^{\eta_0}} 	\,dy
\lesssim \f{|R|^{-1/2} 2^{-(j-k)(N+\delta)}}{(1+2^k|x-x_R|)^{\eta}}.
\end{align*}
For $I_2$ we apply instead Definition \ref{def: molecules} (i) to obtain
\begin{align*}
|I_2| 
&\le  \sum_{|\gamma| = N} \f{1}{\gamma !}\int_{|x-y|> 2^{-k}} \Big[\big|\partial^\gamma m_R(\wt{y})\big|+\big|\partial^\gamma m_R(x)\big|\Big] \,|y-x|^N \big|\vp_j(\sqrt{\LL})(x,y)\big|\,dy\\
&\lesssim \int_{|x-y|> 2^{-k}} \bigg[\f{1}{(1+2^k |\wt{y}-x_R|)^{\mu}}+\f{1}{(1+2^k |x-x_R|)^{\mu}}\bigg]\f{2^{jn}|R|^{-1/2}2^{kN}|y-x|^{N}}{\big(1+2^k |x-y|\big)^{\eta_0}} \,dy. 
\end{align*}
Taking into account that $j>k,$  $\mu\ge\eta$ and $|\wt{y}-x|\le |y-x|$ and using  the triangle inequality we have
$$ \f{1}{(1+2^k |\wt{y}-x_R|)^{\mu}}  \le \Big(\f{1+2^k|\wt{y}-x|}{1+2^k |x-x_R|}\Big)^{\eta} \le \Big(\f{1+2^j|y-x|}{1+2^k |x-x_R|}\Big)^{\eta}.$$
Inserting this inequality into the previous estimate, and using the fact that $(2^k|x-y|)^\delta >1$, we arrive at
\begin{align*}
|I_2| 
&\lesssim \f{|R|^{-1/2} 2^{k(N+\delta)}}{(1+2^k|x-x_R|)^{\eta}} \int_{|x-y|> 2^{-k}}  \f{2^{jn} |y-x|^{N+\delta}}{(1+2^j |y-x|)^{\eta_0-\eta}}	\,dy
\lesssim \f{|R|^{-1/2} 2^{-(j-k)(N+\delta)}}{(1+2^k|x-x_R|)^{\eta}}.
\end{align*}

Let us turn to  term $II$. Here we apply \eqref{phiest B} with $K=N+\delta$ and Definition \ref{def: molecules} (i) with  $0\le|\gamma|\le N,$ and use that $\mu\ge \eta,$ to obtain
\begin{align*}
|II|
&\lesssim 
\sum_{|\gamma|\le N} \f{|R|^{-1/2} 2^{k|\gamma|}}{(1+2^k|x-x_R|)^{\mu}} \Big(1+\f{|x|}{2^k}\Big)^{-N-\delta} \, 2^{-j|\gamma|}\Big(\f{1+|x|}{2^j}\Big)^{N+\delta-|\gamma|} \\
&= \sum_{|\gamma|\le N} \f{|R|^{-1/2} 2^{-(j-k)(N+\delta)}}{(1+2^k|x-x_R|)^\eta} 2^{-k(N+\delta-|\gamma|)}(1+|x|)^{N+\delta-|\gamma|} \Big(1+\f{|x|}{2^k}\Big)^{-N-\delta}.
\end{align*}
 Considering $|x|\le 2^k$ and $|x|\ge 2^k,$ we see that
\begin{align*}
2^{-k(N+\delta-|\gamma|)}(1+|x|)^{N+\delta-|\gamma|} \Big(1+\f{|x|}{2^k}\Big)^{-N-\delta}  \lesssim 1,
 \end{align*}
which completes the estimate for term $II$. In conjunction with the estimate for $I$ (as encapsulated in the estimates for $I_1$ and $I_2$) we arrive at \eqref{eq:AO2b}. 

Thus, we have obtained \eqref{eq:AO1a}-\eqref{eq:AO2b}, concluding the proof of Lemma \ref{lem: AO}.
\end{proof}

\subsection{Molecular decomposition}\label{sec:decomp}
The molecular decomposition of Besov and Triebel spaces follows from their frame decompositions. 
\begin{Theorem}[Molecular decomposition]\label{th:decomp}
Let $\alpha\in\RR$, $0<q\le\infty,$ and $0<p<\infty$ if $A^{p,q}_{\alpha}(\LL)=F^{p,q}_\alpha(\LL)$ or $0<p\le \infty$ if $A^{p,q}_{\alpha}(\LL)=B^{p,q}_\alpha(\LL)$. Let $\mu \ge 1$, $(M,\theta)\in \{\NN_0\times (0,1)\}\cup \{(-1,1)\},$  $N\in \NN_0$ and $0\le \delta\le 1$. Then there exists a family of $(M,\theta,N,\delta,\mu)$-molecules $\{m_R\}_{R\in \E}$ such that for any $f\in A^{p,q}_\alpha(\LL)$ there is a sequence of scalars $\{s_R\}_{R\in\E}$ satisfying 
\begin{align}\label{eq:decomp1}
f=\sum_R s_R m_R \quad \text{in } \sz'(\RR^n)
\end{align}
and 
\begin{align}
\label{eq:decomp2}
\Vert s\Vert_{a^{p,q}_\alpha} \lesssim \Vert f\Vert_{A^{p,q}_\alpha}.
\end{align}
\end{Theorem}
\begin{proof}[Proof of Theorem \ref{th:decomp}]
Let $\{\vp_j\}_{j\in \NN_0}$ and $\{\psi_j\}_{j\in\NN_0}$ be admissible systems satisfying \eqref{CRF 1}. By the frame decompositions  given in Theorem \ref{th: frame}, we have 
$$ f=\sum_{R\in \E}  \ip{f,\vp_R} \psi_R = \sum_{R\in \E} s_R m_R,$$
where $s_R=c^{-1}\ip{f,\vp_R}$ and $m_R=c\,\psi_R$. By Lemma \ref{lem: needlets},   $m_R$ is a $(M,\theta,N,\delta,\mu)$-molecule for some appropriate uniform constant $c$; this gives \eqref{eq:decomp1}. Secondly, by part (b) of Theorem \ref{th: frame}, we have
$$ \Vert s\Vert_{a^{p,q}_\alpha} =c^{-1}\Vert \{\ip{f,\vp_R}\}\Vert_{a^{p,q}_\alpha}=c^{-1} \Vert S_\vp f\Vert_{a^{p,q}_\alpha} \lesssim \Vert f\Vert_{A^{p,q}_\alpha},$$
which gives \eqref{eq:decomp2} and concludes our proof. 
\end{proof}

\subsection{Molecular synthesis}\label{sec:synth}

We next state and prove molecular synthesis estimates. Recall that the notation $n_{p,q}$ has been defined in Section \ref{sec: spaces}.

\begin{Theorem}[Molecular synthesis]\label{th:synth}
Let $\alpha\in\RR$, $0<q\le\infty,$ and $0<p<\infty$ if $A^{p,q}_{\alpha}(\LL)=F^{p,q}_\alpha(\LL)$ or $0<p\le \infty$ if $A^{p,q}_{\alpha}(\LL)=B^{p,q}_\alpha(\LL)$. Suppose $\{m_R\}_{R\in\E}$ is a collection of $(M,\theta, N, \delta, \mu)$-molecules satisfying
\begin{enumerate}[\upshape(i)]
\item $M\ge \max(\floor{n_{p,q}-n-\alpha},-1),$
\item $\theta>\left
\lbrace 
	\begin{array}{ll}
			\max\{n_{p,q}^*, (n_{p,q}-\alpha)^*\}  \qquad &\text{if}\quad n_{p,q}-n-\alpha\ge 0,\\
			0\qquad &\text{if}\quad n_{p,q}-n-\alpha<0,
	\end{array}
\right. $
\item $N\ge \max\{\floor{\alpha},0\},$
\item $\delta>\alpha^* $ if $\alpha\ge0 $,
and $\delta\ge 0 $ if $\alpha<0,$
\item $\mu>\max\{n_{p,q}, n+M+\theta\}.$
\end{enumerate}
Then for any sequence of numbers $s=\{s_R\}_{R\in\E}\in a^{p,q}_{\alpha}(\LL)$, we have
$$ \Big\Vert \sum_{R\in \E} s_R m_R \Big\Vert_{A^{p,q}_\alpha}\lesssim \Vert s\Vert_{a^{p,q}_{\alpha}}.$$
\end{Theorem}

\begin{Remark}\label{rem:synth}
\begin{enumerate}[(i)]
\item Note that our hypotheses on $M, \theta ,N$ and $\delta$ ensure that 
\begin{align}\label{eq: NM conditions}
N+\delta>\alpha \qquad\text{and}\qquad n+M+\theta+\alpha>n_{p,q}.
\end{align}
In fact, it can be seen from its proof that Theorem \ref{th:synth} holds if conditions (i)-(iv) are replaced by the weaker inequalities in \eqref{eq: NM conditions}. This fact will be used in the proofs of our results on Hermite pseudo-multipliers in Section \ref{sec: PDOs}. 
 \item The following are examples of minimal conditions on $(M,\theta,N,\delta,\mu)$ in Theorem \ref{th:synth} for some special cases of $\alpha, p$ and $q$. 
 
 \begin{enumerate}[1.]
 \item {$\alpha>0$, $\min\{p,q\}\ge1$:}  $n_{p,q}=n,$  $M=-1$, $\theta=1$, $N=\floor{\alpha}$, $\delta>\alpha^*$ and $\mu>n$. 
  \item {$\alpha=0$, $\min\{p,q\}\ge1$:} $n_{p,q}=n,$  $M=0$, $\theta\in (0,1)$, $N=0$, $\delta>0$ and $\mu>n+\theta$.
   \item {$\alpha<0$, $p,q>0$}:  $M=\floor{n_{p,q}-n-\alpha}$, $\theta>\max\{n_{p,q}^*, (n_{p,q}-\alpha)^*\}$, $N=0$, $\delta=0$, $\mu>\floor{n_{p,q}-\alpha}+\theta$. 
   \item {$\alpha>n_{p,q}-n$, $p,q>0$}: $M=-1$, $\theta=1$, $N=\floor{\alpha}$, $\delta>\alpha^*$, $\mu>n_{p,q}$.
   \item {$0\le \alpha\le n_{p,q}-n$, $p,q>0$}: $M=\floor{n_{p,q}-n-\alpha}$, $\theta>\max\{n_{p,q}^*, (n_{p,q}-\alpha)^*\}$, $N=\floor{\alpha}$, $\delta>\alpha^*$, $\mu>\floor{n_{p,q}}+\theta$. 
 \end{enumerate}
 
 \medskip
 
\noindent Examples of spaces corresponding to the cases described  include Sobolev type spaces for case 1, $L^p$ spaces for case 2, and Hardy type spaces for case 5.
\end{enumerate}
\end{Remark}

 The proof of Theorem \ref{th:synth} requires certain inequalities involving a maximal operator, which we next present. 
For each $s>0$ and for a locally integrable function $f$ on $\RR^n,$ we define
\begin{align}\label{mf}
	\MM_s f(x)=\sup_{x\in Q}\Big(\aver{Q} |f(y)|^s\,dy\Big)^{1/s}\quad \forall x\in \RR^n,
\end{align}
where the supremum is taken over all cubes $Q\subset\RR^n$ with sides parallel to the axes that contain $x.$ The reader may observe that  $\MM_s$ coincides with the usual Hardy--Littlewood maximal operator for $s=1$. The maximal operator $\MM_s$ satisfies the  following well-known inequality. For the case $s=1$ one may consult \cite{MR1232192}, from which the general case follows readily.
\begin{Lemma}[Fefferman--Stein inequality]\label{fefferman-stein B}
If $0<p<\infty,$ $0<q<\infty$ and $0<s<\min\{p,q\},$ it holds that
\begin{align*}
	\Big\Vert \Big(\sum_{j\in \NN}\big|\MM_s(f_j)\big|^q\Big)^{1/q}\Big\Vert_{L^p} \le C\Big\Vert\Big(\sum_{j\in\NN} |f_j|^q\Big)^{1/q}\Big\Vert_{L^p}
\end{align*}
for any sequence $\{f_j\}_{j\in \NN}$ of locally integrable functions defined on $\RR^n.$
\end{Lemma}
We next state and prove a  lemma involving sequences of numbers and the maximal operator that is an extension of  \cite[Lemma 4]{MR2399106}.
\begin{Lemma}\label{lem: maximal2}
Let $r>0,$  $\eta>\f{n}{\min\{1,r\}}$ and $j,k\in \NN_0.$ Given a sequence of numbers $\{a_R\}_{R\in\E_k},$  set
$$ a^*_k(x)=\sum_{R\in\E_k}\f{|a_R|}{(1+2^{j\land k}|x-x_R|)^\eta}\qquad \forall x\in\RR^n.$$
Then it holds that
$$ a^*_k(x)\lesssim 2^{\f{n}{1\land r} ((k-j)\vee 0)}\MM_r\Big(\sum_{R\in\E_k} |a_R|\Ind_R\Big)(x) \quad \forall x\in \RR^n,$$
where the implicit constant  is independent of $k$ and $j.$
\end{Lemma}
\begin{proof}

For $k\le j$ one may apply \cite[Lemma 4]{MR2399106} directly to obtain
$$a^*_k(x)\lesssim \MM_r\Big(\sum_{R\in\E_k}|a_R|\Ind_R\Big)(x).$$

For the case $k>j,$ we proceed as in the proof of \cite[Lemma 4]{MR2399106} with 
$$ \wt{a}_k(x)=\sum_{R\in\E_k}\f{|a_R|}{(1+2^j d(x,R))^\eta},$$
where $d(x,R) = \inf_{y\in R}\Vert x-y\Vert_{\ell^\infty};$ note that  
$a^*_k(x)\lesssim\wt{a}_k(x)$ for all $x\in\RR^n.$
Let  $c_3$ and $c_4$  be the constants from Lemma \ref{lem: tiles}. We consider two cases.

\bigskip

\underline{Case 1: $|x|_\infty > 2(c_3+c_4)2^k.$}

\medskip

For each $R\in\E_k$  we have $d(x,R)>|x|_\infty/2$ by part (c) of Lemma~\ref{lem: tiles}  and the assumption on $|x|_{\infty}.$ Set $\nu=1-\min(1,1/r)$. Recalling that $\eta>{n}/{\min\{1,r\}},$ and using  the fact that $\# \E_k \sim 4^{kn}$ along with H\"older's inequality if $r>1$ or the triangle inequality if $r\le 1$, we have
\begin{align*}
\wt{a}_k(x)
&\le \sum_{R\in\E_k} \f{|a_R|}{(1+2^jd(x,R))^{\f{n}{1\land r}}}\\
&\lesssim 2^{(k-j)\f{n}{1\land r}} \Big(\f{2^{-k}}{|x|_\infty}\Big)^{\f{n}{1\land r}} \sum_{R\in\E_k} |a_R| \\
&\lesssim 2^{(k-j)\f{n}{1\land r}} \Big(\f{2^{-k}}{|x|_\infty}\Big)^{\f{n}{1\land r}}4^{kn\nu} \Big(\sum_{R\in\E_k} |a_R|^r\Big)^{1/r}.
\end{align*}
Let $Q_x=Q(0,2|x|_\infty);$ then  $x\in Q_x$ and $\bigcup_{R\in\E_k}R\subseteq Q_x$ by part (c) of Lemma~\ref{lem: tiles} and the fact that $|x|_\infty>c_3 2^k$. Invoking  H\"older's inequality with $1/r$ if $r<1$  or the triangle inequality if $r\ge 1$, and using that $|R|\gtrsim 2^{-kn}$ for every tile $R\in\E_k$ by part (b) of Lemma~\ref{lem: tiles}, we obtain
\begin{align*}
\Big(\sum_{R\in\E_k} |a_R|^r\Big)^{1/r} 
&\le   |Q_x|^{1/r}\Big\{\aver{Q_x} \Big(\sum_{R\in\E_k} |a_R| |R|^{-1/r}\Ind_R(y)\Big)^r\,dy\Big\}^{1/r} \\
&\lesssim  (2^{k}|x|_\infty)^{n/r} \MM_r\Big(\sum_{R\in\E_k} |a_R| \Ind_R\Big)(x).
\end{align*}
Inserting this estimate into the previous calculation and applying the assumption $|x|_\infty\gtrsim 2^k$ we obtain
\begin{align*}
\wt{a}_k(x) 
&\lesssim  2^{(k-j)\f{n}{1\land r}}  2^{-2k(\f{n}{1\land r}-n\nu-\f{n}{r})}\MM_r\Big(\sum_{R\in\E_k} |a_R| \Ind_R\Big)(x).
\end{align*}
By considering $r\le 1$ and $r>1$ separately, we see that $2^{-2k(\f{n}{1\land r}-n\nu-\f{n}{r})}=1$, completing the proof of case 1. 

\bigskip

\underline{Case 2: $|x|_\infty \le 2(c_3+c_4)2^k.$}

\medskip

Let $\widehat{\E}_k$ be the collection of cubes defined in \eqref{hermite cubes}. For each $Q\in\widehat{\E}_k$ we set $a_Q=a_R$ whenever $Q\subset R$. We have
\begin{align}\label{maximal 1}
	\widetilde{a}_k(x)\le \sum_{Q\in\widehat{\E}_k} \f{|a_Q|}{(1+2^jd(x,Q))^\eta}
	\quad \text{and}\quad
	\sum_{R\in\E_k}|a_R|\Ind_R = \sum_{Q\in\widehat{\E}_k}|a_Q|\Ind_Q.
\end{align}
For $m\ge 1,$ define
\begin{align*}
\AM_0&=\AM_0(x,k,j)=\{Q\in\widehat{\E}_k: |x-x_Q|_\infty\le c_4 2^{-j}\}, \\
\AM_m &= \AM_m(x,k,j)=\{Q\in\widehat{\E}_k: c_4 2^{m-j-1}<|x-x_Q|_\infty\le  c_4 2^{m-j}\}.
\end{align*}
For $m\ge 0,$ set
\begin{align*}
\B_m&=\B_m(x,j)= Q(x,c_42^{m+1-j}).
\end{align*}
Note that these sets satisfy the following properties:
\begin{align}\label{maximal 2}
 \# \AM_m \lesssim 2^{(m-j+k)n}, 
 &&
 \widehat{\E}_k = \bigcup\limits_{m\ge 0} \AM_m,
 &&
 \bigcup_{Q\in\AM_m} Q\subseteq \B_m.
 \end{align}
 The first inequality in \eqref{maximal 2} holds because
 $$\#\AM_m \sim \f{|\cup_{Q\in \AM_m} Q|}{2^{-kn}}\le \f{|\B_m|}{2^{-kn}}\sim \f{2^{(m-j)n}}{2^{-kn}},$$
 where we have used that the cubes in $\AM_m$ are disjoint and have measure comparable to $2^{-kn}.$
Using \eqref{maximal 1}, the fact that $d(x,Q)\sim 2^{m-j}$ whenever $m\ge 2$ (recall that $j<k$ and part (d) of Lemma~\ref{lem: tiles}), and either the $r$-H\"older inequality  along with the first property in \eqref{maximal 2} if $r>1$ or the triangle inequality otherwise, we have
\begin{align*}
\wt{a}_k(x)\
\le \sum_{m\ge 0}\sum_{Q\in\AM_m} \f{|a_Q|}{(1+2^jd(x,Q))^\eta} 
\lesssim \sum_{m\ge 0} 2^{-m\eta+(m-j+k)n\nu} \Big(\sum_{Q\in\AM_m}|a_Q|^r\Big)^{1/r}.
\end{align*}
From the last property in \eqref{maximal 2} and the $1/r$-H\"older or triangle inequality as appropriate,  for each $m\ge 0,$ we have 
\begin{align*}
 \Big(\sum_{Q\in\AM_m}|a_Q|^r\Big)^{1/r}
 &\le \Big(\int_{\B_m} \sum_{Q\in\AM_m} |a_Q|^r |Q|^{-1} \Ind_Q(y)\,dy\Big)^{1/r} \\
 &\le |\B_m|^{1/r}\Big(\aver{\B_m} \Big(\sum_{Q\in\AM_m} |a_Q| |Q|^{-1/r} \Ind_Q(y)\,dy\Big)^r\Big)^{1/r} \\
 &\lesssim 2^{(m+k-j)\f{n}{r}} \MM_r\Big(\sum_{R\in\E_k} |a_R|\Ind_R\Big)(x),
\end{align*}
where in the last inequality we applied the estimates $|\B_m|\sim 2^{(m-j)n}$ and $|Q|\sim 2^{-kn}$, the fact that $\B_m$ contains $x$ and \eqref{maximal 1}. Combining the previous two calculations gives
\begin{align*}
\wt{a}_k(x)
&\lesssim 2^{(k-j)\f{n}{1\land r}}\MM_r\Big(\sum_{R\in\E_k} |a_R|\Ind_R\Big)(x) \sum_{m\ge 0}2^{-m(\eta-n\nu-\f{n}{r})}.
\end{align*}
Since the assumption $\eta>\f{n}{1\land r}$ ensures that the sum is finite, the  proof of the Lemma is finished. 
\end{proof}

We turn  to the proof of Theorem \ref{th:synth}. In the rest of this section, for a given sequence of numbers $\{s_R\}_{R\in\E}$, $\alpha\in\RR$ and $x\in\RR^n$ we set
$$ s_k(\alpha,x) =2^{k\alpha}\sum_{R\in\E_k}|s_R||R|^{-1/2}\Ind_R(x).$$
\begin{proof}[Proof of Theorem \ref{th:synth}] We separate the proof in two cases, one for the Triebel--Lizorkin spaces and another for the Besov spaces.

\bigskip

\underline{Case $(A^{p,q}_\alpha(\LL),a^{p,q}_\alpha(\LL))=(F^{p,q}_\alpha(\LL),f^{p,q}_\alpha(\LL))$.}

\medskip

Let $\eta>0$ and $0<r<1$ be such that
\begin{align}\label{eq:synth1}
\min\{n+M+\theta+\alpha, \mu\}>\eta > \f{n}{r} >n_{p,q}.
\end{align}
This is possible because of our hypotheses on $M, \theta,$ $\mu$ and $\alpha.$ 

Set $n_r=n/r$. Using the hypothesis on $\mu,$  \eqref{eq:synth1} and Lemma \ref{lem: AO}  in the second inequality, and  using   \eqref{eq:synth1} and Lemma \ref{lem: maximal2}  in the third inequality, we obtain
\begin{align*}
&\Big\Vert \sum_R s_R m_R \Big\Vert_{F^{p,q}_\alpha} \\
&\qquad\le \Big\Vert \Big\{\sum_{j\ge 0}\Big(2^{j\alpha}\sum_{k\ge 0}\sum_{R\in\E_k}|s_R|\, |\vp_j(\sqrt{\LL}) m_R |\Big)^q\Big\}^{1/q}\Big\Vert_{L^p} \\
&\qquad\le \Big\Vert \Big\{\sum_{j\ge 0}\Big(2^{j\alpha}\sum_{k\ge 0}\sum_{R\in\E_k}|s_R|\, |R|^{-1/2} \f{2^{-(n+M+\theta)[(k-j)\vee 0]-(N+\delta)[(j-k)\vee 0]}}{(1+2^{j\land k}|\cdot-x_R|)^{\eta}}  \Big)^q\Big\}^{1/q}\Big\Vert_{L^p}  \\
&\qquad\lesssim \Big\Vert \Big\{\sum_{j\ge 0}\Big(2^{j\alpha}\sum_{k\ge 0}2^{-(n+M+\theta-n_r)[(k-j)\vee 0]-(N+\delta)[(j-k)\vee 0]} \MM_r\Big(\sum_{R\in\E_k} |s_R|\, |R|^{-1/2} \Ind_R\Big) \Big)^q\Big\}^{1/q}\Big\Vert_{L^p} \\
&\qquad= \Big\Vert \Big\{\sum_{j\ge 0}\Big(\sum_{k\ge 0}2^{\alpha(j-k)+(n+M+\theta-n_r)[(j-k)\land 0]-(N+\delta)[(j-k)\vee 0]} \MM_r\big(s_k(\alpha,\cdot)\big) \Big)^q\Big\}^{1/q}\Big\Vert_{L^p}\\
&\qquad=\Big\Vert \Big\{\sum_{j\ge 0}\Big(\sum_{k\ge 0}a_{j-k}b_k(\cdot) \Big)^q\Big\}^{1/q}\Big\Vert_{L^p},
\end{align*}
where, for $j\in\ZZ$, 
\begin{align*}
a_j=2^{j\alpha+(j\land 0)(n+M+\theta-n_r)-(j\vee 0)(N+\delta)}\quad \quad\text{and}\quad \quad
b_j(x)=\MM_r\big(s_j(\alpha,\cdot)\big)(x) \Ind_{j\ge 0}(j).
\end{align*}
Define $a=\{a_j\}_{j\in \ZZ}$ and $b(x)=\{b_j(x)\}_{j\in \ZZ};$
note that for any $t>0,$ it holds that 
\begin{align}\label{eq:asum}
\Vert a\Vert_{\ell^t}<\infty.
\end{align}
  Indeed, we have
\begin{align*}
	\Vert a\Vert_{\ell^t}^t
	= \sum_{j\in\ZZ} 2^{[j\alpha+(j\land 0)(n+M+\theta-n_r)-(j\vee 0)(N+\delta)]t} 
	= \sum_{j\ge 0} 2^{-j(N+\delta-\alpha)t}+ \sum_{j<0}2^{j(n+M+\theta-n_r+\alpha)t}.
\end{align*}
Then first sum converges because $N+\delta>\alpha$ by our hypotheses on $N$ and $\delta$, and the second sum converges because of  \eqref{eq:synth1}.

We next use estimate \eqref{eq:asum} with $t=1\land q$ and Young's inequality with exponent $q$ if $q\ge 1$ or the $q$-triangle inequality with Young's inequality with exponent 1 if $q<1$. This gives
\begin{align*}
\Big\Vert \sum_{R\in\E} s_R m_R \Big\Vert_{F^{p,q}_\alpha}
\lesssim \big\Vert \Vert a\Vert_{\ell^{1\land q}}\Vert b(\cdot)\Vert_{\ell^q}\big\Vert_{L^p} 
\lesssim \big\Vert\Vert b(\cdot)\Vert_{\ell^q}\big\Vert_{L^p}  
= \Big\Vert \Big(\sum_{j\ge 0} \big(\MM_r\big(s_j(\alpha,\cdot)\big)\big)^q\Big)^{1/q} \Big\Vert_{L^p}.
\end{align*}
Inequality \eqref{eq:synth1} and Lemma \ref{fefferman-stein B} lead to
\begin{align*}
\Big\Vert \sum_{R\in\E} s_R m_R \Big\Vert_{F^{p,q}_\alpha}
\lesssim \Big\Vert \Big(\sum_{j\ge 0} s_j(\alpha,\cdot)^q\Big)^{1/q} \Big\Vert_{L^p}  
= \Vert s\Vert_{f^{p,q}_\alpha}.
\end{align*}
This concludes the proof  for the Triebel--Lizorkin spaces. 

\bigskip

\underline{Case $(A^{p,q}_\alpha(\LL),a^{p,q}_\alpha(\LL))=(B^{p,q}_\alpha(\LL),b^{p,q}_\alpha(\LL))$.}

\medskip

Let $\eta>0$ and $0<r<1$ be such that
\begin{align}\label{eq:synth2}
\min\{n+M+\theta+\alpha,\mu\} > \eta> \f{n}{r}>n_{p,q},
\end{align}
which is possible because of our hypotheses on $M, \theta$ and $\mu$. 

As in the previous case we set $n_r=n/r$. Using the hypothesis on $\mu,$ \eqref{eq:synth2} and  Lemma \ref{lem: AO} in the first inequality, and using \eqref{eq:synth2} and Lemma \ref{lem: maximal2} in the second inequality, leads to
\begin{align*}
&\Big\Vert \sum_R s_R m_R \Big\Vert_{B^{p,q}_\alpha}\\
&\lesssim \Big\{\sum_{j\ge 0} \Big(2^{j\alpha} \Big\Vert \sum_{k\ge 0}\sum_{R\in\E_k}|s_R|\, |R|^{-1/2} \f{2^{-(n+M+\theta)[(k-j)\vee 0]-(N+\delta)[(j-k)\vee 0]}}{(1+2^{j\land k}|\cdot-x_R|)^{\eta}} \Big\Vert_{L^p}\Big)^q\Big\}^{1/q}\\
&\lesssim \Big\{\sum_{j\ge 0}\Big(2^{j\alpha} \Big\Vert\sum_{k\ge 0}2^{-(n+M+\theta-n_r)[(k-j)\vee 0]-(N+\delta)[(j-k)\vee 0]} \MM_r\Big(\sum_{R\in\E_k} |s_R|\, |R|^{-1/2} \Ind_R\Big) \Big\Vert_{L^p}\Big)^q\Big\}^{1/q}\\
&=  \Big\{\sum_{j\ge 0}\Big\Vert\sum_{k\ge 0}2^{\alpha(j-k)+(n+M+\theta-n_r)[(j-k)\land 0]-(N+\delta)[(j-k)\vee 0]} \MM_r\big(s_k(\alpha,\cdot)\big) \Big\Vert_{L^p}^q\Big\}^{1/q}\\
&\le  \Big\{\sum_{j\ge 0}\Big(\sum_{k\ge 0}\Big[ 2^{\alpha(j-k)+(n+M+\theta-n_r)[(j-k)\land 0]-(N+\delta)[(j-k)\vee 0]} \big\Vert\MM_r\big(s_k(\alpha,\cdot)\big) \big\Vert_{L^p}\Big]^{1\land p}\Big)^{\f{q}{1\land p}}\Big\}^{1/q}\\
&= \Big\{\sum_{j\ge 0}\Big(\sum_{k\ge 0}| a_{j-k}b_k|^{1\land p}\Big)^{\f{q}{1\land p}}\Big\}^{1/q},
\end{align*}
where
\begin{align*}
a_j=2^{j\alpha+(j\land 0)(n+M+\theta-n_r)-(j\vee 0)(N+\delta)}\quad \quad\text{and}\quad\quad
b_j =\big\Vert\MM_r\big(s_j(\alpha,\cdot)\big)\big\Vert_{L^p} \Ind_{j\ge 0}(j)
\end{align*}
for $j\in\ZZ;$ in the next to the last line we applied  Minkowski's inequality for infinite sums if $p\ge 1$ or the $p$-triangle inequality if $p<1$.

Set $a=\{a_j\}_{j\in \ZZ}$ and $b=\{b_j\}_{j\in \ZZ};$  by  applying \eqref{eq:asum} with $t=1\land \f{q}{1\land p}$ and  $a^{1\land p}$ in place of $a,$ we have
\begin{align}\label{eq:asum2}
	\Vert a^{1\land p}\Vert_{\ell^{1\land \f{q}{1\land p}}} <\infty.
\end{align}

Applying  Young's inequality with exponent $\f{q}{1\land p}$ if $\f{q}{1\land p}\ge 1$ or the $\f{q}{1\land p}$-triangle inequality with Young's inequality with exponent 1 if $\f{q}{1\land p}<1$,  and using \eqref{eq:asum2}, we obtain
\begin{align*}
\Big\Vert \sum_R s_R m_R \Big\Vert_{B^{p,q}_\alpha}
\lesssim \big\Vert  a^{1\land p}\big\Vert_{\ell^{1\land \f{q}{1\land p}}}^{\f{1}{1\land p}} \big\Vert b^{1\land p}\big\Vert_{\ell^{\f{q}{1\land p}}}^{\f{1}{1\land p}}
&\lesssim \big\Vert b^{1\land p}\big\Vert_{\ell^{ \f{q}{1\land p}}}^{\f{1}{1\land p}}  
=\Big(\sum_{j\ge 0} \Big\Vert\MM_r\big(s_j(\alpha,\cdot)\big)\Big\Vert_{L^p}^q\Big)^{1/q}.
\end{align*}
Lemma \ref{fefferman-stein B} gives
\begin{align*}
\Big\Vert \sum_R s_R m_R \Big\Vert_{B^{p,q}_\alpha}
\lesssim \Big(\sum_{j\ge 0} \big\Vert s_j(\alpha,\cdot)\big\Vert_{L^p}^q\Big)^{1/q}
= \Big\{\sum_{j\ge 0}\Big(2^{j\alpha}\Big\Vert \sum_{R\in\E_j}|s_R|\,|R|^{-1/2}\Ind_R\Big\Vert_{L^p}\Big)^q\Big\}^{1/q},
\end{align*}
and, since for each $j$ the tiles in $\E_j$ are disjoint, we obtain
\begin{align*}
\Big\Vert \sum_R s_R m_R \Big\Vert_{B^{p,q}_\alpha}
\lesssim  \Big\{\sum_{j\ge 0}2^{j\alpha q}\Big( \sum_{R\in\E_j}\big[|s_R|\,|R|^{-1/2}|R|^{1/p}\big]^p\Big)^{q/p}\Big\}^{1/q} 
=\Vert s\Vert_{b^{p,q}_\alpha}.
\end{align*}
This concludes the proof of the theorem for the Besov spaces.
\end{proof}

\section{Hermite pseudo-multipliers}\label{sec: PDOs}

In this section, we pursue the study of boundedness properties  in Hermite Besov and Hermite Triebel--Lizorkin spaces for pseudo-multipliers with symbols in H\"ormander-type classes adapted to the Hermite setting. In Section~\ref{sec: symbols and needlets}, we define the classes of symbols and study the action of the corresponding pseudo-multipliers on needlets
by proving smoothness  and  cancellation estimates  (Theorems~\ref{th: Tsmooth} and \ref{th: Tcanc}, respectively). In Section~\ref{sec: main thms}, we state the theorems regarding boundedness results and present their proofs: Theorem~\ref{th: main1} gives results for spaces with positive smoothness; by assuming extra cancellation conditions on the symbols or more regularity conditions on the symbols, Theorem~\ref{th: main2} gives results for spaces with zero or negative smoothness as well. The proofs of  Theorem~\ref{th: main1} and \ref{th: main2} use as tools the results of  Theorems~\ref{th: frame}, \ref{th:synth}, \ref{th: Tsmooth} and \ref{th: Tcanc}.

\subsection{Classes of symbols and the action of pseudo-multipliers on needlets}\label{sec: symbols and needlets}
Given a symbol $\sigma: \RR^n\times \NN_0\to\mathbb{C}$ we define the operator $T_\sigma$ by
\begin{align}\label{eq:T}
T_\sigma f(x)
=\sum_{\kk \in \NN_0}\sigma(x,\lambda_\kk) \PP_\kk f(x) 
=\sum_{\kk\in\NN_0}\sigma(x,\lambda_\kk) \sum_{|\xi|=\kk} \ip{f,h_\xi} h_\xi(x).
\end{align}

 We introduce the function
\begin{align}\label{eq:rho} \cro(x)=\f{1}{1+|x|},\qquad x\in\RR^n;\end{align}
observe that 
\begin{align}\label{eq:rhoproperty} 
\cro(y)\sim \cro(x) \qquad \forall \;y\in B(x,\cro(x)).
\end{align}
We call a  non-negative function $g:\RR^n\times\NN_0\to[0,\infty)$ an \emph{admissible  growth function} if for some $0\le \vk<1$ and $\ve>16$
\begin{align}\label{eq:growthfn1}
g(x,\xs) \lesssim \e_{\ve \xs}(x)^{-\vk} 
\end{align}
and
\begin{align}
\label{eq:growthfn2}
 g(x,\xs)\sim g(y,\xs) \qquad \forall \;y\in B(x,\cro(x)),
\end{align}
where $\e_N(x)$ was defined in \eqref{efn} for $N\ge 0.$

\begin{Definition}[Symbols with growth]\label{def: smooth}
Let $m\in \RR$, $\rr,\dd\ge 0$, and $\N, \K\in \NN_0\cup\{\infty\}$. We say the symbol $\sigma: \RR^n\times \NN_0\to\mathbb{C}$ satisfies $\sigma \in \SM^{m,\K,\N}_{\rr,\dd}$ if $\sigma(\cdot,\xs)\in C^{\N}(\RR^n)$ for all $\xs\in \NN_0$ and there exists an admissible growth function $g$ such that
\begin{align}\label{eq:sm}
	| \partial^\nu_x \diff^\kappa_\xs \sigma(x,\xs)| \lesssim g(x,\xs)(1+\sqrt{\xs})^{m-2\rr\kappa+\dd|\nu|}  \qquad \forall(x,\xs)\in \RR^n\times\NN_0
\end{align}
for $\nu\in\NN_0^n$ satisfying $0\le |\nu|\le \N$ and $0\le \kappa\le \K$.
If $\N=\infty$ (respectively $\K=\infty$) then we mean that \eqref{eq:sm} holds for every multi-index $\nu \in \NN_0^n$ (respectively every $\kappa\in\NN_0$) with the implicit constant depending on $\nu$ (respectively $\kappa$).
\end{Definition} 

We note that $\SM^{m,\K,\N}_{\rr',\dd}\subset  \SM^{m,\K,\N}_{\rr,\dd}$  if $\rr\le \rr'$ and $\SM^{m,\K,\N}_{\rr,\dd'}\subset  \SM^{m,\K,\N}_{\rr,\dd}$ if $\dd'\le \dd.$ In particular, all results stated below for $\SM^{m,\K,\N}_{1,1}$ hold true for $\SM^{m,\K,\N}_{1,\dd}$ with $0\le \dd\le 1.$ 

For situations that require some degree of cancellation or orthogonality, we introduce the following condition on the symbols. A related condition, sufficient for the boundedness of a pseudo-differential operator on $L^2(\RR^n)$, has been considered in \cite{MR887496}; see in particular   \cite[Corollary 2.2]{MR887496}.

\begin{Definition}[Cancellation class]\label{def: canc}
Let $m\in\RR$ and $M\in\NN_0\cup \{\infty\}$. We say the symbol $\sigma: \RR^n\times \NN_0\to\mathbb{C}$ belongs to $\CN^{m,M}$  if 
\begin{align}\label{eq:cn1}
\Big(\aver{B(x,\cro(x))}\big|\cro(y)^{|\gamma|} \partial_y^\gamma \sigma(y,\xs)\big|^2\,dy\Big)^{1/2}\lesssim (1+\sqrt{\xs})^m\qquad \forall (x,\xs)\in \RR^n\times\NN_0
\end{align}
for $\gamma\in \NN_0^n$ satisfying $0\le |\gamma| \le 2\floor{(n+M)/2}+2$ and where the implicit constant may depend on $\gamma.$
\end{Definition}
\noindent In view of  \eqref{eq:rhoproperty}, the condition \eqref{eq:cn1} is equivalent to 
\begin{align*}
\Big(\aver{B(x,\cro(x))}\big| \partial_y^\gamma \sigma(y,\xs)\big|^2\,dy\Big)^{1/2}\lesssim (1+\sqrt{\xs})^m\cro(x)^{-|\gamma|}\qquad \forall (x,\xs)\in \RR^n\times\NN_0.
\end{align*}

We next consider the action of $T_\sigma$ on needlets. 
\begin{Theorem}[Smoothness estimates for $T_\sigma\vp_R$]\label{th: Tsmooth}
Let $m\in\RR,$  $\N\in \NN_0$ and $\K\in\NN,$ and suppose that 
 $\sigma\in \SM^{m,\K,\N}_{1,1}$. 
Let $\{\vp_j\}_{j\in\NN_0}$ be an admissible system. Then there exists $0\le \vk<1$ and $\ve >4$ such that for each   $\gamma\in\NN_0^n$ satisfying  $0\le |\gamma|\le \N$ and $1\le N\le \K$, it holds that
\begin{align}\label{eq:Tsmooth}
|\partial_x^\gamma T_\sigma\vp_R(x)|
\lesssim \f{|R|^{-1/2}2^{j(m+|\gamma|)}}{(1+2^j|x-x_R|)^N} \e_{\ve 4^j}(x)^{1-\vk} \quad \forall j\in \NN_0, R\in \E_j,  x\in\RR^n.
\end{align}
\end{Theorem}
Before giving the proof of Theorem \ref{th: Tsmooth}, note that if $\{\vp_j\}_{j\in\NN_0}$ is an admissible system, by \eqref{eq:T} and the orthogonality of the Hermite functions $\{h_\xi\}_{\xi\in\NN_0^n},$ we have the expression
\begin{align}\label{eq:Tphi}
T_\sigma \vp_R(x) 
&=\tau_R^{1/2} \sum_{\kk\in I_j} \sigma(x,\lambda_\kk)\,\vp_j(\sqrt{\lambda_\kk})  \,\PP_\kk(x,x_R),
\end{align}
where the sets $I_j$ are defined in Section \ref{sec: spaces}. 

In the sequel,  the notation $\diff_\kk \sigma(x,\lambda_\kk)$ means that the finite difference is being applied to $\sigma(x,\lambda_\kk)$ as a function of $\kk;$ that is, $\diff_\kk \sigma(x,\lambda_\kk) = \sigma(x,\lambda_{\kk+1})-\sigma(x,\lambda_\kk)$.

\begin{proof}[Proof of Theorem~\ref{th: Tsmooth}] Let $j\in \NN_0$ and $R\in \E_j.$

We first consider the case $|\gamma|=0$. 

\medskip

\underline{Subcase 1 for $|\gamma|=0$}: $|x-x_R|\ge 2^{-j}$.

\medskip

 Let $N\in \NN$ be such that $1\le N\le\K$. We apply the identity \eqref{eq:identity A} to \eqref{eq:Tphi} to get, for $i=1,\dots,n,$
\begin{align*}
&2^N (x_i - x_{R,i})^N T_\sigma\vp_R(x) \\
&\qquad=\tau_R^{1/2} \sum_{\f{N}{2} \le \ell\le N} c_{\ell,N}\sum_{\kk \in I_j} \diff_\kk^\ell \big[\sigma(x,\lambda_\kk)\vp_j(\sqrt{\lambda_\kk})\big]\big(\A{x_R}_i-\A{x}_i\big)^{2\ell-N}\PP_\kk(x,x_R).
\end{align*}
Since $|\tau_R|\sim|R|$, we have
\begin{align}\label{eq:Tsmooth1}
&\big| (x_i-x_{R,i})^NT_\sigma\vp_R(x)\big| \\
&\qquad \lesssim |R|^{1/2}\sum_{\f{N}{2} \le \ell\le N} |c_{\ell,N}|\sum_{\kk \in I_j} \Big|\diff_\kk^\ell \big[\sigma(x,\lambda_\kk)\vp_j(\sqrt{\lambda_\kk})\big]\Big| \,\Big|\big(\A{x_R}_i-\A{x}_i\big)^{2\ell-N}\PP_\kk(x,x_R)\Big| \notag.
\end{align}
We next estimate each factor in the summation over $\kk$. Firstly, from the Leibniz formula for finite differences \eqref{eq:leibniz1} we have
\begin{align*}
\Big|\diff_\kk^\ell \big[\sigma(x,\lambda_\kk)\vp_j(\sqrt{\lambda_\kk})\big]\Big| 
\le \sum_{r=0}^\ell \tbinom{\ell}{r} \big|\diff_\kk^r(\vp_j(\sqrt{\lambda_\kk}))\big| \, \big|\diff_\kk^{\ell-r}\sigma(x,\lambda_{\kk+r})\big|.
\end{align*}
Lemma~\ref{lem: hoppe} gives
\begin{align}\label{eq:phiMVT}
\big|\diff_\kk^r(\vp_j(\sqrt{\lambda_\kk}))\big| \lesssim \lambda_\kk^{N/2-r}2^{-jN}.
\end{align}
 By the assumption on $\sigma,$ there exists an admissible growth function $g,$ $0\le \vk<1$ and $\ve>4$ such that
\begin{align*}
\big|\diff_\kk^{\ell-r}\sigma(x,\lambda_{\kk+r})\big| 
\lesssim \lambda_\kk^{m/2-\ell+r} g(x,\kk)
\lesssim \lambda_\kk^{m/2-\ell+r}\e_{4\ve\kk}(x)^{-\vk}.
\end{align*} 
These last three facts give
\begin{align}\label{eq:Tsmooth2}
\big|\diff_\kk^\ell\big[\sigma(x,\lambda_{\kk}) \vp_j(\sqrt{\lambda_\kk})\big]\big|\lesssim 2^{-jN}  \lambda_\kk ^{N/2+m/2-\ell}\e_{4\ve \kk}(x)^{-\vk}.
\end{align}
It may be worth observing at this point that the implicit constant in \eqref{eq:Tsmooth2} depends on $N$, $\Vert \vph^{(N)}\Vert_{L^\infty}$ and $\Vert \vph_0^{(N)}\Vert_{L^\infty}$.
Secondly, since $0\le \ell-N/2\le N/2$, then observe that 
\begin{align*}
\big(2(\kk +2\ell-N)+2\big)^{\ell-N/2}
\le (2\kk +2N+2)^{\ell-N/2}
\le(2N+2)^{N/2} \lambda_\kk^{\ell-N/2}.
\end{align*}
With this observation in mind, the binomial theorem and an application of \eqref{eq:identity d1} gives
\begin{align*}
\Big|\big(\A{x_R}_i-\A{x}_i\big)^{2\ell-N}\PP_\kk(x,x_R)\Big| 
&\le \sum_{|\xi|=\kk} \sum_{l=0}^{2\ell-N} \tbinom{2\ell-N}{l} \Big|\big(\A{x}_i\big)^{2\ell-N-l} h_\xi(x)\Big| \,\Big|\big(\A{x_R}_i\big)^l h_\xi(x_R)\Big|\\
&\lesssim  \lambda_\kk^{\ell-N/2} \sum_{|\xi|=\kk} \sum_{l=0}^{2\ell-N} \tbinom{2\ell-N}{l} \big| h_{\xi+(2\ell-N-l)e_i}(x)\big|\,\big| h_{\xi+le_i}(x_R)\big|,
\end{align*}
with an implicit  constant that depends on $N$.
Applying  the Cauchy--Schwarz inequality we obtain
\begin{align}\label{eq:Tsmooth3}
\Big|\big(\A{x_R}_i-\A{x}_i\big)^{2\ell-N}\PP_\kk(x,x_R)\Big| 
\lesssim \lambda_\kk^{\ell-N/2} \Big(\sum_{l=0}^{2\ell-N} \PP_{\kk+2\ell-N-l}(x,x)\Big)^{1/2}\Big(\sum_{l=0}^{2\ell-N} \PP_{\kk+l}(x_R,x_R)\Big)^{1/2}.
\end{align}
Recall that $\lambda_\kk^{m/2}\sim 2^{jm}$. Then inserting the estimates \eqref{eq:Tsmooth2} and \eqref{eq:Tsmooth3} into \eqref{eq:Tsmooth1} and using the Cauchy--Schwarz inequality, we get
\begin{align*}
	\big| (x_i-x_{R,i})^NT_\sigma\vp_R(x)\big| 
	&\lesssim |R|^{1/2} 2^{j(m-N)}\e_{\ve 4^j}(x)^{-\vk}\\
	&\quad\times \sum_{\f{N}{2} \le \ell\le N}  \sum_{\kk \in I_j} \Big(\sum_{l=0}^{2\ell-N} \PP_{\kk+2\ell-N-l}(x,x)\Big)^{1/2}\Big(\sum_{l=0}^{2\ell-N} \PP_{\kk+l}(x_R,x_R)\Big)^{1/2} \\
	&\lesssim |R|^{1/2} 2^{j(m-N)}\e_{\ve 4^j}(x)^{-\vk} \,\QQ_{4^j+N}(x,x)^{1/2} \QQ_{4^j+N}(x_R,x_R)^{1/2}.
\end{align*}
Since this estimate holds for  $i=1,\dots,n,$ it follows that
\begin{align*}
|x-x_R|^N |T_\sigma\vp_R(x)| \lesssim |R|^{1/2} 2^{j(m-N)}\e_{\ve 4^j}(x)^{-\vk} \QQ_{4^j+N}(x,x)^{1/2} \QQ_{4^j+N}(x_R,x_R)^{1/2} 
\end{align*}
with constants depending on $N, m ,\vp, \sigma$. 

\bigskip

\underline{Subcase 2 for $|\gamma|=0$}: $|x-x_R|<2^{-j}$. 

\medskip

By the assumption on $\sigma$ and the Cauchy--Schwarz inequality, we get
\begin{align*}
|T_\sigma\vp_R(x)|
&\le \big|\tau_R^{1/2}\big| \sum_{\kk\in I_j} |\sigma(x,\lambda_\kk)| \,|\vp_j(\sqrt{\lambda_\kk})|\,|\PP_\kk(x,x_R)| \\
&\lesssim |R|^{1/2} 2^{jm}\e_{\ve 4^j}(x)^{-\vk} \sum_{\kk\in I_j}\sum_{|\xi|=\kk} |h_\xi(x)| \,|h_\xi(x_R)|\\
&\lesssim |R|^{1/2} 2^{jm}\e_{\ve 4^j}(x)^{-\vk} \QQ_{4^j+N}(x,x)^{1/2} \QQ_{4^j+N}(x_R,x_R)^{1/2} 
\end{align*}

Combining the estimates for both subcases  along with \eqref{QQ est}  and \eqref{eq:tile control} we have
\begin{align*}
	|T_\sigma\vp_R(x)|
	&\lesssim
	\f{|R|^{1/2} 2^{jm}\e_{\ve 4^j}(x)^{-\vk}}{(1+2^j|x-x_R|)^N} \QQ_{4^j+N}(x,x)^{1/2} \QQ_{4^j+N}(x_R,x_R)^{1/2} \\
	&\lesssim |R|^{1/2} 2^{jn} \e_{\ve 4^j}(x_R) \f{2^{jm}}{(1+2^j|x-x_R|)^N} \e_{\ve 4^j}(x)^{1-\vk} \\
	&\lesssim |R|^{-1/2}\f{2^{jm}}{(1+2^j|x-x_R|)^N}\e_{\ve 4^j}(x)^{1-\vk},
\end{align*}
with constants independent of $j\in\NN_0,$ $R\in \E_j$ and $x\in \RR^n.$

\bigskip

We turn to the case $|\gamma|>0$. Note first that we can represent $\partial^\gamma$ by
$$\partial^\gamma =\sum_{\alpha+\beta\le\gamma} C_{\alpha,\beta} A^\alpha x^\beta;$$
see \cite[(6.15)]{MR2399106}. This means that in order to prove \eqref{eq:Tsmooth} it suffices to show
\begin{align}\label{eq:Tsmooth4}
|A^\alpha x^\beta T_\sigma\vp_R(x)|
\lesssim \f{|R|^{-1/2}2^{j(m+|\alpha| +|\beta|)}}{(1+2^j|x-x_R|)^N} \e_{\ve 4^j}(x)^{1-\vk}
\end{align}
for any $\alpha,\beta\in \NN_0^n$ such that $0\le |\alpha| + |\beta|\le\N$ and $1\le N\le\K$. 

\bigskip

\underline{Subcase 1 for $|\gamma|>0$}: $|x-x_R|\ge 2^{-j}$.

\medskip

 We first prove bounds for each component $i=1,\dots,n$ by expressing the operator $A^\alpha x^\beta$ in terms of two commuting operators:
\begin{align*}
\big(\A{x}\big)^\alpha x^\beta =\big(\A{x}\big)^{\alpha-\alpha_ie_i}x^{\beta-\beta_ie_i}\big(\A{x}_i\big)^{\alpha_i}x_i^{\beta_i},
\end{align*}
where $\{e_i\}_{1\le i\le n}$ is the canonical basis for $\RR^n$.
Using identity  \eqref{eq:identity b2}, we have
\begin{align}\label{eq:Tsmooth5}
&(x_i-x_{R,i})^N \big(\A{x}\big)^\alpha x^\beta T_\sigma \vp_R(x) \\
&\qquad= \sum_{s=0}^{\alpha_i}\tbinom{\alpha_i}{s}\tfrac{N!}{(N-s)!}\big(\A{x}\big)^{\alpha-se_i}x^\beta (x_i-x_{R,i})^{N-s}T_\sigma\vp_R(x)\notag.
\end{align}
By \eqref{eq:identity A} and the Leibniz rule for finite differences \eqref{eq:leibniz1},
\begin{align*}
 &(x_i-x_{R,i})^{N-s}T_\sigma\vp_R(x)\\
 &=\tau_R^{1/2}2^{-(N-s)} \sum_{\f{N-s}{2}\le\ell\le N-s} c_{\ell, N-s} \\
 &\qquad\qquad\times \sum_{\kk\in I_j} \sum_{r=0}^\ell \tbinom{\ell}{r}\diff_\kk^r (\vp_j(\sqrt{\lambda_\kk})) \,\diff_\kk^{\ell-r}\sigma(x,\lambda_{\kk+r}) \big(\A{x_R}_i-\A{x}_i\big)^{2\ell-N+s}\PP_\kk(x,x_R).
\end{align*}
Applying $\big(\A{x}\big)^{\alpha-se_i}x^\beta$ to this expression and using \eqref{eq:leibniz2}, we get
\begin{align}\label{eq:Tsmooth6}
&\big(\A{x}\big)^{\alpha-se_i}x^\beta (x_i-x_{R,i})^{N-s}T_\sigma\vp_R(x) \\
&\qquad=\tau_R^{1/2}2^{-(N-s)} \sum_{\f{N-s}{2}\le\ell\le N-s} c_{\ell, N-s} \sum_{\kk\in I_j} \sum_{r=0}^\ell \tbinom{\ell}{r}\diff_\kk^r (\vp_j(\sqrt{\lambda_\kk})) \notag\\
&\qquad\qquad\times\sum_{\nu\le\alpha-se_i} \tbinom{\alpha-se_i}{\nu} (-1)^\nu \partial_x^\nu \diff_\kk^{\ell-r} \sigma(x,\lambda_{\kk+r}) \big(\A{x}\big)^{\alpha-se_i -\nu} \big[x^\beta \big(\A{x_R}_i-\A{x}_i\big)^{2\ell-N+s}\PP_\kk(x,x_R)\big].\notag
\end{align}
Writing $x^\beta=x^{\beta-\beta e_i}x_i^{\beta_i},$ applying \eqref{eq:identity b1} to $x_i^{\beta_i}\big(\A{x_R}_i-\A{x}_i\big)^{2\ell-N+s}$, and commuting $x^{\beta-\beta_i e_i}$ with powers of $\A{x_R}_i-\A{x}_i,$ we obtain
\begin{align*}
&x^\beta \big(\A{x_R}_i-\A{x}_i\big)^{2\ell-N+s}\PP_\kk(x,x_R)\\
&\qquad=x^{\beta-\beta e_i} \sum_{t=0}^{\beta_i} \tbinom{\beta_i}{t}\tfrac{(2\ell-N+s)!}{(2\ell-N+s-t)!} (-1)^t \big(\A{x_R}_i-\A{x}_i\big)^{2\ell-N+s-t}x_i^{\beta_i-t}\PP_\kk(x,x_R) \\
&\qquad= \sum_{t=0}^{\beta_i} \tbinom{\beta_i}{t}\tfrac{(2\ell-N+s)!}{(2\ell-N+s-t)!} (-1)^t \big(\A{x_R}_i-\A{x}_i\big)^{2\ell-N+s-t}x^{\beta-te_i}\PP_\kk(x,x_R).
\end{align*}
Applying $\big(\A{x}\big)^{\alpha-s e_i-\nu}$ to this expression, expanding  powers of $\big(\A{x_R}_i-\A{x}_i\big)$ by the binomial theorem, and then absorbing $x^{\beta-t e_i}$ into $\PP_\kk(x,x_R)$ via \eqref{eq:identity C}, we get
\begin{align*}
& \big(\A{x}\big)^{\alpha-se_i -\nu} \big[x^\beta \big(\A{x_R}_i-\A{x}_i\big)^{2\ell-N+s}\PP_\kk(x,x_R)\big] \\
&\qquad= \sum_{t=0}^{\beta_i} \tbinom{\beta_i}{t}\tfrac{(2\ell-N+s)!}{(2\ell-N+s-t)!} (-1)^t \sum_{l=0}^{2\ell-N+s-t} (-1)^l \tbinom{2\ell-N+s-t}{l} \\
&\qquad\qquad\times\big(\A{x_R}_i\big)^{2\ell-N+s-t-l} \big(\A{x}\big)^{\alpha-s e_i-\nu+l e_i}\\
&\qquad\qquad\qquad\times\sum_{|\xi|=\kk}\sum_{\omega\le\beta-t e_i} b_{\omega,\beta-t e_i}(\xi) h_{\xi+\beta-te_i -2\omega}(x) h_\xi (x_R).
\end{align*}
Inserting this last expression into \eqref{eq:Tsmooth6}, and its result into \eqref{eq:Tsmooth5}, we arrive at 
\begin{align}\label{eq:Tsmooth7}
&(x_i-x_{R,i})^N \big(\A{x}\big)^\alpha x^\beta T_\sigma \vp_R(x) \\
&\qquad=\tau_R^{1/2} 
\sum_{s=0}^{\alpha_i} \sum_{\kk\in I_j} \sum_{\f{N-s}{2}\le\ell\le N-s} \sum_{r=0}^\ell \sum_{\nu\le\alpha-se_i} \sum_{t=0}^{\beta_i}\sum_{l=0}^{2\ell-N+s-t} C' \notag \\
&\qquad\qquad\times 
 \diff_\kk^r(\vp_j(\sqrt{\lambda_\kk})) \,\partial_x^\nu \diff_\kk^{\ell-r} \sigma(x,\lambda_{\kk+r})  
\sum_{|\xi|=\kk}\sum_{\omega\le\beta-t e_i} b_{\omega,\beta-t e_i}(\xi) \notag \\
&\qquad\qquad\qquad\times\big(\A{x_R}_i\big)^{2\ell-N+s-t-l} \big(\A{x}\big)^{\alpha-s e_i-\nu+l e_i} h_{\xi+\beta-te_i -2\omega}(x) h_\xi (x_R),\notag
\end{align}
where 
\begin{align*}
C' = (-1)^{t+l+\nu}2^{s-N} \tbinom{\alpha_i}{s}\tbinom{\ell}{r}\tbinom{\alpha-se_i}{\nu}\tbinom{\beta_i}{t}\tbinom{2\ell-N+s-t}{l} \tfrac{N!}{(N-s)!} \tfrac{(2\ell-N+s)!}{(2\ell-N+s-t)!} c_{\ell,N-s}.
\end{align*}
We now estimate the above expression. By \eqref{eq:identity d1} and \eqref{eq:identity d2}, we have
\begin{align*}
	&\big|\big(\A{x_R}_i\big)^{2\ell-N+s-t-l} \big(\A{x}\big)^{\alpha-s e_i-\nu+l e_i} h_{\xi+\beta-te_i -2\omega}(x) h_\xi (x_R)\big| \\
	&\lesssim \lambda_\kk^{\ell-N/2+|\alpha-\nu|/2-t/2} \big| h_{\xi+(2\ell-N+s-t-l)e_i}(x_R)\big|\,\big| h_{\xi+\alpha+\beta-\nu-2\omega+(l-t-s)e_i}(x)\big|,
\end{align*}
and by  part (c) of Lemma \ref{lem: identities}, we have
\begin{align*}
 |b_{\omega,\beta-te_i}(\xi)|\lesssim \lambda_\kk^{(|\beta|-t)/2}.
\end{align*}
Furthermore, our assumption on $\sigma$ gives
\begin{align*}
\big| \partial_x^\nu \diff_\kk^{\ell-r} \sigma(x,\lambda_{\kk+r})  
\big|
\lesssim \lambda_\kk^{m/2+|\nu|/2-\ell+r} g(x,\kk)
\lesssim \lambda_\kk^{m/2+|\nu|/2-\ell+r} \e_{\ve \kk}(x)^{-\vk}.
\end{align*} 
Inserting these three estimates along with \eqref{eq:phiMVT} into \eqref{eq:Tsmooth7}, and noting the binomial bounds $\binom{a}{b}\le 2^a$ and $\f{a!}{(a-b)!}\le a^b$, we obtain
 \begin{align*}
&\big|(x_i-x_{R,i})^N \big(\A{x}\big)^\alpha x^\beta T_\sigma \vp_R(x)\big| \\
&\qquad\lesssim 
2^{-jN} |R|^{1/2} \e_{\ve 4^j}(x)^{-\vk} \,2^{j(m+|\alpha|+|\beta|)} \\
&\qquad\qquad\times \sum_{\kk\in I_j} \sum_{s=0}^{\alpha_i} \sum_{\f{N-s}{2}\le\ell\le N-s} \sum_{r=0}^\ell \sum_{t=0}^{\beta_i}\sum_{\nu\le\alpha-se_i} \sum_{\omega\le\beta-t e_i}  
\sum_{|\xi|=\kk}  \\
&\qquad\qquad\qquad\times\sum_{l=0}^{2\ell-N+s-t}  \big| h_{\xi+(2\ell-N+s-t-l)e_i}(x_R)\big|\,\big| h_{\xi+\alpha+\beta-\nu-2\omega+(l-t-s)e_i}(x)\big|. 
\end{align*}
By two applications of the Cauchy--Schwarz inequality, it follows that
\begin{align*}
&\sum_{\kk\in I_j} \sum_{l=0}^{2\ell-N+s-t}  \sum_{|\xi|=\kk}
\big| h_{\xi+(2\ell-N+s-t-l)e_i}(x_R)\big|\,\big| h_{\xi+\alpha+\beta-\nu-2\omega+(l-t-s)e_i}(x)\big|  \\
&\qquad\le \sum_{\kk\in I_j} \sum_{l=0}^{2\ell-N+s-t}  \PP_{\kk+|\alpha-\nu| +|\beta-2\omega| +l-t-s}(x,x)^{\f{1}{2}} \PP_{\kk +2\ell-N+s-t-l}(x_R,x_R)^{\f{1}{2}} \\
&\qquad\lesssim \QQ_{4^j +N+|\alpha|+|\beta|}(x,x)^{\f{1}{2}} \QQ_{4^j+N}(x_R,x_R)^{\f{1}{2}}.
\end{align*}
We then conclude that
\begin{align*}
&\big|(x_i-x_{R,i})^N \big(\A{x}\big)^\alpha x^\beta T_\sigma \vp_R(x)\big| \\
&\qquad\lesssim 2^{-jN} |R|^{1/2} \e_{\ve 4^j}(x)^{-\vk} 2^{j(m+|\alpha|+|\beta|)} \QQ_{4^j +N+|\alpha|+|\beta|}(x,x)^{\f{1}{2}} \QQ_{4^j+N}(x_R,x_R)^{\f{1}{2}}
\end{align*}
with constants depending only on $N,\alpha,\beta,\vp,\sigma$. 

\bigskip

\underline{Subcase 2 for $|\gamma|>0$}: $|x-x_R|<2^{-j}$. 

\medskip

By \eqref{eq:leibniz2} and \eqref{eq:identity C}, we get
\begin{align*}
&\big(\A{x}\big)^\alpha x^\beta T_\sigma \vp_R(x)\\
&\qquad= \tau_R^{\f{1}{2}} \sum_{\kk\in I_j} \vp_j(\sqrt{\lambda_\kk}) \big(\A{x}\big)^\alpha\big[\sigma(x,\lambda_\kk)x^\beta\PP_\kk(x,x_R)\big]  \\
&\qquad= \tau_R^{\f{1}{2}} \sum_{\kk\in I_j} \vp_j(\sqrt{\lambda_\kk}) 
\sum_{\nu\le\alpha} \tbinom{\alpha}{\nu} (-1)^\nu\partial_x^\nu \sigma(x,\lambda_\kk) \big(\A{x}\big)^{\alpha-\nu}\big[x^\beta\PP_\kk(x,x_R)\big] \\
&\qquad= \tau_R^{\f{1}{2}} \sum_{\kk\in I_j} \vp_j(\sqrt{\lambda_\kk}) 
\sum_{\nu\le\alpha} \tbinom{\alpha}{\nu} (-1)^\nu\partial_x^\nu \sigma(x,\lambda_\kk) 
\sum_{\omega\le\beta}  \sum_{|\xi|=\kk} b_{\omega,\beta}(\xi)\big(\A{x}\big)^{\alpha-\nu} h_{\xi+\beta-2\omega}(x) \, h_\xi(x_R).
\end{align*}
From \eqref{eq:identity d2}, part (c) of Lemma \ref{lem: identities} and our assumption on $\sigma$,  we have the following three estimates:
\begin{gather*}
\big|\big(\A{x}\big)^{\alpha-\nu} h_{\xi+\beta-2\omega}(x)\big| \lesssim \lambda_\kk^{\f{|\alpha-\nu|}{2}} \big|h_{\xi+\beta-2\omega+\alpha-\nu}(x)\big|,  \\
|b_{\omega,\beta}(\xi)|\lesssim \lambda_\kk^{\f{|\beta|}{2}}, \\
\big| \partial_x^\nu \sigma(x,\lambda_\kk)\big| \lesssim \lambda_\kk^{\f{m}{2}+\f{|\nu|}{2}} \e_{4\ve \kk}(x)^{-\vk}.
\end{gather*}
Applying these estimates and making use of  the Cauchy--Schwarz  inequality,  we obtain
\begin{align*}
&\big| \big(\A{x}\big)^\alpha x^\beta T_\sigma \vp_R(x)\big| \\
&\qquad\lesssim |R|^{\f{1}{2}} \sum_{\kk\in I_j} \lambda_\kk^{\f{m+|\alpha|+|\beta|}{2}} \e_{4\ve \kk}(x)^{-\vk}  \sum_{\substack{\omega\le\beta\\ \nu\le\alpha}}\sum_{|\xi|=\kk}  |h_{\xi+\beta-2\omega+\alpha-\nu}(x)| |h_\xi(x_R)| \\
&\qquad\lesssim  |R|^{\f{1}{2}} \e_{\ve 4^j}(x)^{-\vk} 2^{j(m+|\alpha|+|\beta|)}  \QQ_{4^j +N+|\alpha|+|\beta|}(x,x)^{\f{1}{2}} \QQ_{4^j+N}(x_R,x_R)^{\f{1}{2}}.
\end{align*}

Finally, combining the estimates for both subcases lead to
\begin{align*}
&\big| \big(\A{x}\big)^\alpha x^\beta T_\sigma \vp_R(x)\big| \\
	&\qquad\lesssim
	\f{|R|^{\f{1}{2}} 2^{j(m+|\alpha|+|\beta|)}\e_{\ve 4^j}(x)^{-\vk}}{(1+2^j|x-x_R|)^N} \QQ_{4^j+N+|\alpha|+|\beta|}(x,x)^{\f{1}{2}} \QQ_{4^j+N}(x_R,x_R)^{\f{1}{2}}.
\end{align*}
By making use of \eqref{QQ est} and  \eqref{eq:tile control} we obtain \eqref{eq:Tsmooth4}, completing the proof of Theorem \ref{th: Tsmooth}.
\end{proof}

\begin{Theorem}[Cancellation estimates for $T_\sigma\vp_R$]\label{th: Tcanc}
Let $m\in\RR$, $M,\N, \K\in\NN_0$  and  $\K\ge n+M+1$.
Assume that $\sigma:\RR^n\times\NN_0\to\mathbb{C}$ satisfies one of the following conditions:
\begin{enumerate}[\upshape(a)]
\item $\sigma\in \SM^{m,\K,\N}_{1,1}\cap \CN^{m,M},$
\item $\sigma \in \SM^{m,\K,\N}_{1,\dd}$ for some $0\le \delta<1$ and  $\N\ge 2\ceil{\f{n+M+1}{2(1-\dd)}}$.  
\end{enumerate}
Let $\{\vp_j\}_{j\ge 0}$ be an admissible system and $0<\theta\le 1$. Then for  each  $\gamma\in \NN_0^n$ satisfying  $0\le |\gamma|\le M,$ it holds that 
\begin{align}\label{eq:Tcanc}
\Big|\int_{\RR^n}(x-x_R)^\gamma T_\sigma\vp_R(x)\,dx\Big|
\lesssim |R|^{-1/2}2^{j(m-n-|\gamma|)} \Big(\f{1+|x_R|}{2^j}\Big)^{M+\theta-|\gamma|}
\end{align}
for every $j\in\NN_0$ and each $R\in\E_j$.
\end{Theorem}
\begin{proof}[Proof of Theorem \ref{th: Tcanc}] Let $j\in \NN_0,$  $R\in \E_j$ and $\gamma\in \NN_0^n$ such that $|\gamma|\le M.$

We first prove the theorem assuming (a).
Set $B=B(x_R,\cro(x_R));$ fix a function $\chi_R\in C^\infty(\RR^n)$ supported in $B$ that satisfies $\chi_R=1$ on $\f{1}{2}B,$  $0\le\chi_R\le 1$ and
\begin{align*}
 \Vert \chi_R^{(\eta)}\Vert_\infty \lesssim \f{1}{\cro(x_R)^{|\eta|}}\qquad \forall\eta\in\NN^n_0.
 \end{align*}
 We split the integral into two terms:
 \begin{align*}
\int_{\RR^n}(x-x_R)^\gamma T_\sigma\vp_R(x)\,dx
 &= \int_{\RR^n}(1-\chi_R(x))(x-x_R)^\gamma T_\sigma \vp_R(x)\,dx \\
 &\qquad+\int_{\RR^n} \chi_R(x)(x-x_R)^\gamma T_\sigma \vp_R(x)\,dx \\
 &=:I + II.
 \end{align*}
 
 To estimate $I$ we use the bounds on $T_\sigma \vp_R$ in Theorem \ref{th: Tsmooth}, which hold because  $\sigma\in  \SM^{m,\K,\N}_{1,1}$. We have
 \begin{align*}
 |I| 
  &\lesssim |R|^{-1/2} 2^{j(m-|\gamma|)} \int_{\RR^n\backslash \frac{1}{2}B}  \f{(2^j|x-x_R|)^{|\gamma|}}{(1+2^j|x-x_R|)^{\K}}\,dx  \\
 &\le  |R|^{-1/2} 2^{j(m-|\gamma|)}\Big(\f{1+|x_R|}{2^j}\Big)^{M+\theta-|\gamma|}\int_{\RR^n }\f{1}{(1+2^j|x-x_R|)^{\K-M-\theta}}\,dx, 
 \end{align*}
 which yields \eqref{eq:Tcanc} since $\K\ge n+M+1>n+M+\theta$.
 
 For the second term, we apply  the Cauchy-Schwarz inequality and obtain
\begin{align*}
|II|
&=\Big| \tau_R^{1/2} \sum_{\kk\in I_j}\vp_j(\sqrt{\lambda_\kk}) \sum_{|\xi|=\kk}h_\xi(x_R)\int_{\RR^n}\chi_R(x) (x-x_R)^\gamma \sigma(x,\lambda_\kk) h_\xi(x)\,dx\Big|\\
&\lesssim \Vert\vp\Vert_{L^\infty} |R|^{1/2}\,\QQ_{4^j}(x_R,x_R)^{1/2}\Big(\sum_{\kk\in I_j}\sum_{|\xi|=\kk} \Big| \int_{\RR^n} \chi_R(x) (x-x_R)^\gamma \sigma(x,\lambda_\kk) h_\xi(x)\,dx\Big|^2\Big)^{1/2}.
\end{align*}
We next estimate  the second factor. For $|\xi|=\kk$ and $N \in \NN_0$, we have
\begin{align*}
&\Big| \int_{\RR^n} \chi_R(x) (x-x_R)^\gamma \sigma(x,\lambda_\kk) h_\xi(x)\,dx\Big| \\
&\qquad=\lambda^{-N}_{\kk} \Big|\int_{\RR^n} \LL_x^N\big [ \chi_R(x) (x-x_R)^\gamma \sigma(x,\lambda_\kk)\big] h_\xi(x)\,dx\Big| \\
&\qquad\le \lambda^{-N}_{\kk} \Big\Vert \LL^N \big[\chi_R(\cdot) (\cdot-x_R)^\gamma \sigma(\cdot,\lambda_\kk)\big] \Big\Vert_{L^2(B)} \Vert h_\xi\Vert_{L^2(B)}.
\end{align*}
Repeated application of the Leibniz' rule gives
\begin{align*}
\LL^N \big[\chi_R(\cdot) (\cdot-x_R)^\gamma \sigma(\cdot,\lambda_\kk)\big](x)
= \sum_{a,b, \beta,\eta,\nu,\gamma} C_{a,b,\beta,\eta,\nu} \,x^a (x-x_R)^{\gamma-\beta}\chi_R^{(\eta)}(x) \partial_x^\nu \sigma(x,\lambda_\kk),
\end{align*}
where the sum runs over indices such that  $|a|+|b|\le 2N,$ $\beta+\eta+\nu=b$ and $|\beta| \le |\gamma|.$
Set $N=\floor{(n+M)/2}+1;$  by applying the condition $\sigma\in \CN^{m,M},$ we have
\begin{align*}
 \Big\Vert \LL^N \big[\chi_R(\cdot)(\cdot-x_R)^\gamma \sigma(\cdot,\lambda_\kk)\big] \Big\Vert_{L^2(B)}
 &\lesssim  \sum_{a,b, \beta,\eta,\nu,\gamma}  \cro(x_R)^{|\gamma|-|\beta|-|\eta|} \sup_{x\in B}|x|^{|a|} \Big(\int_B|\partial_x^\nu \sigma(x,\lambda_\kk)|^2\,dx\Big)^{\frac{1}{2}} \\
 &\lesssim \sum_{a,b, \beta,\eta,\nu,\gamma} \cro(x_R)^{|\gamma|-|\beta| - |\eta|-|\nu|+n/2} (1+|x_R|)^{|a|} \lambda_\kk^{m/2}\\
 &\lesssim \lambda_\kk^{m/2} (1+|x_R|)^{2N-|\gamma|-n/2},
\end{align*}
where the sums run over indices such that $|a|+|b|\le 2N,$ $\beta+\eta+\nu=b$ and $|\beta| \le |\gamma|.$
Thus we obtain
\begin{align*}
\Big| \int_{\RR^n} \chi_R(x) (x-x_R)^\gamma \sigma(x,\lambda_\kk) h_\xi(x)\,dx\Big| 
\lesssim  \lambda^{-N+m/2}_{\kk}(1+|x_R|)^{2N-|\gamma|-n/2} \Vert h_\xi\Vert_{L^2(B)}.
\end{align*}
Inserting this into the estimate for $II,$ we get
\begin{align*}
|II|
&\lesssim |R|^{1/2}\,\QQ_{4^j}(x_R,x_R)^{1/2}\Big(\sum_{\kk\in I_j}\sum_{|\xi|=\kk} \Big| \lambda^{-N+m/2}_{\kk}(1+|x_R|)^{2N-|\gamma|-n/2}\Big|^2 \Vert h_\xi\Vert_{L^2(B)}^2\Big)^{1/2} \\
&\lesssim  |R|^{1/2} \Big(\f{1+|x_R|}{2^j}\Big)^{2N-n/2-|\gamma|}2^{j(m-|\gamma|-n/2)}\QQ_{4^j}(x_R,x_R)^{1/2}\Big(\sum_{\kk\in I_j}\sum_{|\xi|=\kk} \Vert h_\xi\Vert_{L^2(B)}^2\Big)^{1/2} .
\end{align*}
Observe that
\begin{align}\label{eq:Tcanc1}
\sum_{\kk\le 4^j}\sum_{|\xi|=\kk}\Vert h_\xi\Vert_{L^2(B)}^2 
=\int_{B}\sum_{\kk\le 4^j}\sum_{|\xi|=\kk} h_\xi(y)^2\,dy
=\int_{B} \QQ_{4^j}(y,y) \,dy.
\end{align}
This, the bounds in \eqref{QQ est} and the estimate \eqref{eq:tile control} give
\begin{align*}
|II|
&\lesssim  |R|^{1/2} \Big(\f{1+|x_R|}{2^j}\Big)^{2N-n/2-|\gamma|}2^{j(m-|\gamma|-n/2)} \big(2^{jn}\e_{\ve 4^j}(x_R)^2\big)^{1/2} (2^{jn}|B|)^{1/2}\\
&\sim \Big(\f{1+|x_R|}{2^j}\Big)^{2N-n-|\gamma|}2^{j(m-|\gamma|-n)} |R|^{1/2}2^{jn}\e_{\ve 4^j}(x_R)\\
&\sim \Big(\f{1+|x_R|}{2^j}\Big)^{2N-n-|\gamma|}2^{j(m-|\gamma|-n)} |R|^{-1/2}.
\end{align*}
Finally, since $|x_R|\lesssim 2^j$ and $2N =2\floor{(n+M)/2}+2 \ge n+M+1>n+M+\theta$, we conclude that $|II|$ can be controlled by the right hand side of \eqref{eq:Tcanc}.

\medskip

We next prove the theorem assuming  (b). We proceed along similar lines to the proof for (a). Since  $\sigma \in \SM^{m,\K,\N}_{1,\dd},$  we may utilize Theorem \ref{th: Tsmooth} and estimate term $I$ as above. 

Turning to term $II$, the assumption $\sigma\in \SM^{m,\K,\N}_{1,\dd}$  and \eqref{eq:growthfn2} give, for every $|\nu|\le \N$,
\begin{align*}
\Big(\int_B|\partial_x^\nu \sigma(x,\lambda_\kk)|^2\,dx\Big)^{1/2} 
\lesssim \lambda_\kk^{(m+\dd|\nu|)/2} \Big(\int_B|g(x,\kk)|^2\,dx\Big)^{1/2} 
\sim  \lambda_\kk^{(m+\dd|\nu|)/2} g(x_R,\kk) \cro(x_R)^{n/2}.
\end{align*}
Set $N=\ceil{\f{n+M+1}{2(1-\dd)}};$ then $\N\ge 2N$ and by the above estimate, the fact that $\delta\le 1$, and  inequality \eqref{eq:growthfn1} with some $\ve>4$ and $0\le \vk<1$, we have
\begin{align*}
 &\Big\Vert \LL^N \big[\chi_R(\cdot) (\cdot-x_R)^\gamma \sigma(\cdot,\lambda_\kk)\big] \Big\Vert_{L^2(B)}\\
 &\qquad\lesssim \sum_{a,b,\beta,\eta,\nu,\gamma} \cro(x_R)^{|\gamma|-|\beta| - |\eta|+n/2} (1+|x_R|)^{|a|} \lambda_\kk^{(m+\dd|\nu|)/2} g(x_R,\kk)\\ 
  &\qquad\le \sum_{a,b,\beta,\eta,\nu,\gamma} \cro(x_R)^{|\gamma|-|\beta| - |\eta|-|\nu|+n/2} (1+|x_R|)^{|a|} \lambda_\kk^{m/2}\big(\tfrac{\sqrt{\lambda_\kk}}{1+|x_R|}\big)^{\dd|\nu|} g(x_R,\kk)\\ 
 &\qquad\lesssim \lambda_\kk^{m/2} \max\Big\{1,\tfrac{\sqrt{\lambda_\kk}}{1+|x_R|}\Big\}^{2N\dd} (1+|x_R|)^{2N-|\gamma|-n/2} g(x_R,\kk)\\
 &\qquad\lesssim \lambda_\kk^{m/2} \max\Big\{1,\tfrac{\sqrt{\lambda_\kk}}{1+|x_R|}\Big\}^{2N\dd} (1+|x_R|)^{2N-|\gamma|-n/2} \e_{4\ve\kk}(x_R)^{-\vk},
 \end{align*}
 where the sums run over indices such that $|a|+|b|\le 2N,$ $\beta+\eta+\nu=b$ and  $|\beta| \le |\gamma|.$
  Therefore, we obtain
 \begin{align*}
&\Big| \int_{\RR^n} \chi_R(x) (x-x_R)^\gamma \sigma(x,\lambda_\kk) h_\xi(x)\,dx\Big| \\
&\qquad\lesssim  \lambda^{-N+m/2}_{\kk}\max\Big\{1,\tfrac{\sqrt{\lambda_\kk}}{1+|x_R|}\Big\}^{2N\dd} (1+|x_R|)^{2N-|\gamma|-n/2}\e_{4\ve\kk}(x_R)^{-\vk}\Vert h_\xi\Vert_{L^2(B)}.
\end{align*}
Note that since $|x_R|\lesssim 2^j,$ we have
$$ \max\Big\{1,\f{\sqrt{\lambda_\kk}}{1+|x_R|}\Big\}^{2N\dd} \lesssim \Big(\f{2^j}{1+|x_R|}\Big)^{2N\dd}\quad \forall \kk\in I_j.$$
Inserting the above two estimates into term $II$, gives
\begin{align*}
|II|
&\lesssim |R|^{1/2}\,\QQ_{4^j}(x_R,x_R)^{1/2}\\
&\quad\times\Big(\sum_{\kk\in I_j}\sum_{|\xi|=\kk} \Big| \lambda^{-N+m/2}_{\kk}\max\Big\{1,\tfrac{\sqrt{\lambda_\kk}}{1+|x_R|}\Big\}^{2N\dd} (1+|x_R|)^{2N-|\gamma|-n/2} \e_{4\ve\kk}(x_R)^{-\vk}\Big|^2 \Vert h_\xi\Vert_{L^2(B)}^2\Big)^{1/2} \\
&\lesssim  |R|^{1/2} \Big(\f{1+|x_R|}{2^j}\Big)^{2N(1-\dd)-n/2-|\gamma|}\\
&\quad\times 2^{j(m-|\gamma|-n/2)} \e_{\ve 4^j}(x_R)^{-\vk}\QQ_{4^j}(x_R,x_R)^{1/2} \Big(\sum_{\kk\in I_j}\sum_{|\xi|=\kk} \Vert h_\xi\Vert_{L^2(B)}^2\Big)^{1/2}.
\end{align*}
By \eqref{eq:Tcanc1}, \eqref{QQ est} and \eqref{eq:tile control}, we have
\begin{align*}
|II|
&\lesssim \Big(\f{1+|x_R|}{2^j}\Big)^{2N(1-\dd)-n-|\gamma|}2^{j(m-|\gamma|-n)} |R|^{1/2}2^{jn}\e_{\ve 4^j}(x_R)^{1-\vk}\\
&\lesssim \Big(\f{1+|x_R|}{2^j}\Big)^{2N(1-\dd)-n-|\gamma|}2^{j(m-|\gamma|-n)} |R|^{-1/2}.
\end{align*}
 Our proof is finished once we observe that our choice of $N$ ensures  $2N(1-\dd)\ge n+M+1>n+M+\theta$.
\end{proof}

\subsection{Pseudo-multipliers on distribution spaces}\label{sec: main thms}

In this section, we state and prove our results on boundedness properties of pseudo-multipliers  in Hermite Besov and Hermite Triebel--Lizorkin spaces. 

For smoothness index $\alpha> 0$ we have the following result. 
\begin{Theorem}[No cancellation of molecules required]\label{th: main1}
Let  $m\in\RR$ and $\N,\K\in\NN,$  and suppose that $\sigma\in \SM^{m,\K,\N}_{1,1}.$ 
Let  $\alpha>0,$  $0<q\le\infty$    and $0<p<\infty$ for Triebel--Lizorkin spaces or $0<p\le \infty$ for Besov spaces. If $\alpha, p, q$  satisfy
\begin{align}\label{eq:apq1}
n_{p,q}-n<\alpha<\N \qquad\text{and}\qquad n_{p,q}<\K,
\end{align}
then the operator $T_\sigma$ extends to a bounded operator from $A^{p,q}_{\alpha+m}(\LL)$ to  $A^{p,q}_\alpha(\LL).$

\end{Theorem}

The next result allows for negative values of $\alpha$ by taking $M$ large enough. 

\begin{Theorem}[Cancellation of molecules required]\label{th: main2}
Let $m\in\RR$, $M\in \NN_0$ and $\N,\K\in\NN.$ 
Assume that $\sigma:\RR^n\times\NN_0\to\mathbb{C}$ satisfies one of the following conditions:
\begin{enumerate}[\upshape(a)]
\item $\sigma\in \SM^{m,\K,\N}_{1,1}\cap \CN^{m,M},$
\item $\sigma \in \SM^{m,\K,\N}_{1,\dd}$ for some $0\le \dd<1$ and $\N\ge 2\ceil{\f{n+M+1}{2(1-\dd)}}$. 
\end{enumerate}
Let $\alpha\in\RR,$ $0<q\le\infty$  and  $0<p<\infty$ for Triebel--Lizorkin spaces or $0<p\le \infty$ for Besov spaces. If $\alpha, p, q$ satisfy
\begin{align}\label{eq:apq2}
n_{p,q}-n-M-1 < \alpha<\N  \qquad\text{and}\qquad
\max\{n_{p,q},n+M\}<\K, 
\end{align}
then the operator $T_\sigma$ extends to a bounded operator from $A^{p,q}_{\alpha+m}(\LL)$ to $A^{p,q}_\alpha(\LL).$
\end{Theorem}

Before proving Theorems~\ref{th: main1} and \ref{th: main2}, we state a direct corollary regarding the classes $\SM^{m,\infty,\infty}_{1,\dd}$ for $0\le \delta\le 1.$

\begin{Corollary}\label{coro:main} Let $m\in\RR,$  $\alpha\in\RR,$ $0<q\le\infty,$    $0<p<\infty$ for Triebel--Lizorkin spaces or $0<p\le \infty$ for Besov spaces. Assume that $\sigma:\RR^n\times\NN_0\to\mathbb{C}$ satisfies one of the following conditions:
\begin{enumerate}[\upshape(a)]
\item $\sigma\in \SM^{m,\infty,\infty}_{1,1}\cap \CN^{m,\infty},$
\item $\sigma \in \SM^{m,\infty,\infty}_{1,\dd}$ for some $0\le \dd<1,$
\item $\sigma\in \SM^{m,\infty,\infty}_{1,1}$ and $\alpha>n_{p,q}-n.$
\end{enumerate}
 Then the operator $T_\sigma$ extends to a bounded operator from $A^{p,q}_{\alpha+m}(\LL)$ to $A^{p,q}_\alpha(\LL).$
\end{Corollary}

We first prove Theorem~\ref{th: main2} and then briefly sketch the proof of Theorem~\ref{th: main1}, which follows similarly.

\begin{proof}[Proof of Theorem \ref{th: main2}]
Fix a $\theta\in(0,1)$ such that $\alpha>n_{p,q}-n-M-\theta$ and $\K>\max\{n_{p,q},n+M+\theta\}$, which is possible from our assumption on $\alpha$ in \eqref{eq:apq2}. 

Let $\{\psi_j\}_{j\in \NN_0}$ be an admissible system;
we will  first show that if $j\in \NN_0$ and  $R\in\E_j$, then $2^{-jm}T_\sigma \psi_R$ is a constant multiple of a $(M,\theta, \N-1,1,\K)$-molecule for $A^{p,q}_\alpha(\LL)$ under assumptions (a) or (b).

Firstly, note that  the smoothness estimates (parts (i) and (ii) in Definition~\ref{def: molecules}) follow from Theorem \ref{th: Tsmooth} and \eqref{eq: ebound}. Indeed, 
since $\sigma\in\SM^{m,\K,\N}_{1,\dd},$ Theorem \ref{th: Tsmooth} implies that there exists $0\le \vk<1$ and $\varepsilon>4$ such that for $\gamma\in \NN_0^n$ satisfying $0\le |\gamma|\le \N,$ it holds that
\begin{align*}
|2^{-jm}\partial_x^\gamma T_\sigma\psi_R(x)|
&\lesssim \f{|R|^{-1/2}2^{j|\gamma|}}{(1+2^j|x-x_R|)^\K} \e_{\ve 4^j}(x)^{1-\vk} \qquad  \forall j\in\NN_0,  R\in \E_j,  x\in \RR^n.
\end{align*}
Then, by \eqref{eq: ebound}, and for $\beta\ge 0,$
\begin{align}\label{Tsmooth}
|2^{-jm}\partial_x^\gamma T_\sigma\psi_R(x)|
&\lesssim \f{|R|^{-1/2}2^{j|\gamma|}}{(1+2^j|x-x_R|)^\K} \Big(1+\f{|x|}{2^j}\Big)^{-\beta} \quad \forall j\in\NN_0,  R\in \E_j, x\in \RR^n,
\end{align}
where $0\le |\gamma|\le \N.$
Taking $\beta=\N$ in \eqref{Tsmooth} gives the smoothness estimates in part (i) of Definition~\ref{def: molecules} for $0\le |\gamma|\le \N-1$; taking $\beta=\N$ and $|\gamma|=\N$ in \eqref{Tsmooth} and using Remark~\ref{re: molecules} lead to the smoothness estimates in part (ii) of Definition~\ref{def: molecules} for $|\gamma|=\N-1.$

Secondly, from Theorem \ref{th: Tcanc} we have the cancellation estimates (part (iii) of Definition~\ref{def: molecules})
\begin{align}\label{Tcanc}
\Big|\int_{\RR^n}(x-x_R)^\gamma 2^{-jm}T_\sigma\psi_R(x)\,dx\Big|
\lesssim |R|^{-1/2}2^{-j(n+|\gamma|)} \Big(\f{1+|x_R|}{2^j}\Big)^{M+\theta-|\gamma|},
\end{align}
for  $0\le|\gamma|\le M,$  $j\in \NN_0,$ and $R\in \E_j.$

Notice that  \eqref{eq:apq2} and the facts that $\alpha>n_{p,q}-n-M-\theta$ and $0<\theta<1$  ensure that 
\begin{align*} \N>\alpha, && n+M+\theta+\alpha>n_{p,q}, &&\K>\max\{n_{p,q}, n+M+\theta\}.\end{align*}
We may then apply Theorem \ref{th:synth} (see Remark \ref{rem:synth} (i)) to yield
\begin{align}\label{Tsynth}
 \Big\Vert \sum_{R\in \E} s_R 2^{-jm}T_\sigma\psi_R \Big\Vert_{A^{p,q}_\alpha}\lesssim \Vert \{s_R\}_{R\in \E}\Vert_{a^{p,q}_{\alpha}}
 \end{align}
for any sequence of numbers $\{s_R\}_{R\in \E}$.

Using Theorem \ref{th: frame} (c) and the linearity of $T_\sigma,$ we have
\begin{align*}
	\Vert T_\sigma f \Vert_{A^{p,q}_{\alpha}}
	&=\Big\Vert \sum_{R\in\E} 2^{jm}\ip{f,\varphi_R} 2^{-jm} T_\sigma(\psi_R)\Big\Vert_{A^{p,q}_{\alpha}},
		\end{align*}
where the admissible systems $\{\vp_j\}_{j\in \NN_0}$ and $\{\psi_j\}_{j\in \NN_0}$ satisfy \eqref{CRF 1}.
By \eqref{Tsynth}, the definition of sequence spaces and Theorem \ref{th: frame} (b), we finally have
\begin{align*}	
	\Vert T_\sigma f \Vert_{A^{p,q}_{\alpha}}
\lesssim \big\Vert  \{2^{jm}\ip{f,\varphi_R}\}_{R\in \E} \big\Vert_{a^{p,q}_{\alpha}} 
	= \big\Vert  \{\ip{f,\varphi_R}\}_{R\in \E} \big\Vert_{a^{p,q}_{\alpha+m}} 
	\lesssim \Vert f\Vert_{A^{p,q}_{\alpha+m}}.
	\end{align*}
This completes the proof of Theorem \ref{th: main2}. 
\end{proof}

\begin{proof}[Proof of Theorem \ref{th: main1}]
The proof of Theorem \ref{th: main1} is similar to the proof of Theorem \ref{th: main2} except that  only the estimates in parts (i) and (ii) of Definition~\ref{def: molecules} hold. Thus $2^{-jm}T_\sigma\psi_R$ are $(-1,1,\N-1,1,\K)$-molecules. The conditions  \eqref{eq:apq1}   and the fact that $n_{p,q}\ge n$ ensure 
\begin{align*} \N>\alpha, && \alpha+n>n_{p,q}, && \K>\max\{n_{p,q}, n\}. \end{align*}
Therefore, \eqref{Tsynth} holds in view of Remark \ref{rem:synth} (i). The rest of the proof follows as in the proof of Theorem \ref{th: main2}.
\end{proof}
\section{Examples and applications}\label{sec: applications}

In this section, we present examples and applications of the results obtained in Section~\ref{sec: PDOs}. We start by giving examples of symbols in the classes $\SM^{m,\K,\N}_{\rho,\delta}.$ We next  consider implications of Theorem \ref{th: main2}  for boundedness properties of pseudo-multipliers on Lebesgue spaces $L^p(\RR^n)$  with $1<p<\infty$  (Corollaries~\ref{cor:Lp1} and \ref{cor:Lp2}) and we compare them with existing results in the literature;  as a byproduct, we obtain weighted estimates for  pseudo-multipliers with symbols  of order zero (Corollary~\ref{cor:wLp}). For $0<p\le 1,$ we show that Theorem~\ref{th: main2} leads to boundedness of pseudo-multipliers in the setting of Hermite local Hardy spaces (Corollary~\ref{cor:hardy}). We also comment on boundedness properties of Hermite multipliers. Finally, we present an example of a linearization process of a non-linear problem inspired by the works \cite{MR639462} and \cite{MR631751} (Theorem~\ref{thm: appnonlin}) that along with Theorem~\ref{th: main1} implies that Hermite Besov spaces and Hermite Triebel--Lizorkin spaces are closed under non-linearities (Corollary~\ref{coro: appnonlin}).

\begin{Example}[A symbol in $\SM^{m,\infty,\infty}_{\rr,\dd}$]\label{ex: 1}
Consider $\sigma(x,\xs)=\Phi(x)\Psi(\xs)$ where $\Phi\in C_0^\infty(\RR^n)$ and $\Psi\in \mathscr{S}(\RR)$. From the mean-value property 
$|\diff_\xs^\kappa \Psi(\xs)|=|\Psi^{(\kappa)}(\nu)|$
for some $\nu\in(\xs,\xs+\kappa)$; it then follows  that $\sigma\in\SM^{m,\infty,\infty}_{\rr,\dd}$ for any $m,\rr,\dd$.
\end{Example}

\begin{Example}[A symbol in $\SM^{0,\infty,\infty}_{1,\dd}$]\label{ex: 2}
For  $0\le \dd\le 1,$ let
$$\sigma(x,\xs)=\sum_{j\in\NN} \sigma_j(x)\,\vph_j(\sqrt{\lambda_\xs})$$
where $\{\vph_j\}_{j\in \NN_0}$ is an admissible system and $\sigma_j\in C^\infty(\RR^n)$ satisfies
$$ |\partial^\nu \sigma_j(x)|\le  C_\nu\,2^{j\dd|\nu|}\Big(1+\f{|x|}{2^j}\Big)^\beta, \qquad \forall\;\nu\in\NN^n_0,  j\in \NN, x\in \RR^n,$$
for some $\beta\ge 0.$
We next show that $\sigma\in\SM^{0,\infty,\infty}_{1,\dd}$. 

For $j\in \NN,$ $\vph_j(\sqrt{\lambda_\xs})$ is supported, as a function of $\xs\in \NN_0,$ in $I_j=[\frac{1}{2}4^{j-2}-\floor{\f{n}{2}}, \frac{1}{2}4^j-\ceil{\f{n}{2}}]\cap \NN_0;$ for each fixed $\xs$, the sum contains at most five  nonzero terms, and these occur for those $j$ such that $\lambda_\xs\sim 4^j$ (more precisely, $4^{j-2}\le \lambda_\xs\le 4^j$).
For a given $\nu\in \NN_0^n,$ we have
    $$ |\partial^\nu \sigma_j(x)|\lesssim  \lambda_\xs^{\dd|\nu|/2}\Big(1+\f{|x|}{\sqrt{\lambda_\xs}}\Big)^\beta,\quad \lambda_\xs\sim 4^j.$$
    From Lemma \ref{lem: hoppe} (a), given $\kappa\in \NN_0$ and $N>\kappa,$
    $$ |\diff_\xs^\kappa \vph_j(\sqrt{\lambda_\xs})|\lesssim \lambda_\xs^{N/2-\kappa}2^{-jN}\Vert \vph^{(N)}\Vert_\infty \sim \lambda_\xs^{-\kappa},\qquad \lambda_\xs\sim 4^j.$$
    Therefore, we obtain
        \begin{align*}
         |\partial_x^\nu \diff_\xs^\kappa \sigma(x,\xs)|
         \lesssim \lambda_\xs^{\dd|\nu|/2-\kappa}\Big(1+\f{|x|}{\sqrt{\lambda_\xs}}\Big)^\beta\sim (1+\sqrt{\xs})^{\dd|\nu|-2\kappa}g(x,\xs)
         \end{align*}
where 
$$ g(x,\xs)=\Big(1+\f{|x|}{1+\sqrt{\xs}}\Big)^\beta. $$
It follows immediately from \eqref{eq: ebound} that $g$ satisfies \eqref{eq:growthfn1} and \eqref{eq:growthfn2}.
\end{Example}

\begin{Example}[A symbol in $\SM^{0,\infty,\infty}_{1,1}\cap \CN^{0,M},$ $M\in \NN_0$]\label{ex: 3}
Consider Example~\ref{ex: 2} with $\dd=1$ and assume further that $\sigma_j$ is supported in the set $\{x:\; 2^{j}\le |x|< 2^{j+1}\}$. Then $\sigma\in \SM^{0,\infty,\infty}_{1,1}\cap \CN^{0,M}$. Indeed, from Example~\ref{ex: 2} we have $\sigma\in\SM^0_{1,1};$ thus it remains to be checked that $\sigma\in\CN^{0,M}$. We first note if $|x|\ge2$ then there is a unique $j\in \NN$ such that  $x\in\supp \sigma_j$ and we have 
$$ 2^{j-1}\le 1+|x|\le 2^{j+2};$$
therefore, $\cro(x)\sim 2^{-j}$. For each such $x$ and any $\gamma\in\NN_0^n$, we obtain 
\begin{align*}
|\partial_x^\gamma \sigma(x,\xs)|
\le \sum_{\nu\in \NN}|\partial^\gamma\sigma_\nu(x)|\,|\vph_\nu(\sqrt{\lambda_\xs})| \lesssim 2^{j|\gamma|}\sim (1+|x|)^{|\gamma|}=\cro(x)^{-|\gamma|}.
\end{align*}
If $|x| \le 2$ then $\sigma(x,\xs)=0$ and the above estimate holds trivially. Thus for any $x\in\RR^n$ we get
\begin{align*}
\Big(\aver{B(x,\cro(x))}|\partial_y^\gamma \sigma(y,\xs)|^2\,dy\Big)^{1/2}\lesssim \Big(\aver{B(x,\cro(x))}\cro(y)^{-2|\gamma|}\,dy\Big)^{1/2}\sim \cro(x)^{-|\gamma|},
\end{align*}
where in the last step we used property \eqref{eq:rhoproperty}. This gives that $\sigma\in\CN^{0,M}$. 

\end{Example}

\subsection{Results for the $L^p$ scale, $1<p<\infty$}\label{sec:lebesgue}
In this section we present consequences  of Theorem \ref{th: main2}  for boundedness properties of pseudo-multipliers on $L^p(\RR^n)$  for  $1<p<\infty$.
We first recall the relation 
\begin{align}\label{eq:Lp equiv} F^{p,2}_0(\LL) = L^p(\RR^n),\qquad 1<p<\infty,\end{align}
with equivalent norms; see \cite[Proposition 5]{MR2399106} or \cite[Theorem 1.2]{MR1330218} for dimension 1.
Combining \eqref{eq:Lp equiv} with Theorem \ref{th: main2} we obtain the following result.
\begin{Corollary}[$L^p$ boundedness for symbols with growth]\label{cor:Lp1}
Assume that $\sigma:\RR^n\times\NN_0\to\mathbb{C}$ satisfies one of the following conditions:
\begin{enumerate}[\upshape(a)]
\item $\sigma\in \SM^{0,\K,1}_{1,1}\cap \CN^{0,0}$  for some $\K\ge n+1,$
\item $\sigma \in \SM^{0,\K,\N}_{1,\dd}$ for some $0\le \dd<1$, $\N\ge 2\ceil{\f{n+1}{2(1-\dd)}}$ and $\K \ge n+1$.
\end{enumerate}
Then $T_\sigma$ extends to a bounded operator on  $L^p(\RR^n)$ for all $1<p<\infty$.
\end{Corollary}
\begin{proof} For $1<p<\infty$ and $q=2,$ we have  $n_{p,q}=n$.  The conditions in \eqref{eq:apq2} are satisfied with $n_{p,q}=n,$ $M=0,$ $\alpha=0,$ $\N\in\NN$  and $\K\ge n+1.$ Applying  Theorem \ref{th: main2} with those values and with $m=0$   yield the desired result through  the use of \eqref{eq:Lp equiv}.
\end{proof}

\subsubsection{Comparisons with other results on $L^p$}
In this section we compare our results in Corollary \ref{cor:Lp1} with existing results in the literature and give some further consequences. 
The works \cite{MR3280055,CR,MR1343690} address symbols that satisfy \eqref{eq:sm} but without admissible growth (i.e. $g\equiv 1$). In order to  continue the discussion we define the following class.
\begin{Definition}[Symbols without admissible  growth]\label{def:sc}
Let $m\in \RR$, $\rr,\dd\ge 0$, and $\N, \K\in \NN_0\cup\{\infty\}$. We say the symbol $\sigma: \RR^n\times \NN_0\to\mathbb{C}$ satisfies $\sigma \in \SC^{m,\K,\N}_{\rr,\dd}$ if 
\begin{align}\label{eq:sc}
	| \partial^\nu_x \diff^\kappa_\kk \sigma(x,\xs)| \lesssim (1+\sqrt{\xs})^{m-2\rr\kappa+\dd|\nu|}  \qquad \forall (x,\xs)\in \RR^n\times\NN_0
\end{align}
for $\nu\in \NN_0^n$ such that $0\le |\nu|\le \N$ and $0\le \kappa\le \K$. If $\N=\infty$ or $\K=\infty,$ the implicit constant in \eqref{eq:sc} may depend on $\nu$ or $\kappa$ respectively.
\end{Definition} 
Since $\SC^{m,\K,\N}_{\rr,\dd}\subset \SM^{m,\K,\N}_{\rr,\dd}$ we have the following immediate consequence of Corollary \ref{cor:Lp1}.
\begin{Corollary}[$L^p$ boundedness for symbols without growth]\label{cor:Lp2}
Assume that $\sigma:\RR^n\times\NN_0\to\mathbb{C}$ satisfies one of the following conditions:
\begin{enumerate}[\upshape(a)]
\item $\sigma\in \SC^{0,\K,1}_{1,1}\cap \CN^{0,0}$  for some $\K\ge n+1,$
\item $\sigma \in \SC^{0,\K,\N}_{1,\dd}$ for some $0\le \dd<1$, $\N\ge 2\ceil{\f{n+1}{2(1-\dd)}}$ and $\K \ge n+1$.
\end{enumerate}
Then $T_\sigma$ extends to a bounded operator  on $L^p(\RR^n)$ for all $1<p<\infty$.
\end{Corollary}
\noindent In addition, by invoking \cite[Theorem 1.4]{MR3280055}, we obtain the following weighted estimates, where $A_p$ denotes the Muckenhoupt class of weights.
\begin{Corollary}[Weighted $L^p$ boundedness]\label{cor:wLp}
If   $\sigma\in \SC^{0,\K,1}_{1,0}\cap \CN^{0,0}$  with $\K\ge n+1$ (which implies assumption (a) of Corollary \ref{cor:Lp2}) or $\sigma $ satisfies assumption (b) of Corollary \ref{cor:Lp2}  with $\delta=0,$  the operator $T_\sigma$ is bounded on $L^p_w(\RR^n)$ for every $w\in A_p$ and $1<p<\infty$.
\end{Corollary}

Let us compare Corollary \ref{cor:Lp2} with existing results.
First,  assuming that $T_\sigma$ is a priori bounded on $L^2(\RR^n)$, the authors in \cite[Theorem 1.4]{MR3280055} prove $L^p(\RR^n)$ boundedness for all $1<p<\infty$ (actually $L^p_w(\RR^n)$ with $w\in A_p$) provided that $\sigma\in \SC^{0,n+1,1}_{1,0}$. 
 On the other hand, \cite[Corollary 2.12]{CR} does not assume $L^2(\RR^n)$ boundedness but assumes  $\N=0$ and $\K=2n+1$, requiring no regularity in $x$ but more regularity in $\kk$.
By contrast, say for $\dd=0$, in place of $L^2(\RR^n)$ boundedness we assume  more regularity in the $x$ variable with $\N=2\ceil{(n+1)/2}=2\floor{n/2}+2$, but retain the same level of $\kk$-regularity as \cite{MR3280055} with $\K=n+1$.
 It is worth noting that \cite[Corollary 2.12]{CR} is a consequence of \cite[Theorem 1.1]{CR}, which addresses symbols satisfying estimates  in terms of Sobolev norms (H\"ormander-type conditions). Such conditions do not imply those assumed in Corollary~\ref{cor:Lp2}.

\subsection{Results for $0<p\le 1$}\label{sec:hardy} For $0<p\le 1,$  the local Hardy spaces $h^p(\RR^n)$ (as defined in \cite{MR523600}) are better suited than $L^p(\RR^n)$ or the Hardy spaces $H^p(\RR^n)$ for  boundedness of pseudo-differential operators in the Euclidean setting;  see for instance \cite{MR523600}, \cite{MR2500920}  and \cite{MR975206}. In this context, it holds that $h^p(\RR^n)=F^{p,2}_0(\RR^n)$ for $0<p\le 1$ and $h^p(\RR^n)=L^p(\RR^n)=F^{p,2}_0(\RR^n)$ for $1<p<\infty,$ where $F^{p,2}_0(\RR^n)$ are the classical Triebel--Lizorkin spaces in Euclidean space.

With this in mind, given an admissible system $\{\vph_j\}_{j\in\NN_0},$ we make the following definition.
\begin{Definition}[Hermite local Hardy spaces] For  $0<p\le 1,$   we define  the  Hermite local Hardy space $h^p(\LL)$ as the class of tempered distributions $f\in\mathscr{S}'(\RR^n)$ such that
\begin{align*}
\Vert f\Vert_{h^p(\LL)} =\Big\Vert \Big(\sum_{j\in\NN_0}|\vph_j(\sqrt{\LL})f|^2\Big)^{1/2}\Big\Vert_{L^p}<\infty.
\end{align*}
\end{Definition}
 Note that
\begin{align}\label{eq:hardyrelation}
 F_0^{p,2}(\LL) &= h^p(\LL) , \qquad 0<p\le 1.
\end{align}
In particular, as discussed in  Section \ref{sec: spaces},  the spaces $h^p(\LL)$ are independent of the choice of $\{\vph_j\}_{j\in\NN_0}$ and are quasi-Banach spaces. These spaces coincide with those introduced in \cite{MR1631616}; see Remark \ref{rem:atomic hardy} below. 

From Theorem \ref{th: main2} and \eqref{eq:hardyrelation} we have the following result. 
\begin{Corollary}[$h^p$ boundedness]\label{cor:hardy}
Let $0<p\le 1,$ $M\in\NN_0$ and  $M\ge\floor{n(\f{1}{p}-1)}$.
Assume that $\sigma:\RR^n\times\NN_0\to\mathbb{C}$ satisfies one of the following conditions:
\begin{enumerate}[\upshape(a)]
\item $\sigma\in \SM^{0,\K,1}_{1,1}\cap \CN^{0,M}$  for some $\K\ge \floor{n/p}+1,$
\item $\sigma \in \SM^{0,\K,\N}_{1,\dd}$ for some $0\le \dd<1$, $\N\ge 2\ceil{\f{n+M+1}{2(1-\dd)}}$ and $\K \ge \floor{n/p}+1$.
\end{enumerate}
Then $T_\sigma$ extends to a bounded operator on  $h^p(\LL)$.
\end{Corollary}
\begin{proof} For $0<p\le 1$ and $q=2,$ we have  $n_{p,q}=n/p.$ The conditions in \eqref{eq:apq2} are satisfied with $n_{p,q}=n/p,$   $M>n/p-n-1,$ $\alpha=0,$ $\N\in \NN$ and $\K>\max\{n/p, n+M\}.$ Thus, we can apply Theorem \ref{th: main2} with $0<p\le 1,$  $q=2,$ $\alpha=0,$  $M\ge \floor{n(\f{1}{p}-1)},$  $\K\ge\floor{n/p}+1,$ $\N\ge1$ or $\N\ge 2\ceil{\f{n+M+1}{2(1-\dd)}}$ to obtain the desired result through the use of \eqref{eq:hardyrelation}.
\end{proof}
\begin{Remark} The following consequences can be easily deduced from Corollary \ref{cor:hardy}. If $\sigma\in \SM^{0,\infty,\infty}_{1,1}\cap \CN^{0,M}$ for some $M\in\NN_0$ then $T_\sigma$ is bounded on $h^p(\LL)$ for  $\tfrac{n}{n+M+1}<p\le 1$.
On the other hand, if $\sigma\in\SM^{0,\infty,\infty}_{1,\dd}$ for some $0\le \dd<1,$ then $T_\sigma$ is bounded on $h^p(\LL)$ for every $0<p\le 1$.  
\end{Remark}
\begin{Remark}\label{rem:atomic hardy}
We next observe that Corollary \ref{cor:hardy} also yields results on the classical local Hardy spaces $h^p(\RR^n)$. Indeed, if $\sigma\in \SM^{0,\infty,\infty}_{1,1}\cap \CN^{0,M}$ for some $M\in\NN_0,$ then $T_\sigma$ maps $h^p(\RR^n)$ into $h^p(\LL)$ for  $\tfrac{n}{n+M+1}<p\le 1$, while if $\sigma\in\SM^{0,\infty,\infty}_{1,\dd}$ for some $0\le \dd<1$, then $T_\sigma$ maps $h^p(\RR^n)$ into $h^p(\LL)$ for every $0<p\le 1$. 

 In order to show this, we first  recall the atomic Hardy space  associated to the Hermite operator introduced in \cite{MR1631616}. Given $0<p\le 1,$ a function $a$ defined on $\RR^n$ is a  $p$-atom  if there exists a ball $B=B(x,r),$ where $x\in \RR^n$ and $r>0,$ such that
\begin{enumerate}[(i)]
\item $\supp a\subset B,$
\item $\Vert a\Vert_\infty\le |B|^{-1/p},$
\item if $r\le \f{1}{2}\cro(x)$ then $\displaystyle\int_{B} x^\gamma a(x)\,dx=0$ for every $|\gamma|\le \floor{n\big(\f{1}{p}-1\big)}$.
\end{enumerate}
The atomic Hardy space $h^p_\cro(\RR^n)$ associated to the Hermite operator is defined as the class of all $f\in\mathscr{S}'(\RR^n)$ such that 
$$ \Vert f\Vert_{h^p_\cro}=\inf\Big\{\big(\sum_{j\in \NN}|c_j|^p\big)^{1/p}\Big\}<\infty,$$
where the infimum is taken over all representations $f=\sum_{j\in \NN} c_j a_j$ with scalars $c_j$ and $p$-atoms $a_j$. 
Note that $h^p(\RR^n) \subset h^p_\cro(\RR^n)$ for $0<p\le 1$ since atoms in $h^p(\RR^n),$ as defined in \cite{MR523600}, satisfy conditions (i), (ii) and (iii).

It was recently proved  that the atomic Hardy space $h^p_\cro(\RR^n)$ coincides with  the Hermite Triebel--Lizorkin space $F^{p,2}_0(\LL)$ for every $0<p\le 1$ (see \cite[Theorem 9]{MR3354367} and \cite[Remarks 2.20 and 2.7]{MR4205254}). Thus, the relations
$$ h^p(\RR^n) \subset h^p_\cro(\RR^n)=h^p(\LL)$$ 
and Corollary~\ref{cor:hardy} lead to the desired results. 
\end{Remark}

\subsection{Hermite multipliers}

Let $\sigma=\sigma(\xs)$ be a symbol in $S^{m,\K,\infty}_{1,0}$  for some $m\in\RR$ and $\K\in \NN_0;$ thus, $\sigma$ satisfies 
\begin{align} \label{eq:multiplier1}
|\diff^\kappa \sigma(\xs)| \lesssim (1+\sqrt{\xs})^{m-2\kappa}\qquad \forall \xs\in \NN_0, \;0\le \kappa\le \K.
\end{align}

For the  case $m=0,$ it was shown in \cite[Theorem 4.2.1]{MR1215939} that  $T_\sigma$ is bounded on $L^p(\RR^n)$ for all $1<p<\infty$ provided $\K= \floor{n/2}+1$. In dimension $n=1$ with $m=0$ and $\K=1$, the boundedness on $F^{p,q}_\alpha(\LL)$ for $p,q>1$ and $\alpha\in\RR$ was obtained in \cite[Theorem 1]{MR1343690}.

By applying Theorem \ref{th: main2}, we obtain that if $M\in \NN_0$ and $\K> \max\{n_{p,q}, n+M\},$ then $T_\sigma$ is bounded from $A^{p,q}_{\alpha+m}(\LL)$ to  $A^{p,q}_\alpha(\LL)$ provided
$$ n_{p,q}-n-M-1<\alpha.$$
In particular, if $\K=\infty$ then the result holds for all $\alpha\in\RR$ and every $p,q>0$.

\subsection{On a result of Meyer and Bony}\label{sec: meyer}
In this section, we  present an example of a linearization process of a non-linear problem that along with Theorem~\ref{th: main1} implies that Hermite Besov spaces and Hermite Triebel--Lizorkin spaces are closed under non-linearities.

\begin{Theorem}[A linearization formula]\label{thm: appnonlin}
Let  $H\in \mathcal{C}^\infty(\RR)$ be such that $H(0)=0.$  If $f\in \sz(\RR^n)$  is real-valued,
there exists $\sigma_f\in S^{0,\infty,\infty}_{1,1}$ such that $H(f)=T_{\sigma_f}(f).$ In particular, if $\nu\in \NN_0^n$ and $\kappa\in \NN_0,$ the symbol $\sigma_f$ satisfies
\begin{equation}\label{eq: appnonlin}
|\partial_x^\nu\diff_\xs^\kappa\sigma_f(x,\xs)|\lesssim \left(\sup_{|\lambda|\lesssim \|f\|_{L^\infty}}\sum_{\ell=0}^{|\nu|}|H^{(\ell+1)}(\lambda)|\|f\|_{L^\infty}^\ell\right)\, \lambda_\xs^{|\nu|/2-\kappa} \quad \forall x\in \RR^n,  \xs\in \NN_0,
\end{equation}
where the implicit constant may depend on $\nu$ and $\kappa$ and is independent of $f.$
\end{Theorem}

\noindent As a consequence of Theorem~\ref{thm: appnonlin} and Theorem~\ref{th: main1} we obtain the following result. 

\begin{Corollary}[Closure under non-linearities] \label{coro: appnonlin} Assume  $0<p<\infty,$ $0<q<\infty,$    $\alpha>n_{p,q}-n$ and  $H\in \mathcal{C}^\infty(\RR)$ is such that $H(0)=0.$ If $f\in A^{p,q}_\alpha(\LL)\cap L^\infty(\RR^n)$ is real-valued, then $H(f)\in A^{p,q}_\alpha(\LL)\cap L^\infty(\RR^n).$
\end{Corollary}

We next proceed with the proofs of the stated results.

\begin{proof}[Proof of Theorem~\ref{thm: appnonlin}] Let $f\in \sz(\RR^n)$ be real-valued.
Assume that $(\varphi_0,\varphi)$ is an admissible pair such that $\sum_{j=0}^\infty \varphi_j(\lambda)= 1$ for all $\lambda\ge 0$ and $\varphi_0^{(\ell)}(0)=0$ for all $\ell\in \NN.$   
We have
\begin{equation*}
\sum_{j\in\NN_0}\varphi_j(\sqrt{\LL})f=f,
\end{equation*}
where the series converges absolutely and uniformly in $\RR^n$  since $f\in \sz(\RR^n).$  
Define $f_j=\sum_{\ell=0}^j \varphi_\ell(\sqrt{\LL})f$ for $j\in \NN_0$ and $f_{-1}=0.$ Since $f_j\to f$ uniformly on $\RR^n$, $H$ is continuous and $H(0)=0,$ it follows that
\begin{equation*}
H(f)=\lim_{j\to\infty} H(f_j)=\sum_{j\in\NN_0} H(f_j)-H(f_{j-1})
\end{equation*}
pointwise in $\RR^n$ (even more, through the Mean Value Theorem, it follows that the convergence is uniform  using that $\sup_{j\in \NN_0} \|f_j\|_{L^\infty}<\infty,$ as shown below, and that $H'$ is continuous).
The Mean Value Theorem gives that
\begin{align*}
H(f_j)-H(f_{j-1})&=\int_0^1H'(t f_j+(1-t)f_{j-1})\,dt\, (f_j-f_{j-1})\\
& =\int_0^1H'(f_{j-1}+t\varphi_j(\sqrt{\LL})f)\,dt\, \varphi_j(\sqrt{\LL})f.
\end{align*}
Setting $m_j=\int_0^1H'(f_{j-1}+t\varphi_j(\sqrt{\LL})f)\,dt,$ we then have 
\begin{align*}
H(f)(x)=\sum_{j\in\NN_0} m_j(x) \varphi_j(\sqrt{\LL})f(x)
=\sum_{\kk\in\NN_0}  \left(\sum_{j\in\NN_0} m_j(x) \varphi_j(\sqrt{\lambda_\kk})\right)\PP_\kk(f)(x),
\end{align*}
which means that $H(f)$ can be realized as the action on $f$ of the pseudo-multiplier with symbol 
\begin{equation*}
\sigma_f(x,\kk)=\sum_{j\in\NN_0} m_j(x) \varphi_j(\sqrt{\lambda_\kk}).
\end{equation*}

We next prove that $\sigma_f\in S^{0,\infty,\infty}_{1,1}$ by showing \eqref{eq: appnonlin}. We have 
\begin{align*}
\partial_x^\nu\diff_\kk^\kappa\sigma_f(x,\kk)= \sum_{j\in\NN_0} \partial^\nu m_j(x) \diff_\kk^\kappa(\varphi_j(\sqrt{\lambda_\kk}))=\sum_{j=\floor{\f{1}{2}(\log_2\lambda_\kk -1)}}^{\ceil{\f{1}{2}(\log_2(\lambda_\kk +2\kappa)+1)}} \partial^\nu m_j(x) \diff_\kk^\kappa(\varphi_j(\sqrt{\lambda_\kk})).
\end{align*}
Note that the number of terms in the  last sum is bounded by a number independent of $\kk$ and dependent on $\kappa.$
 It is then  enough to show that
\begin{equation}\label{eq: nl1}
 |\partial^\nu m_j(x)\diff_\kk^\kappa(\varphi_j(\sqrt{\lambda_\kk}))| \lesssim \left(\sup_{|\lambda|\lesssim \|f\|_{L^\infty}}\sum_{\ell=0}^{|\nu|}|H^{(\ell+1)}(\lambda)|\|f\|_{L^\infty}^\ell\right)\, \lambda_\kk^{|\nu|/2-\kappa} \quad \forall x\in \RR^n, \kk\in \NN_0,
 \end{equation}
  for  $j=\floor{\f{1}{2}(\log_2\lambda_\kk -1)},\dots,\ceil{\f{1}{2}(\log_2(\lambda_\kk +2\kappa)+1)}$ and where the implicit constant may depend on $\nu$ and $\kappa$ and is independent of $f.$
  
Using Lemma~\ref{lem: hoppe} and taking $N$ so that $N>\kappa,$ it holds that 
\begin{equation}\label{eq: nl2}
|\diff_\kk^\kappa(\varphi_j(\sqrt{\lambda_\kk}))|\lesssim \lambda_\kk^{N/2-\kappa} 2^{-jN}\lesssim \lambda_\kk^{-\kappa},
\end{equation}
where  it was used that  $\sqrt{\lambda_\kk}\sim 2^j$  for $j=\floor{\f{1}{2}(\log_2\lambda_\kk-1)},\dots,\ceil{\f{1}{2}(\log_2(\lambda_\kk +2\kappa)+1)}.$

We next observe that if $\gamma\in \NN_0^n$ then
\begin{equation}\label{eq: claim}
\max\{\|\partial^{\gamma}f_{j}\|_{L^\infty},\|\partial^{\gamma}\varphi_j(\sqrt{\LL})f\|_{L^\infty}\}\lesssim 2^{j|\gamma|}\|f\|_{L^\infty}.
\end{equation}   
Indeed, the estimate for $\|\partial^{\gamma}\varphi_j(\sqrt{\LL})f\|_{L^\infty}$ follows from  \eqref{phiest A}. Regarding $\|\partial^{\gamma}f_{j}\|_{L^\infty},$ 
note that $f_j={\varphi_0}_j(\sqrt{\LL})f$ with ${\varphi_0}_j(\lambda)=\varphi_0(2^{-j}\lambda)$ since $\sum_{l\in\NN_0} \varphi_l\equiv 1$ gives that $\varphi_0(2^{-j}\lambda)=\sum_{l=0}^j\varphi_l(\lambda)$ for $\lambda\ge 0.$ This and the fact that $\varphi_0^{(l)}(0)=0$ for $l\in \NN$ imply that  the estimate for $\|\partial^{\gamma}f_{j}\|_{L^\infty}$ is also a consequence of  \eqref{phiest A} (see Remark~\ref{re: phiest}).

For $\nu=0,$ \eqref{eq: claim} gives
\begin{equation}\label{eq: nl4}
|m_j(x)|\le \sup_{|\lambda|\lesssim\|f\|_{L^\infty} }|H'(\lambda)|.
\end{equation}
To estimate $\partial^{\nu}m_j$ for $\nu\neq 0$ we will use  Fa\`a di Bruno's formula for the partial derivatives of a composition $G(g)$, where $g : \RR^n \to \RR$ and $G : \RR \to \RR$ are smooth functions. Namely, for  $\nu \in \NN_0^n$, it holds that
\begin{equation*}
\partial^\nu G(g) = \sum\limits_{\nu = \nu_1 + \dots  + \nu_\ell} G^{(\ell)}(g) \prod_{r=1}^\ell \partial^{\nu_r} g,
\end{equation*}
where the sum is over all the multi-index decompositions $\nu = \nu_1 + \dots + \nu_\ell$ with $\ell \geq 1,$  $\nu_r \in \NN_0^n$ and $\nu_r\neq 0$ for $r=1,\ldots, \ell;$ notice that $\ell \leq |\nu|$.   
We then have
\begin{align*}
\partial^\nu m_j&=\int_0^1 \partial^\nu(H'(f_{j-1}+ t \varphi_j(\sqrt{\LL})f)) \,dt\\
&= \int_0^1 \sum\limits_{\nu = \nu_1 + \dots  + \nu_\ell} H^{(\ell +1)}(f_{j-1}+ t \varphi_j(\sqrt{\LL})f) \prod_{r=1}^\ell \partial^{\nu_r} (f_{j-1}+ t \varphi_j(\sqrt{\LL})f) \,dt.
\end{align*}
By \eqref{eq: claim} and since  $\sqrt{\lambda_\kk}\sim 2^j$ for $j=\floor{\f{1}{2}(\log_2\lambda_\kk -1)},\dots,\ceil{\f{1}{2}(\log_2(\lambda_\kk +2\kappa)+1)},$ we obtain
\begin{equation*}
\left|\prod_{r=1}^\ell \partial^{\nu_r} (f_{j-1}+ t \varphi_j(\sqrt{\LL})f)(x)\right|\lesssim \prod_{r=1}^\ell 2^{j|\nu_r|}\|f\|_{L^\infty}=2^{j|\nu|}\|f\|_{L^\infty}^\ell\lesssim \lambda_\kk^{|\nu|/2}\|f\|_{L^\infty}^\ell.
\end{equation*}
This implies that
\begin{equation}\label{eq: nl3}
|\partial^\nu m_j(x)|\lesssim \lambda_\kk^{|\nu|/2}\sup_{|\lambda|\lesssim \|f\|_{L^\infty}}\sum_{\ell=1}^{|\nu|}|H^{(\ell+1)}(\lambda)|\|f\|_{L^\infty}^\ell\quad \forall x\in\RR^n.
\end{equation}

The desired estimate \eqref{eq: nl1}  then follows from \eqref{eq: nl2}, \eqref{eq: nl4} and \eqref{eq: nl3}.
\end{proof}

We note that \eqref{eq: claim} is true for functions in $L^\infty(\RR^n),$ not just  in $\sz(\RR^n).$ Also, if $f\in A^{p,q}_\alpha(\LL)$ and $\{f_j\}_{j\in \NN}$ is  as in the proof of Theorem~\ref{thm: appnonlin} then  $f_j\in \sz(\RR^n),$ since it is a finite linear combination of Hermite functions, and  $f_j\to f$ in $A^{p,q}_\alpha(\LL);$ the latter  can be proved using the same ideas as in \cite[Section 2.3.3]{MR3024598}. These facts will be used in the proof of Corollary~\ref{coro: appnonlin}.

\begin{proof}[Proof of Corollary~\ref{coro: appnonlin}]  If $f\in \sz(\RR^n),$  then $H(f)\in A^{p,q}_\alpha(\LL)$ as a consequence of Theorem~\ref{thm: appnonlin} and Theorem~\ref{th: main1}. Consider $f\in  A^{p,q}_\alpha(\LL)\cap L^\infty(\RR^n)$ and note that $H(f)\in L^\infty(\RR^n)$ since $f$ is bounded and $H$ is continuous.

 Let $\{f_j\}_{j\in \NN}$ be  as in the proof of Theorem~\ref{thm: appnonlin}; then $f_j\in \sz(\RR^n),$  $f_j\to f$ in $A^{p,q}_\alpha(\LL)$ and, by \eqref{eq: claim}, $\sup_{j\in \NN}\|f_j\|_{L^\infty}\le\|f\|_\infty.$
  We have
\begin{equation*}
|H(f)(x)-H(f_j)(x)|=|H'(c_{j,x})||f_j(x)-f(x)|\lesssim \sup_{|\lambda|\le \|f\|_{L^\infty}}H'(\lambda) \,|f_j(x)-f(x)|,
\end{equation*} 
where  $c_{j,x}$ is a convex linear combination of $f_j(x)$ and $f(x)$ and the supremum is finite since $H'$ is continuous.
This estimate and the fact that  $f_j\to f$ in $L^r(\RR^n)$ for some  $1< r<\infty$  (see Corollary~\ref{coro: emb} in Appendix~\ref{app: embeddings}) imply that $H(f)\to H(f_j)$ in $L^r(\RR^n)$ and, in particular, in $\sz'(\RR^n).$ Let $\sigma_{f_j}$ be as given in Theorem~\ref{thm: appnonlin}; applying Theorem~\ref{th:  main1}  and taking into account \eqref{eq: appnonlin}, we obtain
\begin{align*}
\|H(f_j)\|_{A^{p,q}_\alpha}=\|T_{\sigma_{f_j}}(f_j)\|_{A^{p,q}_\alpha}\lesssim \sup_{|\lambda|\lesssim \|f_j\|_{L^\infty}}\sum_{\ell=0}^{\ceil{\alpha}}|H^{(\ell+1)}(\lambda)|\|f_j\|_{L^\infty}^\ell \|f_j\|_{A^{p,q}_\alpha}.
\end{align*}
(See conditions in Theorem~\ref{th:  main1}, and Appendix~\ref{app: opnorm}.)
Since $\sup_{j\in \NN}\|f_j\|_{L^\infty}\le \|f\|_{L^\infty}$ and $\|f_j\|_{A^{p,q}_\alpha}\to \|f\|_{A^{p,q}_\alpha},$ we conclude that 
\begin{equation*}
\liminf_{j\to\infty} \|H(f_j)\|_{A^{p,q}_\alpha}\lesssim  \sup_{|\lambda|\lesssim\|f\|_{L^\infty}}\sum_{\ell=0}^{\ceil{\alpha}}|H^{(\ell+1)}(\lambda)|\|f\|_{L^\infty}^\ell \,\,\|f\|_{A^{p,q}_\alpha}. 
\end{equation*}

By the Fatou property of $A^{p,q}_\alpha(\LL)$ (see Appendix~\ref{app: embeddings}), the above implies that $H(f)\in A^{p,q}_\alpha(\LL)$ and 
\begin{equation*}
 \|H(f)\|_{A^{p,q}_\alpha}\lesssim  \sup_{|\lambda|\lesssim \|f\|_{L^\infty}}\sum_{\ell=0}^{\ceil{\alpha}}|H^{(\ell+1)}(\lambda)|\|f\|_{L^\infty}^\ell\,\,\|f\|_{A^{p,q}_\alpha}. \qedhere
\end{equation*}

 \end{proof}
 \begin{Remark}If $0<p<\infty,$ $0<q<\infty,$   $\alpha>n_{p,q}-n$ and $f,g\in A^{p,q}_\alpha(\LL)\cap L^\infty(\RR^n),$
 then $fg\in A^{p,q}_\alpha(\LL)\cap L^\infty(\RR^n).$ Indeed, assuming without loss of generality that $f$ and $g$ are real-valued and using Corollary~\ref{coro: appnonlin} with $H(x)=x^2,$  it follows that $f^2,(f+g)^2,g^2\in A^{p,q}_\alpha(\LL)\cap L^\infty(\RR^n);$ since $2fg=(f+g)^2-f^2-g^2,$ we conclude that $fg\in A^{p,q}_\alpha(\LL)\cap L^\infty(\RR^n).$  
\end{Remark}

\appendix
\section{Estimates for $\vp_j(\sqrt{\LL})$}

In this appendix, we state and prove Lemma~\ref{lem: phiest}, which is used in the proofs of Lemmas~\ref{lem: needlets} and \ref{lem: AO}, and Theorem~\ref{thm: appnonlin}.

\begin{Lemma}\label{lem: phiest}
Let $\{\vp_j\}_{j\in \NN_0}$ be an admissible system. If $\eta\ge 1$, $\ve\ge 4,$ $\gamma\in\NN_0^n$ and $K\ge 0,$ it holds that 
\begin{align}\label{phiest A}
|\partial_y^\gamma \vp_j(\sqrt{\LL}) (x,y)| +|\partial_x^\gamma \vp_j(\sqrt{\LL}) (x,y)| \lesssim \f{2^{j(n+|\gamma|)}}{(1+2^j|x-y|)^\eta} \e_{\ve 4^j}(x) e_{\ve 4^j}(y)\quad \forall x,y\in \RR^n, j\in \NN_0,
\end{align}
 and, for $|\gamma|\le K,$
\begin{align}\label{phiest B}
\Big|\int_{\RR^n}(x-y)^\gamma\vp_j(\sqrt{\LL})(x,y)\,dy\Big| \lesssim 2^{-j|\gamma|}\Big(\f{1+|x|}{2^j}\Big)^{K-|\gamma|} \e_{\ve 4^j}(x)\quad \forall x\in\RR^n,j\in \NN_0.
\end{align}
\end{Lemma}
\begin{Remark}
Note that by the symmetry of the kernels $\vp_j(\sqrt{\LL})(x,y)$, \eqref{phiest B} also holds with $dx$ in place of $dy$ on the left hand side, and  $y$ in place of $x$ on the right hand side. 
\end{Remark}
\begin{Remark}\label{re: phiest} Lemma~\ref{lem: phiest} holds true, with the same proof, for a family $\{\varphi_j\}_{j\in\NN_0}$ where $\varphi_j(\lambda)=\varphi(2^{-j}\lambda)$ and $\varphi$ is a smooth function supported in $[0,c]$ for some $c>0$ that satisfies $\varphi^{(k)}(0)=0$ for all $k\in \NN.$
\end{Remark}

\begin{proof}[Proof of Lemma \ref{lem: phiest}]
Regarding \eqref{phiest A}, recall that
\begin{equation*}
\varphi_j(\sqrt{\LL})(x,y)=\sum_{\kk\in \NN_0}\varphi_j(\sqrt{\lambda_\kk})\PP_\kk(x,y).
\end{equation*}
Then \eqref{phiest A} can be proved employing the same ideas in the proof of Theorem~\ref{th: Tsmooth} through the use of Lemmas~\ref{lem: identities} and \ref{lem: leibniz} presented in Appendix~\ref{app:identities}.

We turn to the proof of \eqref{phiest B}. Assume $|\gamma|\le K$ and fix $x\in\RR^n$ and $j\in \NN_0.$ Set $B=B(x,\cro(x))$ where the function $\cro(\cdot)$ is defined in \eqref{eq:rho}. 
Let $\chi$ be a function in $C^\infty(\RR^n)$ supported in $2B$ that satisfies $\chi=1$ on $B,$  $0\le\chi\le 1$ and
\begin{align*}
 \Vert \chi^{(\nu)}\Vert_\infty \le \f{C}{\cro(x)^{|\nu|}}\qquad \forall\nu\in\NN^n_0.
 \end{align*}
 We split the integral into two terms:
 \begin{align*}
\int_{\RR^n}(x-y)^\gamma\vp_j(\sqrt{\LL})(x,y)\,dy
 &= \int_{\RR^n}(1-\chi(y))(x-y)^\gamma\vp_j(\sqrt{\LL})(x,y)\,dy \\
 &\qquad+\int_{\RR^n} \chi(y)(x-y)^\gamma\vp_j(\sqrt{\LL})(x,y)\,dy \\
 &=:I + II.
 \end{align*}
 
To estimate $I$ we use the bounds from \eqref{phiest A} with $\eta>n+K$ and recall that $|\gamma|\le K$  to obtain
 \begin{align*}
 |I| 
  &\lesssim \e_{\ve 4^j(x)} \int_{B^c}  \f{(2^j|x-y|)^{|\gamma|-K}2^{j(n-|\gamma|)}}{(1+2^j|x-y|)^{\eta-K}}\,dy \\
 &\le \Big(\f{1+|x|}{2^j}\Big)^{K-|\gamma|}2^{-j|\gamma|}\e_{\ve 4^j(x)} \int_{\RR^n}\f{2^{jn}}{(1+2^j|x-y|)^{\eta-K}}\,dy \\
 &\lesssim  \Big(\f{1+|x|}{2^j}\Big)^{K-|\gamma|}2^{-j|\gamma|}\e_{\ve 4^j(x)}.
 \end{align*}
 
 For the second term we have, by employing the Cauchy-Schwarz inequality,
\begin{align*}
|II|
&=\Big| \sum_{\kk\in\NN_0}\vp_j(\sqrt{\lambda_\kk}) \sum_{|\xi|=\kk}h_\xi(x)\int_{\RR^n}\chi(y) (y-x)^\gamma  h_\xi(y)\,dy\Big|\\
&\le  \Vert\vp\Vert_\infty \Big(\sum_{\kk\in I_j} \sum_{|\xi|=\kk} h_\xi(x)^2\Big)^{1/2}\Big(\sum_{\kk\in I_j}\sum_{|\xi|=\kk} \Big| \int_{\RR^n} (y-x)^\gamma \chi(y) h_\xi(y)\,dy\Big|^2\Big)^{1/2},
\end{align*}
where we recall that $\kk\in I_j$ means $\frac{1}{2}4^{j-2}-\floor{n/2} \le \kk\le \frac{1}{2}4^j -\ceil{n/2}$.

To estimate the second factor we note that for any $N \in \NN_0,$ it holds that
\begin{align*}
\Big| \int_{\RR^n} (y-x)^\gamma \chi(y) h_\xi(y)\,dy\Big|
&=\lambda^{-N}_{|\xi|} \Big|\int_{\RR^n} \LL_y^N\big [ (y-x)^\gamma \chi(y)\big] h_\xi(y)\,dy\Big| \\
&\le \lambda^{-N}_{|\xi|} \Big\Vert \LL^N \big[(\cdot-x)^\gamma \chi(\cdot)\big] \Big\Vert_{L^2(2B)} \Vert h_\xi\Vert_{L^2(2B)} \\
&\sim (1+|\xi|)^{-N} \Big\Vert \LL^N \big[(\cdot-x)^\gamma \chi(\cdot)\big] \Big\Vert_{L^2(2B)} \Vert h_\xi\Vert_{L^2(2B)}.
\end{align*}
Repeated application of the Leibniz' rule gives, with the sum running over indices such that $|a|+|b|\le 2N,$ $\beta+\nu=b$ and $|\beta| \le |\gamma|,$
\begin{align*}
\LL^N \big[(\cdot-x)^\gamma \chi(\cdot)\big](y)
= \sum_{a,b,\beta,\nu,\gamma} C_{a,b,\beta,\nu} \,y^a (y-x)^{\gamma-\beta}\chi^{(\nu)}(y)
\end{align*}
so that
\begin{align*}
 \Big\Vert \LL^N \big[(\cdot-x)^\gamma \chi(\cdot)\big] \Big\Vert_{L^2(2B)}
 &\sim \sum_{a,b,\beta,\nu,\gamma}  \Big(\int_{2B} \big| |y|^{|a|} |y-x|^{|\gamma|-|\beta|}|\chi^{(\nu)}(y)| \big|^2\,dy\Big)^{1/2} \\
 &\lesssim  \sum_{a,b,\beta,\nu,\gamma}  \cro(x)^{|\gamma|-|\beta|-|\nu| +n/2} \sup_{y\in 2B}|y|^{|a|} \\
 &\lesssim  \sum_{\substack{|a|+|b|\le 2N}} (1+|x|)^{|a|+|b|-|\gamma| -n/2} \\
 &\lesssim (1+|x|)^{2N-|\gamma|-n/2}.
\end{align*}
Inserting this into the estimate for $II$ leads to
\begin{align*}
|II|
&\lesssim \Big(\sum_{\kk\in I_j} \sum_{|\xi|=\kk} h_\xi(x)^2\Big)^{1/2}\Big(\sum_{\kk\in I_j}\sum_{|\xi|=\kk} \bigg|\f{(1+|x|)^{2N-|\gamma|-n/2}}{(1+|\xi|)^N}\bigg|^2 \Vert h_\xi\Vert_{L^2(2B)}^2\Big)^{1/2}  \\
&\lesssim \Big(\f{1+|x|}{2^j}\Big)^{2N-|\gamma|-n/2}2^{-j(|\gamma|+n/2)}\Big(\sum_{\kk\in I_j} \sum_{|\xi|=\kk} h_\xi(x)^2\Big)^{1/2}\Big(\sum_{\kk\in I_j}\sum_{|\xi|=\kk}  \Vert h_\xi\Vert_{L^2(2B)}^2\Big)^{1/2}\\
&\lesssim \Big(\f{1+|x|}{2^j}\Big)^{2N-|\gamma|-n/2}2^{-j(|\gamma|+n/2)}  \Big(\QQ_{4^j}(x,x)\Big)^{1/2} \Big(\int_{2B} \QQ_{4^j}(y,y) \,dy\Big)^{1/2}.
\end{align*}
where in the last line we used that  $\sum_{\kk\le 4^j}\sum_{|\xi|=\kk} h_\xi(y)^2= \QQ_{4^j}(y,y).$
We next apply the bounds \eqref{QQ est} to get
\begin{align*}
|II|
&\lesssim \Big(\f{1+|x|}{2^j}\Big)^{2N-|\gamma|-n/2}2^{-j(|\gamma|+n/2)} \big(2^{jn}\e_{\ve 4^j}(x)^2\big)^{1/2} 2^{jn/2}|2B|^{1/2} \\
&\sim \Big(\f{1+|x|}{2^j}\Big)^{2N-|\gamma|-n}2^{-j|\gamma|}\e_{\ve 4^j}(x)\\
&\le \Big(\f{1+|x|}{2^j}\Big)^{K-|\gamma|}2^{-j|\gamma|}\e_{\ve 4^j}(x)
\end{align*}
by choosing $N$ appropriately depending on whether $\f{1+|x|}{2^j}$ is larger or smaller than 1.  
\end{proof}

\section{Useful identities and estimates}\label{app:identities}
In this appendix, we present identities and estimates used in the proof of Theorem \ref{th: Tsmooth}. 
\begin{Lemma}\label{lem: identities}
\begin{enumerate}[\upshape(a)]
\item Suppose that
$$ \mathbb{F}(x,y) = \sum_{\kk\in\NN_0} f(x,y,\kk) \,\PP_\kk(x,y).$$
If $N\in\ZZ_+,$ it holds that
\begin{align}\label{eq:identity A}
	2^N(x_i- y_i)^N \mathbb{F}(x,y) = \sum_{\f{N}{2}\le \ell\le N}c_{\ell, N}\sum_{\kk\in \NN_0} \diff_\kk^\ell f(x,y,\kk)\big(\A{y}_i-\A{x}_i\big)^{2\ell-N}\PP_\kk(x,y),
\end{align}
where $c_{\ell,N} = (-4)^{N-\ell}(2N-2\ell-1)!!\binom{N}{2\ell-N}$.

\item If $N, M\in\ZZ_+,$ it holds that
\begin{align}
	x_i^M \big(\A{x}_i-\A{y}_i\big)^N &= \sum_{k=0}^M\tbinom{M}{k} \tfrac{N!}{(N-k)!}\big(\A{x}_i-\A{y}_i\big)^{N-k}x_i^{M-k} \label{eq:identity b1}
	\end{align}
	and
	\begin{align}
	(x_i-y_i)^N\big(\A{x}_i\big)^M &=\sum_{k=0}^M\tbinom{M}{k}  \tfrac{N!}{(N-k)!}\big(\A{x}_i\big)^{M-k}(x_i-y_i)^{N-k},\label{eq:identity b2}
\end{align}
where $\f{N!}{(N-k)!}$ is defined to be 0 whenever $N<k$.

\item If $\beta\in\NN_0^n$ and $\kk\in\NN_0$, it holds that
\begin{align}\label{eq:identity C}
x^\beta \,\PP_\kk(x,y)= \sum_{\omega\le \beta} \sum_{|\xi|=\kk}b_{\omega,\beta}(\xi) h_{\xi+\beta-2\omega}(x) h_\xi(y),\end{align}
where $b_{\omega,\beta}(\xi)=\prod_{i=1}^n b_{\omega_i,\beta_i}(\xi_i)$ with $ b_{\omega_i,\beta_i}(\xi_i)=0$ if $\xi_i+\beta_i-2\omega_i<0$ and  $ b_{\omega_i,\beta_i}(\xi_i)\sim \xi_i^{\beta_i/2}$ otherwise. 

\item If $\xi,\alpha\in\NN_0^n,$  $m\in \NN_0$ and $i\in\{1,\dots, n\},$ it holds that 
\begin{align} \label{eq:identity d1}
\big|\big(\A{x}_i\big)^m h_\xi(x)\big|& \le \big[2(\xi_i+m)+2\big]^{\f{m}{2}} | h_{\xi+m e_i}(x)| 
\end{align}
and 
\begin{align} \label{eq:identity d2}
\big|\big(\A{x}\big)^\alpha h_\xi(x)\big| &\le \big[2(|\xi|+|\alpha|)+2\big]^{\f{|\alpha|}{2}} | h_{\xi+\alpha}(x)|.
\end{align}
\end{enumerate}
\end{Lemma} 

\begin{proof}[Proof of Lemma \ref{lem: identities}]
The identity in part (a) can be found in \cite[Lemma 8]{MR2399106} and \cite[p.72]{MR1215939}  with $k$ as a function of $\kk$ only. However, it can be checked that the proof also works  when $k$ depends on both $x$ and $y$. 
Part (b) is from   \cite[Lemma 9]{MR2399106}.  Part (c) is  \cite[equation (6.14)]{MR2399106} with $\mu=0$. In part (d), estimate \eqref{eq:identity d1} follows from \cite[equation (6.5)]{MR2399106}:
$$\big(\A{x}_i\big)^m h_\xi(x) = \prod_{r=0}^{m-1}\sqrt{2(\xi_i+r)+2}\,\, h_{\xi+me_i}(x).$$
The inequality  \eqref{eq:identity d2} follows from applying \eqref{eq:identity d1} repeatedly. 
\end{proof}

\begin{Lemma}\label{lem: leibniz}
\begin{enumerate}[\upshape(a)]
\item If $\ell \in\NN_0,$ it holds that
\begin{align}\label{eq:leibniz1}
 \diff_\kk^\ell \big(f(\kk)\,g(\kk)\big) = \sum_{r=0}^\ell \binom{\ell}{r}\diff_\kk^r f(\kk)\,\diff_\kk^{\ell-r}g(\kk+r).
\end{align}
\item If $\alpha\in\NN_0^n$, it holds that
\begin{align}\label{eq:leibniz2}
A^\alpha (fg) =\sum_{\nu\le\alpha} \binom{\alpha}{\nu} (-1)^\nu\partial^\nu f \, A^{\alpha-\nu} g.
\end{align}
\end{enumerate}
\end{Lemma}
\begin{proof}[Proof of Lemma \ref{lem: leibniz}]
Part (a) is well known. For part (b), first note that the following representation for Hermite derivatives holds:
\begin{align}\label{eq:Arep}
A_i^m = (-1)^me^{x_i^2/2} \partial_i^m e^{-x_i^2/2}\qquad \forall m\in\NN_0.
\end{align}
This identity can be obtained  by direct calculation for $m=1$ and by  induction for all $m.$
We next show that \eqref{eq:Arep} gives
\begin{align}\label{eq:leibniz3}
 A_i^{\alpha_i} (fg) = \sum_{\nu_i=0}^{\alpha_i} \binom{\alpha_i}{\nu_i}(-1)^{\nu_i} \partial_i^{\nu_i} f \, A_i^{\alpha_i-\nu_i} g.
\end{align}
Indeed, by \eqref{eq:Arep} and the Leibniz rule for differentiation we obtain
\begin{align*}
A_i^{\alpha_i} (fg) 
&= (-1)^{\alpha_i}e^{x_i^2/2} \partial_i^{\alpha_i} \big( e^{-x_i^2/2}fg\big) \\
&=  (-1)^{\alpha_i}e^{x_i^2/2} \sum_{\nu_i=0}^{\alpha_i} \binom{\alpha_i}{\nu_i} \partial_i^{\nu_i} f \, \cdot \partial_i^{\alpha_i-\nu_i}\big(e^{-x_i^2/2} g\big)\\
&=\sum_{\nu_i=0}^{\alpha_i} \binom{\alpha_i}{\nu_i} (-1)^{\nu_i}\partial_i^{\nu_i} f \, \cdot(-1)^{\alpha_i-\nu_i}e^{x_i^2/2}\partial_i^{\alpha_i-\nu_i}\big(e^{-x_i^2/2} g\big).
\end{align*}
Equality \eqref{eq:leibniz3} follows by applying \eqref{eq:Arep} again. The identity  \eqref{eq:leibniz2} then follows by applying \eqref{eq:leibniz3} to each component $1\le i\le n$. 
\end{proof}

   \section{Remarks about Hermite Besov and Triebel--Lizorkin spaces}\label{app: embeddings}
   
 In this appendix, we present some embeddings of Hermite Besov and Hermite Triebel--Lizorkin spaces. The embeddings stated in Corollary~\ref{coro: emb}, a consequence of Theorem~\ref{thm: emb}, are used in the proof of Corollary~\ref{coro: appnonlin}. In addition, we comment on the Fatou property of Hermite Besov and Hermite Triebel--Lizorkin spaces, which is also used in the proof of Corollary~\ref{coro: appnonlin}.

 \begin{Theorem} \label{thm: emb}
 \begin{enumerate}[\upshape(a)]
\item\label{item: emb1}  If  $\alpha\in \RR, \varepsilon>0,$ $0<q\le \infty,$ $0<q_1\le \infty,$  and $0<p\le \infty$ for Besov spaces or $0<p<\infty$ for Triebel--Lizorkin spaces, it holds that 
 \begin{equation*}
 A^{p,q}_{\alpha+\varepsilon}(\LL)\hookrightarrow A^{p,q_1}_{\alpha}(\LL).
 \end{equation*}
 \item\label{item: emb2} If $0<q\le \infty,$ $0<p<\infty$ and $\alpha\in \RR,$  it holds that 
  \begin{equation*}
 B^{p,\min(p,q)}_\alpha(\LL)\hookrightarrow  F^{p,q}_\alpha(\LL)\hookrightarrow   B^{p,\max(p,q)}_\alpha(\LL).
 \end{equation*}
 \item \label{item: emb3} If $0<q\le \infty,$ $0<p<p_1<\infty$ and $\alpha,\alpha_1\in \RR$ are such that $\alpha_1<\alpha,$  it holds that
 \begin{equation*}
 A^{p,q}_{\alpha}(\LL)\hookrightarrow A^{p_1,q}_{\alpha_1}(\LL)\quad \text{if}\quad \alpha-\frac{n}{p}=\alpha_1-\frac{n}{p_1}.
 \end{equation*}
 \end{enumerate}
 \end{Theorem}

The proofs of the embeddings stated in Theorem~\ref{thm: emb} are the same as those in the Euclidean setting; see  \cite[p.47, Proposition 2]{MR3024598} for \eqref{item: emb1} and \eqref{item: emb2} and  \cite[Propositions 6 and 7]{MR2399106} for \eqref{item: emb3}.
 
 \begin{Corollary}\label{coro: emb}  Let   $0<q\le \infty.$  If $1<p<\infty$ and $ \varepsilon>0,$ then  $A^{p,q}_\varepsilon (\LL)\hookrightarrow L^p(\RR^n);$ if $0<p\le 1$ and $\varepsilon>n(\frac{1}{p}-1),$ there exists $p_1> 1$ such that $A^{p,q}_\varepsilon (\LL)\hookrightarrow L^{p_1}(\RR^n).$
 \end{Corollary}
 \begin{proof} Let $0<q\le \infty.$ 
 
 \medskip
 
 \underline{Case $1<p<\infty$ and  $\varepsilon>0:$} 
 Taking $\alpha=0$ in  part \eqref{item: emb1} of Theorem~\ref{thm: emb} we obtain that $F^{p,q}_{\varepsilon}(\LL)\hookrightarrow F^{p,2}_{0}(\LL)=L^p(\RR^n).$  Using parts \eqref{item: emb1}  and \eqref{item: emb2} of Theorem~\ref{thm: emb} we have $B^{p,q}_{\alpha+\varepsilon}(\LL)\hookrightarrow B^{p,\min(p,2)}_{\alpha}(\LL)\hookrightarrow  F^{p,2}_\alpha(\LL);$  then  $\alpha=0$ implies $B^{p,q}_{\varepsilon}(\LL)\hookrightarrow  F^{p,2}_0(\LL)=L^p(\RR^n).$

\medskip

 \underline{Case $0<p\le 1$ and $\varepsilon>n(\frac{1}{p}-1):$}  Let $p_1> 1$ be such that $p_1>p$ and $\varepsilon>n(\frac{1}{p}-\frac{1}{p_1});$ such $p_1$ exists since $\varepsilon>n(\frac{1}{p}-1)\ge 0.$ Setting $\alpha_1=\varepsilon-n(\frac{1}{p}-\frac{1}{p_1}),$ part~\eqref{item: emb3} of Theorem~\ref{thm: emb} and the previous case imply that 
 \begin{equation*}
 A^{p,q}_{\varepsilon}(\LL)\hookrightarrow A^{p_1,q}_{\alpha_1}(\LL)\hookrightarrow L^{p_1}(\RR^n).\qedhere
 \end{equation*}
     \end{proof}
  
  \medskip
  
  Next, we comment about the Fatou property for Hermite Besov and Hermite Triebel--Lizorkin spaces.
  
    Let  $\mathcal{A}$ be a quasi-Banach space such that $\sz(\RR^n)\hookrightarrow\mathcal{A}\hookrightarrow\sz'(\RR^n).$ The space $\mathcal{A}$ is said to  have the Fatou property if  for every sequence $\{f_j\}_{j\in \NN}\subset \mathcal{A}$ that converges in $\sz'(\RR^n),$ as $j\to \infty,$ and that satisfies $\liminf_{j\to\infty} \|f_j\|_{\mathcal{A}}<\infty,$ it follows that $\lim_{j\to\infty}f_j\in \mathcal{A}$ and $\|\lim_{j\to\infty}f_j\|_{\mathcal{A}}\lesssim \liminf_{j\to\infty} \|f_j\|_{\mathcal{A}},$ where the implicit constant is independent of $\{f_j\}_{j\in\NN}.$ 

It can be shown, using standard proofs (see for instance \cite[p.48, Proposition 2.8]{MR2683024}), that $A^{p,q}_{\alpha}(\LL)$  posses the Fatou property for any $\alpha\in \RR,$  $0<q\le \infty,$ 
$0<p\le \infty$ for Besov spaces and $0<p<\infty$ for Triebel--Lizorkin spaces This is due to the following facts: (1) if $f,g\in L^p(\RR^n)$ and $|f|\le|g|$ pointwise a.e., then $\|f\|_{L^p}\le \|g\|_{L^p};$ (2) if $\{f_j\}_{j\in \NN}\subset L^p(\RR^n)$ and $f_j\ge 0$ poinwise a.e., then $\|\liminf_{j\to\infty} f_j\|_{L^p}\le \liminf_{j\to\infty}\|f_j\|_{L^p};$ (3) if $f_j\to f$ in $\sz'(\RR^n)$ then,   for any $k\in \NN_0$ and any admissible pair $(\varphi_0,\varphi_j),$ $\varphi_k(\sqrt{\LL})f_j\to\varphi_k(\sqrt{\LL})f$ pointwise as $j\to \infty.$

\section{Operator norm}\label{app: opnorm}

The following result about the operator norm of pseudo-multipliers is used in the proof of Corollary~\ref{coro: appnonlin}.

\begin{Lemma} Let $m\in \RR,$ $\rr\ge 0,$ $\dd\ge 0,$   $\N,\K\in \NN_0,$ $\alpha, \tilde{\alpha}\in \RR,$ $0<p< \infty,$ $0<q< \infty,$ and $0<\tilde{p}\le\infty$ for Besov spaces or $0<\tilde{p}<\infty$ for Triebel--Lizorkin spaces.  If $T_\sigma$ is bounded from $A^{p,q}_\alpha(\LL)$ to $A^{\tilde{p},\tilde{q}}_{\tilde{\alpha}}(\LL)$  for all  $\sigma\in \SC^{m,\K,\N}_{\rr,\dd},$ it holds that 
\begin{equation*}
\|T_\sigma\|_{A^{p,q}_\alpha\to A^{\tilde{p},\tilde{q}}_{\tilde{\alpha}}}\lesssim \mathop{\max_{0\le |\nu|\le \N}}_{0\le \kappa\le \K} \mathop{\sup_{x\in \RR^n}}_{\xs\in \NN_0}| \partial^\nu_x \diff^\kappa_\xs \sigma(x,\xs)|  (1+\sqrt{\xs})^{-m+2\rr\kappa-\dd|\nu|} \quad \forall \sigma\in \SC^{m,\K,\N}_{\rr,\dd}.
\end{equation*} 
\end{Lemma}
\begin{proof} We follow ideas from the proof of \cite[Lemma 2.6]{MR3205530}. Set
\begin{equation*}
\|\sigma\|_{\SC^{m,\K,\N}_{\rr,\dd}}= \mathop{\max_{0\le |\nu|\le \N}}_{0\le \kappa\le \K} \mathop{\sup_{x\in \RR^n}}_{\xs\in \NN_0}| \partial^\nu_x \diff^\kappa_\xs \sigma(x,\xs)|  (1+\sqrt{\xs})^{-m+2\rr\kappa-\dd|\nu|} \quad \forall \sigma\in \SC^{m,\K,\N}_{\rr,\dd};
\end{equation*}
then $\SC^{m,\K,\N}_{\rr,\dd}$ is a Banach space with the norm $\|\cdot\|_{\SC^{m,\K,\N}_{\rr,\dd}},$  
Define the linear operator
\begin{equation*}
\mathcal{U}: \SC^{m,\K,\N}_{\rr,\dd}\to L(A^{p,q}_\alpha(\LL), A^{\tilde{p},\tilde{q}}_{\tilde{\alpha}}(\LL)), \quad \mathcal{U}(\sigma)=T_\sigma,
\end{equation*}
where $L(A^{p,q}_\alpha(\LL), A^{\tilde{p},\tilde{q}}_{\tilde{\alpha}}(\LL))$ is the quasi-Banach space of all linear bounded operators from $A^{p,q}_\alpha(\LL)$ to $A^{\tilde{p},\tilde{q}}_{\tilde{\alpha}}(\LL)$ with the usual operator norm. We will show that the graph of $\mathcal{U}$ is closed; as a consequence of the Closed Graph Theorem, it follows that $\mathcal{U}$ is continuous and therefore the desired result follows.

Let $\{(\sigma_j,T_{\sigma_j})\}_{j\in \NN}$ be a sequence in the graph of $\mathcal{U}$ that converges to $(\sigma, T)\in  \SC^{m,\K,\N}_{\rr,\dd}\times L(A^{p,q}_\alpha(\LL), A^{\tilde{p},\tilde{q}}_{\tilde{\alpha}}(\LL))$ in the product topology.  We will show that $T(f)=T_\sigma(f)$ for all $f\in \sz(\RR^n);$ assuming the latter, since $\sz(\RR^n)$ is dense in $A^{p,q}_\alpha(\LL)$ and $T_\sigma, T\in L(A^{p,q}_\alpha(\LL), A^{\tilde{p},\tilde{q}}_{\tilde{\alpha}}(\LL)),$ it follows that $T_\sigma=T.$ As a consequence, the graph of $\mathcal{U} $ is closed. 

Given $f\in \sz(\RR^n)$ and $N$ sufficiently large, using  the definition of  $\|\cdot\|_{\SC^{m,\K,\N}_{\rr,\dd}},$ \cite[Lemma 3]{MR2399106} and that $h_\xi$ are bounded uniformly in $\xi$ by \cite[Lemma 1.5.2, p.27]{MR1215939}, we obtain
\begin{align*}
|T_\sigma(f)(x)-T_{\sigma_j}(f)(x)|&=\left|\sum_{k\in \NN_0} (\sigma(x,\lambda_k)-\sigma_j(x,\lambda_k))\PP_k(f)(x)\right|\\
&\le \|\sigma-\sigma_j\|_{\SC^{m,\K,\N}_{\rr,\dd}} \sum_{k\in \NN_0} (1+\sqrt{\lambda_k})^m \sum_{|\xi|=k} |\langle f,h_\xi\rangle| |h_\xi(x)|\\
&\lesssim \|\sigma-\sigma_j\|_{\SC^{m,\K,\N}_{\rr,\dd}} \sum_{\xi\in \NN_0^n} (1+\sqrt{|\xi|})^m \frac{1}{(1+|\xi|)^N}\\
&\lesssim   \|\sigma-\sigma_j\|_{\SC^{m,\K,\N}_{\rr,\dd}},
\end{align*} 
which implies that $T_{\sigma_j}(f)$  converges to $T_{\sigma}(f)$ uniformly in $\RR^n.$ On the other hand, we have
\begin{align*}
\|T_{\sigma_j}(f)-T(f)\|_{A^{\tilde{p},\tilde{q}}_\alpha}\lesssim \|T_{\sigma_j}-T\|_{A^{p,q}_\alpha\to A^{\tilde{p},\tilde{q}}_{\tilde{\alpha}}} \|f\|_{A^{p,q}_\alpha}\to 0.
\end{align*}

The above implies that $T_{\sigma_j}(f)$  converges to $T_{\sigma}(f)$ in $\sz'(\RR^n)$ and  $T_{\sigma_j}(f)$  converges to $T(f)$ in $\sz'(\RR^n)$ for all $f\in \sz(\RR^n)$ (for the latter see  \cite[Proposition 4, p.385 and Section 5, p.392]{MR2399106}, which state that $A^{\tilde{p},\tilde{q}}_\alpha(\LL)\hookrightarrow \sz'(\RR^n)$). Therefore $T_{\sigma}(f)=T(f)$ for all $f\in \sz(\RR^n),$ as desired.
\end{proof}

\end{document}